\documentclass[article,12pt]{article}

\let\OLDthebibliography\thebibliography
\renewcommand\thebibliography[1]{
  \OLDthebibliography{#1}
  \setlength{\parskip}{0pt}
  \setlength{\itemsep}{0pt plus 0.3ex}
}

\usepackage{amsmath}
\usepackage{amssymb}
\usepackage{amsthm}

\usepackage{a4wide}

\usepackage{hyperref}

\usepackage[T1]{fontenc}

\theoremstyle{plain}
\newtheorem{theorem}{Theorem}
\numberwithin{theorem}{section}
\newtheorem{lemma}[theorem]{Lemma}
\newtheorem{proposition}[theorem]{Proposition}
\newtheorem{corollary}[theorem]{Corollary}

\theoremstyle{remark}
\newtheorem*{remark*}{Remark}
\newtheorem{remark}[theorem]{Remark}

\theoremstyle{definition}
\newtheorem{definition}[theorem]{Definition}
\newtheorem{example}[theorem]{Example}

\DeclareMathOperator{\Tw}{Tw}
\DeclareMathOperator{\TwOp}{TwOp}

\usepackage{tikz}
\usetikzlibrary{shapes,fit,positioning,calc,matrix}
\tikzset{
  optree/.style={scale=.5,thick,grow'=up,level distance=10mm,inner sep=1pt},
  comp/.style={draw=none,circle,fill,line width=0,inner sep=0pt},
  dot/.style={draw,circle,fill,inner sep=0pt,minimum width=3pt},
  circ/.style={draw,circle,inner sep=1pt,minimum width=4mm},
  emptycirc/.style={draw,circle,inner sep=1pt,minimum width=2mm},
  root/.style={level distance=10mm,inner sep=1pt},
  leaf/.style={draw=none,circle,fill,line width=0,inner sep=0pt},
  nodot/.style={draw,circle,inner sep=1pt},
}

\usepackage{tikz-cd}

\begin{document}

\title{Twisted arrow categories, operads and Segal conditions}
\author{Sergei Burkin\thanks{Graduate School of Mathematical Sciences, University of Tokyo. Email: sburkin@protonmail.com}}
\maketitle

\begin{abstract}
We introduce twisted arrow categories of operads and of algebras over operads. Up to equivalence of categories, the simplex category~$\Delta$, Segal's category~$\Gamma$, Connes cyclic category $\Lambda$, Moerdijk--Weiss dendroidal category~$\Omega$, and categories similar to graphical categories of Hackney--Robertson--Yau are twisted arrow categories of symmetric or cyclic operads. Twisted arrow categories of operads admit Segal presheaves and 2-Segal presheaves, or decomposition spaces. Twisted arrow category of an operad $P$ is the $(\infty, 1)$-localization of the corresponding category $\Omega/P$ by the boundary preserving morphisms. Under mild assumptions, twisted arrow categories of operads, and closely related universal enveloping categories, are generalized Reedy. We also introduce twisted arrow operads, which are related to Baez--Dolan plus construction.
\end{abstract}

\tableofcontents

\section*{Introduction}
\addcontentsline{toc}{section}{Introduction}

The twisted arrow category $\Tw(\mathcal{C})$ of a category $\mathcal{C}$, introduced by Quillen, is defined uniquely by the following sequence of categories.
\[\Delta/\mathcal{C}\rightarrow \Tw(\mathcal{C})\rightarrow \mathcal{C}^{op}\times \mathcal{C}\rightarrow\mathcal{C}^{op} \label{seq:delta0}\tag{$\Delta$}\]
We introduce twisted arrow categories of operads. The twisted arrow category $\Tw(P)$ of an operad $P$ is defined uniquely by the lower row of the following diagram. 
\begin{center}
\begin{tikzcd}
& \mathcal{C}_P^{op} \arrow[r] \arrow[d, hook] & PROP(P)^{op} \arrow[d, hook]   &   \\
\Omega/P\arrow[r] & \Tw(P) \arrow[r] & \mathcal{U}(P) \arrow[r] & cat(P)^{op}
\end{tikzcd}
\end{center}
This lower row is the only natural generalization of the sequence \eqref{seq:delta0} to operads. The categories in the diagram are: the category $\Omega/P$ of elements of the dendroidal nerve of the operad $P$ (\cite{moerdijk2007dendroidal}); the category $\mathcal{C}_P$ introduced in \cite{deBrito2018catp} and related to operadic categories (\cite{Batanin2015operadic}) and to operator categories (\cite{Barwick2018operator}); the PROP corresponding to $P$; the universal enveloping category $\mathcal{U}(P)$ (\cite{ginzburg1994koszul}); and the category of operators $cat(P)$ (\cite{May1978uniqueness}), on which the approach of Lurie to $\infty$-operads is based. The categories $\Tw(P)$ and $\mathcal{U}(P)$ have wide subcategories $Upper$ and $Lower$ that form the strict factorization system $(Upper,Lower)$. The $(Upper, Lower)$ strict factorization system generates $(Active, Inert)$ orthogonal factorization system. The categories $\mathcal{C}_P^{op}$ and $PROP(P)^{op}$ in the diagram are the subcategories $Upper$ of $\Tw(P)$ and of $\mathcal{U}(P)$ respectively. 
 
The~simplex category $\Delta$, Segal's category $\Gamma$ and Moerdijk--Weiss category $\Omega$ are equivalent to the twisted arrow categories of the operad $uAs$ of monoids, the operad $uCom$ of commutative monoids, and the operad $sOp$ of single-coloured symmetric operads (\cite{berger2007resolution}) respectively. 

\medskip

We introduce Segal presheaves over twisted arrow categories of operads. There are two possible definitions of Segal presheaves: one comes from the work of Chu and Haugseng (\cite{chu2019homotopy}), while another is a sheaf condition on restriction to the subcategory $Lower$ or $Inert$. Under mild conditions on the operad the two definitions are equivalent.

If an operad $P$ is sufficiently nice, Segal presheaves over $\Tw(P)$ can be seen as ``multi-object algebras over $P$'', partial $P$-algebras whose elements have ``objects'', such that composability of elements $a_i$ of an algebra $A$ via an operation $p$ in $P$ and the objects of this composition are determined by the objects of the elements $a_i$. These objects will be called \emph{petals}. For any operad $P$ the category of single-object (i.e.\@ single-petal) Segal presheaves over $\Tw(P)$ is equivalent to the category of algebras over $P$ (\textbf{Theorem~\ref{thm:algebras-as-segal}}). 

We also introduce 2-Segal presheaves over twisted arrow categories of operads. This notion generalizes 2-Segal sets, or discrete decomposition spaces  (\cite{decompositionspaces1,decompositionspaces2,decompositionspaces3,Dyckerhoff2019higher}). For any operad $P$ the category of 2-Segal presheaves over $\Tw(P)$ is equivalent to the category of special morphisms of operads into $P$, called decomposition morphisms.

A \emph{palatable} operad $P$ is an operad such that Segal presheaves over $\Tw(P)$ are 2-Segal.  For a palatable operad $P$ a Segal presheaf $X$ over $\Tw(P)$ corresponds to the algebra $X'$ over the operad $P_{Pl(X)}$ constructed from the operad $P$ and the petals of $X$. The operads $uAs$ and $sOp$ are palatable. For any set of colours $C$ and any $C$-coloured operad $Q$ seen as a Segal presheaf $X$ over the category $\Omega\simeq\Tw(sOp)$ the operad $sOp_{Pl(X)}$ coincides with the operad $sOp_C$ whose algebras are $C$-coloured operads.

\medskip

We further generalize the functor $\Tw$ to algebras over operads. Twisted arrow categories of categories, of operads, of Segal presheaves over twisted arrow categories of palatable operads, and of algebras over operads  belong respectively to the sequences \eqref{seq:delta}, \eqref{seq:omega}, \eqref{seq:twx} and \eqref{seq:twa}. 
\begin{alignat*}{4}
  &\Omega/(uAs, \mathcal{C})\to \Delta/\mathcal{C}\simeq &&\Tw(uAs)/\mathcal{C}\to &&\Tw(\mathcal{C}) \longrightarrow  &&\mathcal{U}(\mathcal{C})=\mathcal{C}^{op} \times \mathcal{C} \label{seq:delta} \tag{Cat}\\
 &\Omega/(sOp, P)\to \Omega/P\simeq &&\Tw(sOp)/P \to &&\Tw(P) \longrightarrow  &&\mathcal{U}(P) \label{seq:omega} \tag{Op}\\
  &\Omega/(Q, X)\xrightarrow{\hspace{1.8cm}} &&\Tw(Q)/X \longrightarrow  &&\Tw_Q(X)\to \, &&\mathcal{U}_Q(X) \label{seq:twx} \tag{SPsh}\\
 &\Omega/(Q, A)\xrightarrow{\hspace{1.8cm}} &&\Tw(Q)/A \longrightarrow  &&\Tw_Q(A)\to  &&\mathcal{U}_Q(A) \label{seq:twa} \tag{Alg}
\end{alignat*}

The sequences increase in generality. The categories on the right are the enveloping categories (\cite{ginzburg1994koszul}). The categories on the left are the categories of elements of dendroidal nerves of operadic algebras. In the second column are the categories of elements of Segal presheaves. The functors on the right are discrete opfibrations, with fibers corresponding to morphisms, operations, and elements of algebras respectively. Up to equivalence of categories, the functors on the left and in the middle are the localizations by the \emph{active} morphisms. The active morphisms in $\Delta$ and $\Omega$ are the endpoint preserving and the boundary preserving morphisms respectively. The opfibration property and the localization property define twisted arrow categories uniquely up to isomorphism or up to equivalence.

The construction $\Tw_Q$, the most general construction of this work, has another meaningful definition. Recall that for any operad $Q$ and $Q$-algebra $A$ the Baez--Dolan plus construction $A^+\equiv(Q,A)^+$ of $A$ (\cite{Baez1998higher3}) is the operad whose algebras are $Q$-algebras endowed with a $Q$-algebra map to $A$, while the enveloping operad $\mathcal{U}Op(A)$ of $A$ (\cite{getzler1994operads,fresse1998lie}) is the operad whose algebras are $Q$-algebras endowed with a $Q$-algebra map from $A$. The enveloping category $\mathcal{U}(A)$ of $A$ is the underlying category of the operad $\mathcal{U}Op_Q(A)$. Similarly, the twisted arrow category $\Tw_Q(A)$ of $A$ is the underlying category of the operad $\TwOp_Q(A)$ that we call \emph{the twisted arrow operad} of $A$. Algebras over the operad $\TwOp_Q(A)$ are equivalently $Q$-algebras endowed with a $Q$-algebra map from $A$ and with a $Q$-algebra map to $A$ such that the composition of these maps is the identity map of $A$. The twisted arrow operad construction can be seen as a composition of these two constructions in two ways: there is an isomorphism $\TwOp_Q(A)\cong\mathcal{U}Op_{(Q,A)^+}(id_A)\cong (\mathcal{U}Op_Q(A),id_A)^+$ and there is a sequence $(Q,A)^+\to\TwOp_Q(A)\to\mathcal{U}Op_Q(A)$ that determines $\TwOp_Q(A)$.

\medskip 

We attempt to understand why categories of elements and twisted arrow categories of Segal presheaves appear in homotopy theory. Twisted arrow categories of categories appear in category\footnote{Twisted arrow categories of categories implicitly appear in (co)end calculus via the functor $\Tw(\mathcal{C})\to\mathcal{C}^{op}\times\mathcal{C}$. We do not know if there is an analogous theory for operads.} and homotopy theory as the middle term of the subsequences $\Delta/\mathcal{C}\to \Tw(\mathcal{C})\to\mathcal{C}^{op}\times\mathcal{C}$ of the sequences \eqref{seq:delta} and \eqref{seq:delta0}. For example, Thomason, Baues--Wirsching and Hochschild--Mitchell cohomology theories are based respectively on the categories $\Delta/\mathcal{C}$,  $\Tw(\mathcal{C})$ and $\mathcal{C}^{op}\times\mathcal{C}$  (\cite{baues1985cohomology,Wells2001ExtensionTF,galvez2013thomason}). For operads the situation should be similar, while the case of algebras is more delicate.  

Twisted arrow categories of algebras have 2-categorical nature. In general for a $Q$-algebra $A$ there is a strict 2-category $\mathbb{T}_Q(A)$ similar to bicategories of correspondences and such that, after inversion of 2-morphisms, $\mathbb{T}_Q(A)$ is the $(\infty, 1)$-localization of the category $\Tw(Q)/A$ by the active morphisms. The homotopy category of this $(\infty, 1)$-localization is $\Tw_Q(A)$ (\textbf{Theorem~\ref{thm:infty-localization}}).

An operad $Q$ is \emph{canonically decomposable} if for any (equivalently, for the terminal) $Q$-algebra $A$ connected components of Hom-categories in $\mathbb{T}_Q(A)$ have initial objects. If an operad $Q$ is canonically decomposable, then for any $Q$-algebra $A$ the category $\Tw_Q(A)$ admits a simple description, and the functor $\Tw(Q)/A\to\Tw_Q(A)$ is the $(\infty, 1)$-localization by the active morphisms. The operads $uAs$ and $sOp$ are canonically decomposable, and the left functors in the sequences \eqref{seq:delta}, \eqref{seq:omega}, \eqref{seq:twa} and \eqref{seq:twx} and the middle functors in the sequences \eqref{seq:delta} and \eqref{seq:omega} are $(\infty, 1)$-localizations. As a corollary, we get \textit{Theorem 3.0.1}  of \cite{walde20172}, which implies \textit{Theorem 1.1} of \cite{deBrito20dendroidal} and the asphericity of $\Omega$, proved in \cite{ara2019dendroidal}.

There is another connection with homotopy theory. For any operad $P$ the twisted arrow category $\Tw(P)$ is endowed with orthogonal factorization system $(R_-,R_+)$. The factorization system $(R_-,R_+)$, with the degree map equal to arity, is a generalized Reedy structure if and only if the underlying category of $P$ is a groupoid. There is a more general statement for graded operads (\textbf{Theorem~\ref{thm:groupoid-reedy}}). We also consider the analogous case of the image of $\Tw(P)$ in $\mathcal{U}(P)$ (\textbf{Theorem~\ref{thm:enveloping-image-reedy}}), and the cases where the generalized Reedy structure behaves particularly well.

\paragraph{Related works.} 
Twisted arrow categories of simplicial operads were introduced independently by Truong Hoang (\cite{hoang2020quillen}) with the aim of generalizing Quillen cohomology and cotangent complex from $\infty$-categories to $\infty$-operads. We give possible dendroidal analogue: the construction of a simplicial set from a dendroidal set. This construction has 2-categorical nature and sends the nerve of an operad to the nerve of its twisted arrow category. Otherwise we consider only the discrete case.

Graphical categories of Hackney--Robertson--Yau are similar to twisted arrow categories of the operads whose algebras are generalized operads (\cite{hackney2015infinity,hackney2017factorizations,hackney2019higher,hackney2019graphical,hackney2019modular,raynor2018compact}). However, the categories $U$ and $\widetilde{U}$ of \cite{hackney2019graphical} are not twisted arrow categories of operads. The closely related graded operad $mOp$ is not palatable, and its twisted arrow category is \emph{non-dualizable} generalized Reedy. We show that this problem, and essentially the same problem considered in \cite{raynor2018compact,hackney2019graphical,hackney2019modular}, is caused by the suboperad $ciuAs$ of $mOp$, the operad of monoids with anti-involution and with compatible bilinear form. We show that this problem does not arise for similar operads $iuAs_{Tr}$ and $iuAs_{iTr}$: these operads are palatable and the corresponding twisted arrow categories are dualizable generalized Reedy. 

The starting point of the present work was the work of Dehling and Vallette (\cite{dehling2015symmetric}) and the observation that, being central to homotopy theory of operads, the category $\Omega$ should arise from some construction applied to the operad $sOp$. 

The functor $\Tw(P)\to\mathcal{U}(P)$ has appeared implicitly in the definition of Hochschild and cyclic homology as functor homology (\cite{pirashvili25hochschild}). The connection with functor homology was also studied by Benoit Fresse: the enveloping category $\mathcal{U}(P)$ of an operad $P$ is the opposite of the category $\Gamma^+_P$ introduced in  \cite{fresse2014functor}. 

Twisted arrow categories and enveloping categories of operads generate \emph{algebraic patterns} of Chu and Haugseng (\cite{chu2019homotopy}). The approach of Chu and Haugseng is closely related to existing nerve theorems (\cite{leinster2004,weber2007familial,berger2012monads}). This should imply that any operad generates a category via the nerve theorems. However, for a general operad $P$, including operads $mOp$ and $mOp_{(g,n)}$, representable presheaves over $\Tw(P)$ are not Segal, thus the corresponding category obtained from $P$ via the nerve theorems in general should be different from $\Tw(P)$. 

\paragraph{Structure of the paper.} In Section~\ref{sec:preliminaries} we recall the operads whose algebras are generalized operads and describe operads whose algebras are cyclic algebras over cyclic operads. In Section~\ref{sec:main-def} we define twisted arrow categories and operads, describe the canonical decomposability property that allows to compute twisted arrow categories in practice, provide examples of twisted arrow categories of operads, including examples related to functor homology and to categories of cobordisms, describe the structure on twisted arrow categories of operads that is used in the rest of the work, and give a criterion for twisted arrow categories of operads to be generalized Reedy. In Section~\ref{sec:segal} we describe the equivalence between algebras over an operad and single-object Segal presheaves over the corresponding twisted arrow category and describe the connection between the categories of elements of single-object Segal presheaves and the twisted arrow categories of the corresponding algebras. Then we give two possible definitions of Segal presheaves and explain the connection between these definitions. We show that the cyclic nerve of a category is a part of Segal condition. We introduce generalized decomposition spaces, or 2-Segal sets, and describe their connection with special morphisms of operads. This connection is used to prove that Segal presheaves over nice operads can be seen as algebras. This allows for a simple proof that the twisted arrow category of the operad of single-coloured operads is equivalent to the category $\Omega$. In Section~\ref{sec:final} we give another definition of twisted arrow category, give an explanation for the concrete definition of canonical decomposability property, and show that twisted arrow categories of algebras have 2-categorical nature. We show that twisted arrow category of an operad is the $\infty$-localization of the category of elements of its dendroidal nerve by the active morphisms. In Appendix~\ref{sec:appendix-cat-th} we recall basic facts about strict factorization systems and prove lemmas used in the main part. In Appendix~\ref{sec:ez-cats} we show how to check if the twisted arrow category of an operad is a particularly nice generalized Reedy category. In Appendix~\ref{sec:quasi-cat} we give a possible construction of twisted arrow set of a dendroidal set.

\paragraph{Acknowledgements.} I would like to thank Toshitake Kohno for his advice and support, and Anton Khoroshkin, Genki Sato, Christine Vespa, Michal Wasilewicz and Jun Yoshida for helpful discussions. I would also like to thank the anonymous referee for suggesting improvements. The work was supported by the Japanese Government (MEXT) scholarship.

\paragraph{Notation and conventions.}  Operads are \emph{graded}, coloured, symmetric, and have identity operations. Gradings of operads are always \emph{nice}, see Definition~\ref{def:graded-operad}. Symmetric groups act on operations on the right. Free and forgetful functors are denoted by $F$ and $U$. Expressions $p(a_1,\dots,a_n)$ denote both elements $(p,a_1,\dots,a_n)$ in free algebras $FU(A)$ and their images in $A$. Free operads are constructed not from symmetric sequences, but from $\mathbb{N}$-sequences: symmetric groups act freely on operations of a free operad. The operad whose category of algebras is the category of uncoloured operads (\cite{berger2007resolution}) is denoted by $sOp$, and the similar operad encoding $C$-coloured operads is denoted by $sOp_C$. 

\section{Preliminaries}
\label{sec:preliminaries}

We recall \emph{graph-substitution operads}, the symmetric operads that encode generalized operads. We list all the graph-substitution operads that appear in the present work. Then we describe the four symmetric operads that encode cyclic monoids.

\subsection{Graph-substitution operads}
\label{subsect:mop}

Graph-substitution operads are the operads that encode generalized operads. In most cases the colours of these operads are the integers $n$, $n\geq -1$. The operations of non-zero arity are non-empty connected graphs with half-edges, endowed with additional structure. The input colours of such an operation $p$ correspond to the vertices of the graph of $p$, with the colour of a vertex of degree $n$ equal to $(n-1)$. The operadic substitution $p\circ_i q$ is obtained by substitution of the graph of $q$ into the $i$-th vertex of the graph of $p$.

\begin{definition}
A graph with half-edges $G=(V,H,t,inv)$ is a finite set of \emph{vertices} $V$, a finite set of \emph{half-edges} $H$, an \emph{adjacency} map $t:H\to V$, and an involution $inv:H\to H$. A \emph{leaf} of $G$ is a fixed point of $inv$. An \emph{inner edge} of $G$ is a two-element orbit of $inv$. The set of vertices $V$ together with the set of inner edges is a graph, which may have loops and multiple edges. This graph is always assumed to be non-empty and connected.
\end{definition}

\begin{definition}
An operadic graph is a graph with half-edges $(V,H,t,inv)$ endowed with linear order on vertices, linear order on leaves, and, for every vertex $v\in V$, linear order on the set $t^{-1}(v)$ of half-edges adjacent to $v$. Linear orders on half-edges will be given by bijections with sets $[n]=\{0,\dots,n\}$. Linear order on vertices will be given by bijection with a set $\{1,\dots,n\}$. 
\end{definition}

\begin{figure}[b]
\centering
\begin{tikzpicture}[scale=1]
\foreach \place/\name/\label in {{(-2,0)/a/$1$}, {(0,0)/b/$2$}, {(2,0)/c/$3$}, {(4,0)/d/$4$}}
    \node[circ] (\name) at \place {\label};
\foreach \place/\name/\label in {{(-2.5,1)/a1/$2$}, {(-2,1)/a2/$0$}, {(2.5,1.5)/c1/$1$}}
    \node[leaf,label=\label] (\name) at \place {};
  \draw (a) to [out=60,in=90,looseness=1.5] (b);
  \draw (b) to [out=145,in=160,looseness=2] (c);
  \draw (b) to [out=60,in=110,looseness=1.2] (c);
  \draw (a) to [out=150,in=150,looseness=2] (d);
  \draw (c) to [out=30,in=80,looseness=1.2] (d);
  \draw (d) to [out=115,in=40,looseness=15] (d);
  \draw (a) to (a1);
  \draw (a) to (a2);
  \draw (c) to (c1);
\end{tikzpicture}
\caption{An operadic graph, an element of $mOp(3,2,3,3;2)$. For each vertex its adjacent half-edges are ordered from left to right. The orders on leaves and on vertices correspond to the indexing.}
\label{fig:mop-example1}
\end{figure}
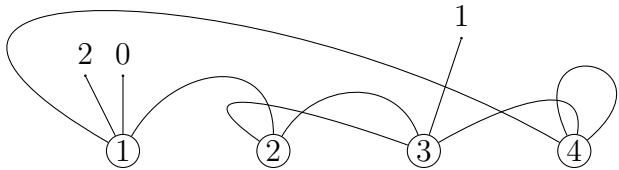

Two operadic graphs are isomorphic if there are bijections between their vertices and between their half-edges that respect all the structure maps and linear orders. Isomorphic graphs will be treated as the same graph.

\begin{definition}
For any $m>0$ and $n_0,n_1,\dots,n_m\geq -1$ the set $mOp(n_1,\dots,n_m;n_0)$ is the set of operadic graphs with $m$ vertices and with $(n_0+1)$ leaves, such that for all $j$ the degree of the $j$-th vertex $v_j$ is equal to $(n_j+1)$. Define additionally $mOp(;1)$ to be the singleton set with element $\mu_0$, called  \emph{the exceptional edge}, and $mOp(;-1)$ to be the singleton set with element $\bigcirc$, called \emph{the nodeless loop}.
\end{definition}

Examples of operadic graphs are given in Figures~\ref{fig:mop-example1}, \ref{fig:mop-example2} and  \ref{fig:mop-example3}. The last figure suggests that $mOp$ is an operad, somewhat similar to the little-disks operad $\mathcal{D}_2$. An example of composition in $mOp$ is given in Figure~\ref{fig:mop-composition}. Formally this composition is defined as follows.

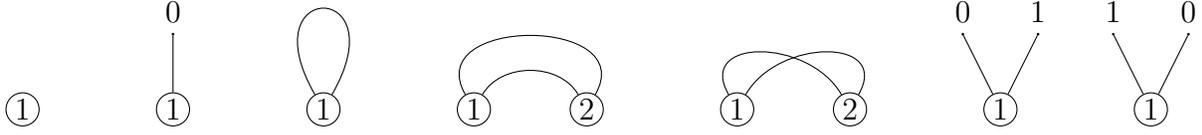
\begin{figure}[t]
\centering
\begin{tikzpicture}[scale=1]
\foreach \place/\name/\label in {{(-4,0)/aa/$1$}, {(0,0)/ba/$1$}, {(2,0)/ca/$1$}, {(3.5,0)/cb/$2$}, {(5.5,0)/da/$1$}, {(7,0)/db/$2$}, {(9,0)/ea/$1$}, {(11,0)/fa/$1$}, {(-2,0)/ga/$1$}}
    \node[circ] (\name) at \place {\label};
\foreach \place/\name/\label in {{(8.5,1)/e1/$0$}, {(9.5,1)/e2/$1$}, {(10.5,1)/f1/$1$}, {(11.5,1)/f2/$0$}, {(-2,1)/g1/$0$}}
    \node[leaf,label=\label] (\name) at \place {};
  \draw (ba) to [out=120,in=60,looseness=20] (ba);
  \draw (ca) to [out=120,in=60,looseness=1.8] (cb);
  \draw (ca) to [out=60,in=120,looseness=1] (cb);
  \draw (da) to [out=60,in=60,looseness=1.5] (db);
  \draw (da) to [out=120,in=120,looseness=1.5] (db);
  \draw (ea) to (e1);
  \draw (ea) to (e2);
  \draw (fa) to (f1);
  \draw (fa) to (f2);
  \draw (ga) to (g1);
\end{tikzpicture}
\caption{The unique element $id_{-1}$ of $mOp(-1;-1)$, the unique element $id_0$ of $mOp(0;0)$, the unique element $\nu$ of $mOp(1;-1)$, the only two elements $\beta$ and $\beta\circ_2\xi$ of $mOp(1,1;-1)$, and the only two elements $id_1$ and $\xi$ of $mOp(1;1)$.}
\label{fig:mop-example2}
\end{figure}

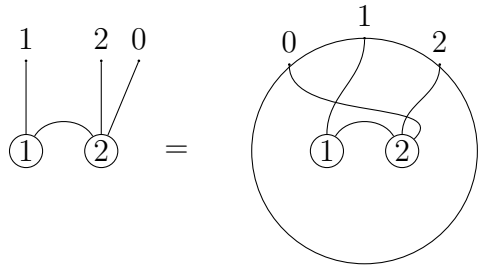
\begin{figure}[b]
\centering
\begin{tikzpicture}[scale=1]
\foreach \place/\name/\label in {{(0,0)/xa/$1$}, {(1,0)/xb/$2$}}
    \node[circ] (\name) at \place {\label};
\foreach \place/\name/\label in {{(0,1.2)/xa1/$1$}, {(1.5,1.2)/xb1/$0$}, {(1,1.2)/xb2/$2$}}
    \node[leaf,label=\label] (\name) at \place {};
    \draw (xa) to [out=60,in=120] (xb);
    \draw (xb) to (xb2);
    \draw (xb) to (xb1);
    \draw (xa) to (xa1);
    \node () at (2,0){$=$};
\foreach \place/\name/\label in {{(4,0)/a/$1$}, {(5,0)/b/$2$}}
    \node[circ] (\name) at \place {\label};
\foreach \place/\name/\label in {{(4.5,1.5)/a1/$1$}, {(3.5,1.15)/b1/$0$}, {(5.5,1.15)/b2/$2$}}
    \node[leaf,label=\label] (\name) at \place {};
    \draw (4.5,0) circle (1.5);
    \draw (a) to [out=60,in=120] (b);
    \draw (b) to [out=90,in=-90] (b2);
    \draw (b) to [out=45,in=-90] (b1);
    \draw (a) to [out=90,in=-90] (a1);
\end{tikzpicture}
\caption{An operadic graph from $mOp(1,2;2)$ and its alternative depiction.}
\label{fig:mop-example3}
\end{figure}

\begin{definition}
Let $p$ and $q$ be operadic graphs, with the number of leaves of $q$ equal to the degree of $v_i$, the $i$-th vertex of $p$. The operadic graph $p\circ_i q$ is the graph obtained by substitution of the graph $q$ into the vertex $v_i$: the set of vertices $V_{p \circ_i q}$ is the set $V_q\sqcup V_p\setminus\{v_i\}$; the set of half-edges $H_{p \circ_i q}$ is the set $(H_p\sqcup H_q)/{\sim}$, where $\sim$ identifies the $j$-th leaf of $q$ with the $j$-th half-edge of $v_i$ for all $j$; the adjacency map $t_{p\circ_i q}$ coincides with $t_q$ on $H_q$ and with $t_p$ on $H_{p \circ_i q}\setminus H_q$; the involution $inv_{p\circ_i q}$ coincides with $inv_p$ on $H_p$ and with $inv_q$ on $H_{p\circ_i q}\setminus H_p$; the set of leaves of $p$ coincides with the set of leaves of $p\circ_i q$, and by definition the orders on these sets are the same; the sets $t_{p\circ_i q}^{-1}(v)$ always coincide with either $t_p^{-1}(v)$ or $t_q^{-1}(v)$, and again by definition the orders on these sets are the same; the $k$-th vertex of $V_q$ is the $(k+i-1)$-th vertex of $V_{p\circ_i q}$, and the $k$-th vertex of $V_p$ is the $k$-th vertex of $V_{p\circ_i q}$ if $k<i$, and the $(k+|V_q|-1)$-th vertex of $V_{p\circ_i q}$ if $k>i$.

Let $p$ be an operadic graph with at least two vertices, with the $i$-th vertex $v_i$ having degree $2$. The operadic graph $p\circ_i \mu_0$ is obtained by replacing the $i$-th vertex of $p$ with an edge: the set $V_{p\circ_i \mu_0}$ is the set $V_p\setminus\{v_i\}$; the set $H_{p\circ_i \mu_0}$ is the set $t_p^{-1}(V_{p\circ_i \mu_0})=H_p\setminus t_p^{-1}(v_i)$; the adjacency map $t_{p\circ_i \mu_0}$ is the restriction of $t_p$ to $H_{p\circ_i \mu_0}$. Let $t_p^{-1}(v_i)=\{h_1,h_2\}$, $h'_1=inv_p(h_1)$, $h'_2=inv_p(h_2)$. The map $inv_{p\circ_i \mu_0}$ coincides with $inv_p$ outside of $h'_1$ and $h'_2$. If $h_1$ is a leaf, then $h'_2$ becomes a leaf, if $h_2$ is a leaf, then $h'_1$ becomes a leaf, otherwise $inv_{p\circ_i \mu_0}(h'_1)$ is equal to $h'_2$. There is a bijection between the leaves of $p\circ_i\mu_0$ and the leaves of $p$.
The orders are preserved.

Finally, let $p$ be an operadic graph with only one vertex, of degree $2$. If $p$ has a loop, then we define $p\circ_1 \mu_0$ to be $\bigcirc$. Otherwise $p\circ_1 \mu_0$ is $\mu_0$.
\end{definition}

\begin{figure}[t]
\centering
\begin{tikzpicture}[scale=1]
\foreach \place/\name/\label in {{(0,0)/xa/$1$}, {(1,0)/xb/$2$}, {(2,0)/xc/$3$}}
    \node[circ] (\name) at \place {\label};
\foreach \place/\name/\label in {{(-0.2,0.5)/xa1/$2$}, {(1.2,0.5)/xb1/$1$}, {(2.2,0.5)/xc1/$0$}}
    \node[leaf,label=\label] (\name) at \place {};
    \draw (xa) to [out=60,in=160,looseness=1.2] (xb);
    \draw (xb) to [out=135,in=90,looseness=20] (xb);
    \draw (xb) to [out=45,in=120,looseness=1.2] (xc);
    \draw (xb) to (xb1);
    \draw (xa) to (xa1);
    \draw (xc) to (xc1);
    \node () at (3,0){$\circ_2$};
\foreach \place/\name/\label in {{(4,0)/ya/$1$}, {(5,0)/yb/$2$}, {(6,0)/yc/$3$}}
    \node[circ] (\name) at \place {\label};
\foreach \place/\name/\label in {{(3.4,0.5)/ya1/$1$}, {(5,0.5)/yb1/$2$}, {(5.6,0.5)/yc1/$0$}, {(6,0.5)/yc2/$3$}, {(6.4,0.5)/yc3/$4$}}
    \node[leaf,label=\label] (\name) at \place {};
    \draw (ya) to [out=30,in=150] (yb);
    \draw (yb) to [out=30,in=150] (yc);
    \draw (ya) to (ya1);
    \draw (yb) to (yb1);
    \draw (yc) to (yc1);
    \draw (yc) to (yc2);
    \draw (yc) to (yc3);
    \draw (ya) to [out=120,in=60,looseness=15] (ya);
    \node () at (7,0){$=$};
\foreach \place/\name/\label in {{(8,0)/za/$1$}, {(9,0)/zb/$2$}, {(10,0)/zc/$3$}, {(11,0)/zd/$4$}, {(12,0)/ze/$5$}}
    \node[circ] (\name) at \place {\label};
\foreach \place/\name/\label in {{(7.6,0.5)/za1/$2$}, {(11,0.5)/zd1/$1$}, {(12.4,0.5)/ze1/$0$}}
    \node[leaf,label=\label] (\name) at \place {};
    \draw (za) to [out=60,in=120,looseness=1.5] (zd);
    \draw (zb) to [out=150,in=90,looseness=2.5] (zc);
    \draw (zb) to [out=30,in=150,looseness=1] (zc);
    \draw (zb) to [out=110,in=70,looseness=10] (zb);
    \draw (zc) to [out=60,in=150,looseness=1] (zd);
    \draw (zd) to [out=60,in=120,looseness=1.2] (ze);
    \draw (za) to (za1);
    \draw (zd) to (zd1);
    \draw (ze) to (ze1);
\end{tikzpicture}
\caption{An example of operadic composition in $mOp$.}
\label{fig:mop-composition}
\end{figure}
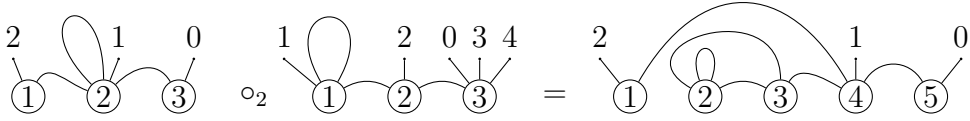

\begin{proposition}
The collection of sets $mOp$ with partial composition as described above is an operad.
\end{proposition}
\begin{proof}
Symmetric groups act on the order of vertices: a vertex indexed by $j$ in $p$ is indexed by $\sigma^{-1}(j)$ in $p^\sigma$. The identity operations $id_n$ are the graphs with one vertex and $(n+1)$ half-edges, with all half-edges being leaves, and with the order on $H_{id_n}$ as the set of leaves equal to the order on $H_{id_n}$ as the set of half-edges adjacent to the same vertex. In particular, $id_{-1}\in mOp(-1;-1)$ is the vertex without half-edges. The operad axioms of right group action, unitality and compatibility of partial composition with group action are straightforward to check. 

To see that parallel associativity holds, take any elements $p$, $q$, $r$ of $mOp$ and integers $i>j$ such that the composition $(p\circ_i q)\circ_j r$ exists. The subsets $t^{-1}_p(v_i)$ and $t^{-1}_p(v_j)$ of $H_p$ do not intersect, and this implies the following. If both $q$ and $r$ have at least one vertex, with $r$ having $k$ vertices, then $H_{(p\circ_i q)\circ_j r}=(H_p\sqcup H_q\sqcup H_r)/{\sim}=H_{(p\circ_j r)\circ_{i+k-1} q}$, and the subsets $H_q$ and $H_r$ of this set do not intersect. Thus the maps $inv$ and $t$ of $(p\circ_i q)\circ_j r$ and of $(p\circ_j r)\circ_{i+k-1} q$ coincide. If $r$ is equal to $\mu_0$, and $q$ has at least one vertex, then $H_{(p\circ_i q)\circ_j \mu_0} = H_{p\circ_i q}\setminus t_{p\circ_i q}^{-1}(v_j) = ((H_p\sqcup H_q)/{\sim})\setminus t_p^{-1}(v_j) = ((H_p\setminus t_p^{-1}(v_j))\sqcup H_q)/{\sim} = H_{(p\circ_j \mu_0)\circ_{i-1} q}$. The map $t_{(p\circ_i q)\circ_j \mu_0}$ is the restriction of $t_{p\circ_i q}$, and it coincides with $t_q$ on $H_q$ and with the restriction of $t_p$ on $H_p\setminus t_p^{-1}(v_j)$, and thus this map is equal to $t_{(p\circ_j\mu_0)\circ_{i-1} q}$. The elements $h'_1$ and $h'_2$ of $H_{(p\circ_i q)\circ_j \mu_0}$ from the definition of the composition with $\mu_0$ are elements of $H_p$, which implies that $inv_{(p\circ_i q)\circ_j \mu_0}=inv_{(p\circ_j \mu_0)\circ_{i-1} q}$. The case $q=r=\mu_0$ is checked case by case: when $v_i$ and $v_j$ are not connected by an inner edge, when they are connected by an inner edge and one of their adjacent half-edges is a leaf, and when all of their half-edges belong to inner edges.

For sequential associativity, the case when the operation $r$ in a composition $p\circ_i(q\circ_j r)$ has non-zero arity is simple. When $r=\mu_0$, either a new leaf in $q$ is created by removing the $j$-th vertex, or a new inner edge is created. Again, in both cases it is easy to see that sequential associativity holds.
\end{proof}

\paragraph{The list of graph-substitution operads.} 

Many important operads are closely related to the operad $mOp$. We call these operads \emph{graph-substitution operads}. Below we list the graph-substitution operads that appear in the present work. Unless stated otherwise, these operads contain the exceptional edge $\mu_0$. Up to this additional operation, graph-substitution operads are defined as follows. 

Several versions of the operad encoding modular operads:

\begin{itemize}
    \item The operad $mOp_{(g,n)}$, with colours given by pairs $(g,n)$ with $g\geq 0$ and $n\geq -1$. The operations are operadic graphs, as in $mOp$, but additionally endowed with a genus map $g:V\to\mathbb{N}$. The colour of a vertex $v$ is $(g(v),deg(v)-1)$. The total genus of a graph $G$ is $\sum_{v\in V} g(v) + b_1(G)$, where $b_1$ is the first Betti number. The operation $\mu_0$ is of colour $(0,1)$. 
    \item The suboperad $mOp_{st}$ of $mOp_{(g,n)}$ on all operations that satisfy $g(v)+deg(v)>2$ for all vertices $v$; it does not contain $\mu_0$. Modular operads, in the original sense of \cite{getzler1998modular}, i.e.\@ \emph{stable} modular operads, are precisely the algebras over $mOp_{st}$.
    \item The suboperad $mOp_{nc}$ of $mOp$ that consists of all operations that are not a circle with marked vertices or the nodeless loop $\bigcirc$.
\end{itemize}

The operads encoding generalized operads based on trees:

\begin{itemize}
    \item The suboperad $cOp'$ of $mOp$ consists of all operadic graphs that are trees. If the set $cOp(n_1,\dots,n_k;n_0)$ is not empty, then $n_0=\sum n_j - k + 1$. We will denote these sets simply as $cOp(n_1,\dots,n_k)$. The same applies to suboperads of $cOp'$.
    \item The suboperad $cOp$ of $cOp'$ consists of all trees with at least one leaf. The operad $cOp$ encodes cyclic operads.
    \item The suboperad $sOp$ of $cOp$ consists of rooted trees, i.e.\@ of trees such that: every inner edge contains exactly one $0$-th half-edge of the two adjacent vertices; the $0$-th leaf, called \emph{root}, is the $0$-th half-edge of its vertex, which is called \emph{the root vertex}. Elements of $sOp$ can be seen as planar rooted trees endowed with arbitrary order on non-root leaves, see Figure~\ref{fig:sop-example1}. The algebras over the operad $sOp$ are the uncoloured symmetric operads (\cite{berger2007resolution}). 
    \item The suboperad $pOp$ of $sOp$ consists of planar rooted trees endowed with planar order on non-root leaves. Its algebras are planar operads. 
    \item The operad $iuAs$ consists of trees with all vertices of degree $2$, i.e.\@ of lines with vertices endowed with orientation. Its algebras are monoids with anti-involution.
    \item The operad $uAs$ is the intersection of $sOp$ and $iuAs$. Its algebras are monoids. There are two natural maps $uAs\to sOp$ that differ by the involution of $uAs$ (categories can be regarded as operads in two ways, equally (non)canonical). 
\end{itemize}

The four operads related to cyclic monoids (monoids with a compatible bilinear form, or, equivalently, a trace map): 

\begin{itemize}
    \item The suboperad $ciuAs$ of $mOp$ consists of the operadic graphs whose vertices have degree $2$. It additionally contains $\bigcirc$ and the vertex without leaves. Its algebras are cyclic monoids with anti-involution. (The letter \emph{c} in $ciuAs$ stands for \emph{cyclic}, \emph{circle} or for the fact that $ciuAs$ is obtained from the cyclic operad $iuAs$ in a universal way described in the next subsection.)
    \item The operads $iuAs_{Tr}$ and $iuAs_{iTr}$, see Definition~\ref{def:operads-involution-trace}. Their algebras can also be seen as cyclic monoids with anti-involution.
    \item The suboperad $cuAs$ of $ciuAs$  consists of the operadic graphs whose vertices have degree $2$, all inner edges contain exactly one $0$-th and one $1$-st half-edge of adjacent vertices; if an element of $cuAs$ has leaves, then the $0$-th leaf is required to be the $0$-th half-edge of a vertex. These operadic graphs can be seen as lines and circles with marked points, without order on half-edges adjacent to vertices. The operad $cuAs$ additionally contains $\bigcirc$ and the vertex without leaves. The corresponding algebras are cyclic monoids.
\end{itemize}

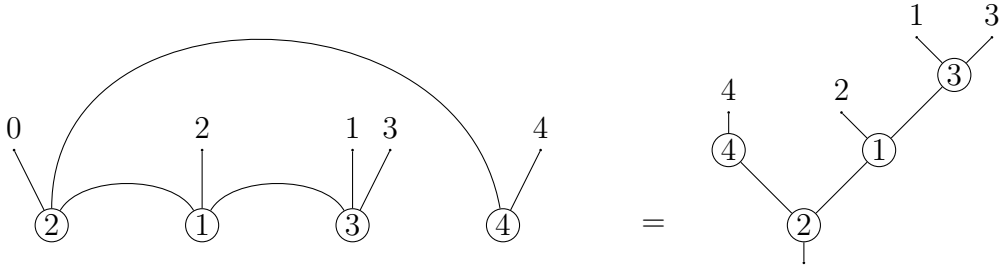
\begin{figure}[t]
\centering
\begin{tikzpicture}[scale=1]
\foreach \place/\name/\label in {{(0,0)/a/$2$}, {(2,0)/b/$1$}, {(4,0)/c/$3$}, {(6,0)/d/$4$}}
    \node[circ] (\name) at \place {\label};
\foreach \place/\name/\label in {{(-0.5,1)/a1/$0$}, {(2,1)/b1/$2$}, {(4,1)/c1/$1$}, {(4.5,1)/c2/$3$}, {(6.5,1)/d1/$4$}}
    \node[leaf,label=\label] (\name) at \place {};
  \draw (a) to [out=60,in=120,looseness=0.8] (b);
  \draw (a) to [out=90,in=100,looseness=1.3] (d);
  \draw (b) to [out=60,in=120,looseness=0.8] (c);
  \draw (a) to (a1);
  \draw (b) to (b1);
  \draw (c) to (c1);
  \draw (c) to (c2);
  \draw (d) to (d1);
  \node () at (8,0) {$=$};
\foreach \place/\name/\label in {{(10,0)/xa/$2$}, {(11,1)/xb/$1$}, {(12,2)/xc/$3$}, {(9,1)/xd/$4$}}
    \node[circ] (\name) at \place {\label};
\foreach \place/\name/\label in {{(10,-0.5)/xa1/{}}, {(10.5,1.5)/xb1/$2$}, {(11.5,2.5)/xc1/$1$}, {(12.5,2.5)/xc2/$3$}, {(9,1.5)/xd1/$4$}}
    \node[leaf,label=\label] (\name) at \place {};
  \draw (xa1) to (xa);
  \draw (xa) to (xb);
  \draw (xa) to (xd);
  \draw (xb) to (xc);
  \draw (xd) to (xd1);
  \draw (xc) to (xc1);
  \draw (xc) to (xc2);
  \draw (xb) to (xb1);
\end{tikzpicture}
\caption{An element of $sOp(2,2,2,1;4)$ represented as a planar rooted tree.}
\label{fig:sop-example1}
\end{figure}

These operads will suffice for our purposes. There are other operads that encode generalized operads (\cite[Table 1]{merkulov2010wheeled}). We mention these only in Lemma~\ref{lem:graph-operads-canonically-decomposable}.  

\begin{remark}
For any set of colours $C$ there is an operad $sOp_C$ such that the category of $sOp_C$-algebras is the category of $C$-coloured operads and colour-preserving morphisms. The set of colours of $sOp_C$ is $C'=\bigsqcup_{n\in\mathbb{N}} C^{\times (n+1)}$. An element of $sOp_C$ is a planar rooted operadic tree with any permutation on leaves, endowed with a map from the set of half-edges to $C$ such that the two half-edges of any inner edge have the same colour. In particular, the operad $sOp$ is an $sOp_\mathbb{N}$-algebra, and the operad $sOp_C$ is an $sOp_{C'}$-algebra.
\end{remark}

\paragraph{Graded operads.}

Operads in this work are graded. Any operad can be graded via arity, and for most operads appearing in practice the arity grading is the only reasonable grading. However, for some operads the correct grading is different from the arity grading.

\begin{definition}
\label{def:graded-operad}
A grading $\psi$ on an operad $P$ is a map from operations of $P$ to $\mathbb{Z}$, or to any other abelian monoid, such that for any operations $p$ and $q$ we have $\psi(p\circ_i q)=\psi(p)+\psi(q)$, $\psi(id_c)=0$ and $\psi(p^\sigma)=\psi(p)$. For any operad $P$ \emph{the arity grading} $\phi$ sends an operation $p$ of arity $n$ to $(n-1)$. We will always assume that the codomain of a grading $\psi$ is $\mathbb{Z}$, that $\psi(p)\geq -1$ for all operations $p$, that operations of grading $(-1)$ have arity $0$, and that units of binary operations have grading $(-1)$. Such grading will be called \emph{nice}. While we do not assume that operations of grading $0$ have arity $1$, we will show that this property is necessary for the twisted arrow category of $P$ to be generalized Reedy.
\end{definition}

\begin{definition}
The grading $\psi$ on the operads $mOp$ and $mOp_{(g,n)}$ and on their suboperads sends an operadic graph $p$ with underlying graph $G=(V,H)$ to $(|V|-1+ 2 b_1(G))$, and sends $\mu_0$ to $(-1)$ and $\bigcirc$ to $1$.  
\end{definition}

The map $\psi$ is a grading since the set of internal edges of $p\circ_i q$ is the disjoint union of the sets of internal edges of $p$ and of $q$, and $b_1(G)=|E|-|V|+1$. On suboperads of $mOp$ that contain only trees the grading $\psi$ coincides with the arity grading $\phi$. 

\subsection{Symmetric operads encoding cyclic algebras}

Denote by $\tau_n=(0\dots n)$ the cyclic permutation in $S_{n+1}$. A cyclic operad $P$ is an algebra over the operad $cOp$, or, equivalently, a single-colour symmetric operad endowed with an action of $\tau_n$ on $P(n)$ for all $n>0$, such that for any $p\in P(m)$ and $q\in P(n)$ the following holds: $(p\circ_1 q)^{\tau_{m+n-1}}=q^{\tau_n}\circ_n p^{\tau_m}$, if $n>0$; $(p\circ_1 q)^{\tau_{m-1}}=p^{\tau^2_m}\circ_m q$, if $n=0$; $(p\circ_i q)^{\tau_{m+n-1}}=p^{\tau_m}\circ_{i-1} q$, for $i>1$; $id^{\tau_1}=id$; and the actions of $\tau_n$ and of $S_n$ on $P(n)$ together generate an action of $S_{n+1}$ on $P(n)$.

\begin{proposition}
The operads $uAs$ and $iuAs$ are cyclic.
\end{proposition}
\begin{proof}
Let $\beta$ be the circle with two vertices from Figure~\ref{fig:mop-example2}. There is a bijection between $uAs$ and $cuAs\setminus uAs$ and a bijection between $iuAs$ and $ciuAs\setminus iuAs$. Both bijections are given by the map $p\mapsto \beta\circ_2 p$. Cyclic permutations $\tau_n$ act on operations in $uAs$ and $iuAs$ by cyclic permutations on indices of vertices of the corresponding circles in $cuAs$ and $ciuAs$. In particular,  permutations $\tau_n$ decrease indices of vertices of circles by $1$ modulo $(n+1)$, thus $\mu_n^{\tau_n}=\mu_n$ and $\xi^{\tau_1}=\xi$, where $\xi$ is the non-trivial element of $iuAs(1)$.
This description makes it easy to check directly that the operads $uAs$ and $iuAs$ are cyclic.
\end{proof}

\begin{definition}
Let $P$ be a cyclic operad in a symmetric monoidal category $\mathcal{C}$. A cyclic $P$-algebra $A$ is an algebra over $P$, seen as a symmetric operad, together with a map $B:A\otimes A\to V$ in $\mathcal{C}$ such that the maps $B_n:P(n)\otimes A^{\otimes(n+1)}\to V$ defined by $B_n(p\otimes a_0\otimes\dots\otimes a_n)=B(a_0\otimes p(a_1\otimes\dots\otimes a_n))$ are compatible with the action of $\tau_n$, i.e.\@ $B_n(p^{\tau_n}\otimes a_0\otimes\dots\otimes a_n)=B_n(p\otimes a_n\otimes a_0\otimes\dots\otimes a_{n-1})$.  
\end{definition}

\begin{proposition}
Let $P$ be a cyclic $\{c_1\}$-coloured operad. Define $cyc(P)$ to be the $\{c_1,c_2\}$-coloured symmetric operad $(P\sqcup F(\beta))/{\sim}$, with $\beta\in F(\beta)(c_1,c_1;c_2)$, and with equivalence $\sim$ generated by $\beta\circ_2 p^{\tau_n}\sim (\beta\circ_2 p)^{(1\dots(n+1))}$. Cyclic $P$-algebras are precisely the algebras over the coloured symmetric operad $cyc(P)$.
\end{proposition}
\begin{proof}
Any $cyc(P)$-algebra $A$ is a $P$-algebra via the inclusion map $P\to cyc(P)$. The structure map of $A$ corresponding to $\beta$ is the bilinear map of the cyclic $P$-algebra $A$. The equivalence relation in $cyc(P)$ ensures that the corresponding maps $B_n$ are compatible with the action of $\tau_n$.
\end{proof}

\begin{remark} For any cyclic operad $P$ the operation $\beta$ in $cyc(P)$ is symmetric: $\beta=\beta\circ_2 id=\beta\circ_2 id^{\tau_1}=(\beta\circ_2 id)^{(12)}=\beta^{(12)}$. This also shows that the bilinear map $B:A\otimes A\to V$ of any cyclic algebra is symmetric with respect to the symmetric product $\otimes$.
\end{remark}

\begin{proposition}
The operad $cuAs$ is isomorphic to $cyc(uAs)$. The operad $ciuAs$ is isomorphic to $cyc(iuAs)$.
\end{proposition}
\begin{proof}
The circle with two vertices corresponds to the operation $\beta$. The vertex without leaves corresponds to $id_{c_2}$. The nodeless loop $\bigcirc$ corresponds to $(\beta\circ_2\mu_0)\circ_1\mu_0$ and encodes the map $I\to V$.

The obvious morphism $uAs\sqcup F(\beta)\to cuAs$ that sends $\beta$ to the circle with two vertices induces the morphism $cyc(uAs)\to cuAs$, since generating relations of $cyc(uAs)$ hold in $cuAs$. Take the operation $\nu=\beta\circ_1\mu_0$ in $cyc(uAs)$. We have $\nu\circ_1\mu_2=(\beta\circ_1\mu_0)\circ_1\mu_2=(\beta\circ_2\mu_2)\circ_1\mu_0=(\beta\circ_2\mu_2)^{(123)}\circ_1\mu_0=((\beta\circ_2\mu_2)\circ_2\mu_0)^{(12)}=(\beta\circ_2(\mu_2\circ_1\mu_0))^{(12)}=(\beta\circ_2 id_{c_1})^{(12)}=\beta$. Thus any operation $p$ in $cyc(uAs)$ with output colour $c_2$ and different from $id_{c_2}$ is of the form $\nu\circ_1 q$, with $q$ in $uAs$. If the operation $p$ above has non-zero arity, then $p=\beta\circ_2 r$ for some $r$ in $uAs$. The morphism $cyc(uAs)\to cuAs$ is bijective: the operation $r$ above is determined uniquely by the image of $p$ in $cuAs$. 

Similarly, there is a morphism $cyc(iuAs)\to cuiAs$, and any operation $p$ in $cyc(iuAs)$ with output colour $c_2$ and different from $id_{c_2}$ is of the form $\nu\circ_1 q$, with $q$ in $iuAs$. In $cyc(iuAs)$ we have $\nu\circ_1\xi=(\beta\circ_2(\xi\circ_1\mu_0))\circ_1\xi=((\beta\circ_2\xi)\circ_2\mu_0)\circ_1\xi=((\beta\circ_2\xi)^{(12)}\circ_2\mu_0)\circ_1\xi=((\beta\circ_2\xi)\circ_1\mu_0)\circ_1\xi=\nu\circ_1\xi\circ_1\xi=\nu$ and $\beta=\nu\circ_1\mu_2=\nu\circ_1\xi\circ_1\mu_2=(\nu\circ_1\mu_2)^{(12)}\circ_1\xi\circ_2\xi=\beta\circ_1\xi\circ_2\xi$. If the operation $p$ above has non-zero arity, then $p$ is of the form $\beta\circ_2 r$ with $r$ in $iuAs$. Again, the operation $r$ is determined uniquely by the image of $p$ in $ciuAs$, and the map $cyc(iuAs)\to ciuAs$ is bijective.
\end{proof}

Later we will see that the operad $ciuAs$ does not satisfy properties shared by many graph-substitution operads. There are two other operads, $iuAs_{Tr}$ and $iuAs_{iTr}$, whose algebras are also monoids with anti-involution and with a trace map. The difference between $ciuAs$, $iuAs_{Tr}$ and $iuAs_{iTr}$ is that the first operad describes algebras such that $Tr(a^*)=Tr(a)$, the second operad does not have this requirement, and the third operad is such that $Tr(a^*)=\theta(Tr(a))$, where $\theta$ is the involution of the target of the trace map. Formally these operads are defined as follows.

\begin{definition}
\label{def:operads-involution-trace}
The operad $iuAs_{iTr}$ is the $\{c_1,c_2\}$-coloured operad generated by elements $\mu_0\in iuAs_{iTr}(;c_1)$, $\mu_2\in iuAs_{iTr}(c_1,c_1;c_1)$, $\xi\in iuAs_{iTr}(c_1;c_1)$, $\nu\in iuAs_{iTr}(c_1;c_2)$ and $\theta\in iuAs_{iTr}(c_2;c_2)$, by relations on $\mu_0$, $\mu_2$ and $\xi$ that are the same as the relations in $iuAs$, and by relations $\nu\circ_1\mu_2^{(12)}=\nu\circ_1\mu_2$, $\nu\circ_1\xi=\theta\circ_1\nu$, and $\theta\circ_1\theta=id_{c_2}$. The suboperad $iuAs_{Tr}$ of $iuAs_{iTr}$ is generated by $\mu_0$, $\mu_2$, $\xi$ and $\nu$.
\end{definition}

There is only one operation in $iuAs_{iTr}(;c_2)$ and in $iuAs_{Tr}(;c_2)$, again denoted by $\bigcirc$. This operation is invariant under involution: $\theta\circ_1\bigcirc=\bigcirc$. The remaining elements in $iuAs_{iTr}\setminus iuAs$ can be seen as \emph{oriented} circles, with marked vertices endowed with orders on adjacent half-edges. Substitution into $\theta$ changes the orientation of the circles.

\section{Twisted arrow categories}
\label{sec:main-def}

\subsection{Twisted arrow operads}

We define the twisted arrow operad $\TwOp_Q(A)$ of an algebra $A$ over a set-operad $Q$. This construction is closely related to two classical constructions: the enveloping operad construction $\mathcal{U}Op_Q(A)$ (\cite{getzler1994operads,fresse1998lie}) and the Baez--Dolan plus construction $(Q,A)^+$ (\cite{Baez1998higher3}).

\begin{definition}
Let $A$ be an algebra over a $C$-coloured set-operad $Q$. The $Q$-algebra $EO(A)$ is the algebra generated by the set $H_C=\{h_c,\,c\in C\}\cong C$ and by a copy of $A$, and factored by relations in $A$. In other words, $EO(A)$ is $F(H_C, U(A))/Rel\, A= F(H_C)\sqcup A$, where $F$ and $U$ are the free and the forgetful functors.
\end{definition}

\begin{definition}
Let $A$ be an algebra over a $C$-coloured set-operad $Q$. The $Q$-algebra $TO(A)$ is the algebra generated by two copies of $A$, with the second copy then factored by relations in $A$, i.e.\@ $TO(A)$ is $F(U(A),U(A))/Rel\, A=F(U(A))\sqcup A$. 
\end{definition}

Any element of $EO(A)$ is representable by an element of $F(H_c,U(A))$ of the form  $q(h_{c_1},\dots,h_{c_m},a_1,\dots,a_n)$, and any element of $TO(A)$ is representable by an element of $F(U(A), U(A))$ of the form $q(a'_1,\dots,a'_m,a_1,\dots,a_n)$, with elements $a'_j$ and $a_j$ taken from the first and from the second copy of $A$ respectively. The elements $h_{c_j}$ and $a'_j$ in the expressions above will be called \emph{variables}. Variables in $TO(A)$ will be marked with an apostrophe. For any element in $EO(A)$ or in $TO(A)$ the multiset of its variables is well defined, while the elements $a_j\in A$ in general depend on the choice of a representative. There are two natural maps from $TO(A)$: the evaluation map $ev:TO(A)=F(U(A))\sqcup A\to A$, with $ev([q(a'_1,\dots,a'_m,a_1,\dots,a_n)]) = q(a'_1,\dots,a'_m,a_1,\dots,a_n)$, and the map $TO(A)\to EO(A)$ that sends variables to their colours. By $[q(\dots)]$ we denote the element of $TO(A)$ or of $EO(A)$ represented by the expression $q(\dots)$.

\begin{definition}
Let $A$ be an algebra over a $C$-coloured operad $Q$. The enveloping operad $\mathcal{U}Op_Q(A)$ of $A$ is $C$-coloured. For any $c_0,\dots,c_m\in C$ the set $\mathcal{U}Op_Q(A)(c_1,\dots,c_m;c_0)$ consists of the elements $[q(h_{c_1},\dots,h_{c_m},a_1,\dots,a_n)]$ of $EO(A)$ endowed with a linear order on variables, and with the output colour of $q$ equal to $c_0$. We can, and will, assume that the variables $h_{c_j}$ in expressions $[q(h_{c_1},\dots,h_{c_m},a_1,\dots,a_n)]$ are ordered from left to right. By definition the composition $[p(h_{c_1},\dots,h_{c_m},a_1,\dots,a_n)]\circ_i [q(h_{d_1},\dots,h_{d_k},b_1,\dots,b_l)]$ is equal to $[(p\circ_i q)(h_{c_1},\dots,h_{c_{i-1}},h_{d_1},\dots,h_{d_k},b_1,\dots,b_l,h_{c_{i+1}},\dots,h_{c_m},a_1,\dots,a_n)]$. Symmetric groups act by permutation on the orders of variables. The identity maps are the maps $[id_c(h_c)]$.
\end{definition}

\begin{definition}
Let $A$ be an algebra over a $C$-coloured operad $Q$. The Baez--Dolan plus construction $(Q, A)^+$ of $A$ is an $A$-coloured operad. For any $a'_0,\dots,a'_m\in A$ the set $(Q, A)^+(a'_1,\dots,a'_m; a'_0)$ consists of the elements $[q(a'_1,\dots,a'_m)]$ of $F(U(A))$ endowed with a linear order on variables, and with the evaluation $q(a'_1,\dots,a'_m)$ equal to $a'_0$. We will assume that the variables $a'_j$ in expressions $[q(a'_1,\dots,a'_m)]$ are ordered from left to right.
Composition $[p(a'_1,\dots,a'_m)]\circ_i [q(b'_1,\dots,b'_k)]$ is equal to $[(p\circ_i q) (a'_1,\dots,a'_{i-1}, b'_1,\dots, b'_k, a'_{i+1},\dots,a'_m)]$. Symmetric groups act by permutation on the orders of variables. The identity maps are the maps $[id_c(a')]$.
\end{definition}

\begin{definition}
Let $A$ be an algebra over a $C$-coloured operad $Q$. The twisted arrow operad $\TwOp_Q(A)$ of $A$ is an $A$-coloured operad. For any $a'_0,\dots,a'_m\in A$ the set of operations $\TwOp_Q(A)(a'_1,\dots,a'_m; a'_0)$ consists of the elements $[q(a'_1,\dots,a'_m, a_1,\dots,a_n)]$ of $TO(A)$ endowed with a linear order on variables, and with the evaluation $q(a'_1,\dots,a'_m,a_1,\dots,a_n)$ equal to $a'_0$.
The variables $a'_j$ in expressions $[q(a'_1,\dots,a'_m,a_1,\dots,a_n)]$ are ordered from left to right.
Composition $[p(a'_1,\dots,a'_m,a_1,\dots,a_n)]\circ_i [q(b'_1,\dots,b'_k,b_1,\dots,b_l)]$ is equal to $[(p\circ_i q) (a'_1,\dots,a'_{i-1}, b'_1,\dots, b'_k,b_1,\dots,b_l, a'_{i+1},\dots,a'_m,a_1,\dots,a_n)]$, where variables are all the elements $a'_j$ and $b'_j$. Symmetric groups act by permutation on the orders of variables. Identity maps are the maps $[id_c(a')]$.
\end{definition}

It follows from the definition of composition that $(Q,A)^+$, $\mathcal{U}Op_Q(A)$ and $\TwOp_Q(A)$ are operads. That this composition is well-defined can be checked as follows. The equivalence relation on the elements of $F(U(A),U(A))$ is generated: by relations of the form $[(p\circ_i q)(\dots,a_i,\dots,a_{i+n-1},\dots)]=[p(\dots,q(a_i,\dots,a_{i+n-1}),\dots)]$, where all the elements in the subexpression $q(\dots)$ are in the second copy of $A$; and by compatibility of the algebra structure with the action of symmetric group, i.e.\@ by relations $[q^\sigma(\dots,a_{\sigma(j)},\dots)]=[q(\dots,a_{j},\dots)]$. These generating equivalences on representatives $t_1$ and $t_2$ of elements of $TO(A)$ produce corresponding generating equivalences on representatives of $t_1\circ_i t_2$. Analogous statements hold for the operad $\mathcal{U}Op_Q(A)$.

\begin{proposition}[\cite{getzler1994operads,fresse1998lie}]
Let $A$ be an algebra over a set-operad $Q$. The category of $\mathcal{U}Op_Q(A)$-algebras is equivalent to the slice category $A/Q\mathrm{-alg}$ of $Q$-algebras under $A$.
\end{proposition}
\begin{proof}
For a $\mathcal{U}Op_Q(A)$-algebra $B$ the corresponding $Q$-algebra $B'$ is defined by $B'(c)=B(c)$, with structure maps given by $q(b_1,\dots,b_n)=[q(h_{c_1},\dots,h_{c_n})](b_1,\dots,b_n)$, where $b_j\in B(c_j)$, and with $Q$-algebra map $A\to B'$ sending an element $a$ to the element of $B$ corresponding to the arity $0$ operation $[id(a)]$.
In the opposite direction, for a $Q$-algebra $B'$ and a $Q$-algebra map $f:A\to B'$ the $\mathcal{U}Op_Q(A)$-algebra $B$ is defined by $B(c)=B'(c)$, with structure maps given by $[q(h_{c_1},\dots,h_{c_n},a_1,\dots,a_m)](b_1,\dots,b_n)=q(b_1,\dots,b_n,f(a_1),\dots,f(a_m))$.
\end{proof}

\begin{proposition}
Let $A$ be an algebra over a set-operad $Q$. The category of $(Q,A)^+$-algebras is equivalent to the slice category $Q\mathrm{-alg}/A$ of $Q$-algebras over $A$.
\end{proposition}
\begin{proof}
Let $B$ be a $(Q,A)^+$-algebra. The corresponding $Q$-algebra $B'$ is defined by $B'(c)=\sqcup_{a'\in A(c)} B(a')$, and by $q(b_1,\dots,b_n)=[q(a'_1,\dots,a'_n)](b_1,\dots,b_n)$, where $b_j\in B(a'_j)$. The $Q$-algebra map $B'\to A$ sends $b\in B(a')$ to $a'$. In the opposite direction, let $\varepsilon:B'\to A$ be a $Q$-algebra map. Define $(Q,A)^+$-algebra $B$ by $B(a')=\varepsilon^{-1}(a')$, with the algebra maps given by $[q(a'_1,\dots,a'_n)](b_1,\dots,b_n)=q(b_1,\dots,b_n)$.
\end{proof}

\begin{proposition}
For any algebra $A$ over a set-operad $Q$ the twisted arrow operad of $A$ is isomorphic to the operad $\mathcal{U}Op_{(Q,A)^+}(id_A)$ and to the operad $(\mathcal{U}Op_Q,id_A)^+$. The category of $\TwOp_Q(A)$-algebras is equivalent to the category of $Q$-algebras $B$ over and under $A$ such that the map $A\to B\to A$ is the identity map.
\end{proposition}
\begin{proof}
The construction of isomorphisms is straightforward. The isomorphisms in turn imply the equivalence of categories of algebras. Explicitly this equivalence works as follows.

Let $B$ be a $\TwOp_Q(A)$-algebra. The corresponding $Q$-algebra $B'$ is defined by $B'(c)=\sqcup_{a'\in A(c)} B(a')$, and by $q(b_1,\dots,b_n)=[q(a'_1,\dots,a'_n)](b_1,\dots,b_n)$, where $b_j\in B(a'_j)$. The $Q$-algebra map $A\to B'$ sends an element $a$ to the element of $B$ corresponding to the arity $0$ operation $[id(a)]$. The $Q$-algebra map $B'\to A$ sends $b\in B(a')$ to $a'$. The composition $A\to B'\to A$ is the identity map. 
Let $B'$ be a $Q$-algebra endowed with compatible $Q$-algebra maps $f:A\to B'$ and $\varepsilon:B'\to A$. The $\TwOp_Q(A)$-algebra $B$ is defined by $B(a')=\varepsilon^{-1}(a')$, and by $[q(a'_1,\dots,a'_n,a_1,\dots,a_m)](b_1,\dots,b_n)=q(b_1,\dots,b_n,f(a_1),\dots,f(a_m))$.
\end{proof}

\begin{lemma}
\label{lem:op-opfibration}
Let $A$ be an algebra over a $C$-coloured operad $Q$. There is an inclusion of operads $(Q,A)^+\to\TwOp_Q(A)$ and a discrete opfibration of operads $\alpha:\TwOp_Q(A)\to \mathcal{U}Op_Q(A)$.
\end{lemma}
\begin{proof}
Notice that for any operad $P$ and $P$-algebra $B$ there is canonical inclusion of operads $P\to\mathcal{U}Op_P(B)$. The obvious inclusion $(Q, A)^+\to\TwOp_Q(A)$ can be seen as the inclusion $(Q,A)^+\to\mathcal{U}Op_{(Q,A)^+}(id_A)$.

For the second statement, the map $TO(A)\to EO(A)$ that sends an element to its colour respects the operad structure on $\TwOp_P(A)$ and $\mathcal{U}Op_Q(A)$. For any choice of elements $a'_1,\dots,a'_m$ and an operation $[q(h_{c_1},\dots,h_{c_m},a_1,\dots,a_m)]$ from corresponding colours $c_1,\dots,c_m$ the unique lift of this operation is $[q(a'_1,\dots,a'_m,a_1,\dots,a_m)]$.
\end{proof}

\begin{remark}
The image of $(Q,A)^+$ in $\TwOp_Q(A)$ is the preimage of the suboperad $Q$ of the operad $\mathcal{U}Op_Q(A)$, i.e.\@ the sequence $(Q, A)^+\to\TwOp_Q(A)\to\mathcal{U}Op_Q(A)$ is in a certain sense short exact.
\end{remark}

Recall that any morphism of operads induces adjoint pair of extension and restriction functors between respective categories of algebras.
The canonical opfibration $\TwOp_Q(A)\to\mathcal{U}Op_Q(A)$ induces the left adjoint extension functor equivalent to the forgetful functor $(A/Q\mathrm{-alg})/id_A\to A/Q\mathrm{-alg}$ and the right adjoint restriction functor that sends a map $f:A\to B$ to the map $(f,id_A):A\to B\times A$, with the projection $B\times A\to A$ on the second factor as the augmentation map. This is opposite to the usual case where right adjoint restriction functor is forgetful.

Twisted arrow operads and enveloping operads are functorial constructions, and the functors below commute with the sequence $(Q,A)^+\to\TwOp_Q(A)\to\mathcal{U}Op_Q(A)$.

\begin{proposition}
A morphism $f:A\to B$ of $Q$-algebras induces morphisms $f_*$ between respective twisted arrow operads, enveloping operads and plus constructions.
\end{proposition}
\begin{proof}
The morphism $f$ of $Q$-algebras induces morphisms $FU(A)\to FU(B)$, $FU(A)\sqcup A\to FU(B)\sqcup B=TO(A)\to TO(B)$ and $EO(A)\to EO(B)$ of $Q$-algebras, which in turn induce the morphisms $f_*:(Q,A)^+\to (Q,B)^+$, $f_*:\TwOp_Q(A)\to\TwOp_Q(B)$ and $f_*:\mathcal{U}Op_Q(A)\to\mathcal{U}Op_Q(B)$.
\end{proof}

\begin{proposition}
For any $P$-algebra $A$ a morphism $f:Q\to P$ of set-operads induces  morphisms $f_*:(Q,A)^+\to(P,A)^+$, $f_*:\TwOp_Q(A)\to\TwOp_P(A)$, and $f_*:\mathcal{U}Op_Q(A)\to\mathcal{U}Op_P(A)$.
\end{proposition}
\begin{proof}
Here we implicitly use the restriction functor $P\mathrm{-alg}\to Q\mathrm{-alg}$.
The morphism $f_*:\TwOp_Q(A)\to\TwOp_P(A)$ sends $[q(a'_1,\dots,a'_m,a_1,\dots,a_n)]$ to $[f(q)(a'_1,\dots,a'_m,a_1,\dots,a_n)]$. The remaining cases are similar.
\end{proof}

\begin{proposition}
For any $Q$-algebra $B$ a morphism of operads $f:Q\to P$ induces morphisms $(Q,B)^+\to(P,B')^+$, $\TwOp_Q(B)\to\TwOp_P(B')$, and $\mathcal{U}Op_Q(B)\to\mathcal{U}Op_P(B')$, where $B'$ is the extension of $B$.
\end{proposition}
\begin{proof}
The map of expressions is defined as in the previous proposition. 
\end{proof}

\subsection{Twisted arrow categories}
\label{sec:decomposability}

We define twisted arrow categories of algebras over operads and introduce the property of operads called \emph{canonical decomposability}. Twisted arrow categories and enveloping categories of canonically decomposable operads admit a simple description.

\begin{definition}
Let $A$ be an algebra over a $C$-coloured operad $Q$. The enveloping category $\mathcal{U}_Q(A)$ of $A$ is the underlying category of the operad $\mathcal{U}Op_Q(A)$, i.e.\@ the category of arity $1$ operations of $\mathcal{U}Op_Q(A)$, with $Hom_{\mathcal{U}_Q(A)}(c_1,c_2) = \mathcal{U}Op_Q(A)(c_1;c_2)$. 
\end{definition}

\begin{definition}
Let $A$ be an algebra over a $C$-coloured operad $Q$. The twisted arrow category $\Tw_Q(A)$ of $A$ is the underlying category of the operad $\TwOp_Q(A)$, i.e.\@ $Hom_{\Tw_Q(A)}(a'_1,a'_2)=\TwOp_Q(A)(a'_1;a'_2)$.
\end{definition}

\begin{remark}
The discrete opfibration of operads $\TwOp_Q(A)\to\mathcal{U}Op_Q(A)$ restricts to discrete opfibration of categories $\alpha:\Tw_Q(A)\to\mathcal{U}_Q(A)$.
\end{remark} 

\begin{example}
Let $uAs$ be the operad of (unital associative) monoids and $A$ be a monoid. The objects of $\Tw_{uAs}(A)$ are the elements of $A$. A morphism from an element $a'$ is representable by an element $[\mu_{n+m+1}(a_1,\dots,a_n,a',a_{n+1},\dots,a_{n+m})]$ of $TO(A)$, and this element is equal to $[\mu_3(b,a',c)]$, with $b=\mu_n(a_1,\dots,a_n)$ and $c=\mu_m(a_{n+1},\dots,a_{n+m})$, where $n$ and $m$ might be equal to $0$. Later we will show that for any morphism in $\Tw_{uAs}(A)$ its representative of the form $\mu_3(b,a',c)$ is unique. 
The composition of morphisms $[\mu_3(b_2,d',c_2)]\circ [\mu_3(b_1,a',c_1)]$ is equal to $[\mu_5(b_2,b_1,a',c_1,c_2)]$, i.e.\@ to $[\mu_3(\mu_2(b_2,b_1),a',\mu_2(c_1,c_2))]$. Thus $\Tw_{uAs}(A)$ is the usual twisted arrow category $\Tw(A)$ of a monoid $A$. The discrete opfibration $\Tw_{uAs}(A)\to\mathcal{U}_{uAs}(A)=A\times A^{op}$ corresponds to the copresheaf $Hom:A\times A^{op}\to Sets$.
\end{example}

The following example shows that in general it is impossible to choose canonical representatives of morphisms in twisted arrow categories.

\begin{example}
For a set $C=\{c_0,c_1,c_2,c_3,c_4\}$ let $P$ be the $C$-coloured operad generated by operations $p\in P(c_1,c_2;c_0)$, $q_1,q_2\in P(c_3,c_4;c_2)$ and relation $p\circ_2 q_1 = p\circ_2 q_2$, and $A=\{a_0,a_1,a_2,b_2,a_3,a_4,b_4\}$ be the $P$-algebra such that $a_i$ and $b_i$ have colour $c_i$, and $p(a_1,a_2)=p(a_1,b_2)=a_0$, $q_1(a_3,a_4)=q_1(a_3,b_4)=a_2$, and $q_2(a_3,a_4)=q_2(a_3,b_4)=b_2$. Then in $\Tw_P(A)$ we have $[(p\circ_2 q_1)(a'_1,a_3,b_4)]=[p(a'_1,a_2)]=[(p\circ_2 q_1)(a'_1,a_3,a_4)]=[(p\circ_2 q_2)(a'_1,a_3,a_4)]=[p(a'_1,b_2)]$, i.e.\@ this morphism does not have canonical representative.
\end{example}

The following property at first may seem artificial. Its true nature is explained by Proposition~\ref{prp:canonical-decomp-initial-objects} and its corollary, the localization lemma.

\begin{definition}
Let $Q$ be a set-operad and $Q'$ be a subset of $Q$. The operad $Q$ is canonically decomposable via $Q'$ if for any operation $q$ in $Q$ of non-zero arity there exist: unique operation $q'$ in $Q'$, unique operations $q_l$ in $Q$, and unique indices $j_{li}$ satisfying $j_{l1}<\dots<j_{lk_l}$ for all $l$, such that $q$ is equal to $q'(id_{c_1}(1),q_0(j_{01},\dots,j_{0k_0}),\dots,q_m(j_{m1},\dots,j_{mk_m}))$, where $(m+2)$ is the arity and $c_1$ is the first input colour of $q'$, and the expression denotes the operadic composition of the operation $q'$ with operations $q_l$, permuted via indices $j_{li}$.
\end{definition}

\begin{example}
The operads $uCom$, $Com$, $uAs$ and $As$ of commutative monoids, commutative semigroups, monoids and semigroups respectively are canonically decomposable: $uCom$ via $\{\mu_2\}$,  $Com$ via $\{\mu_1,\mu_2\}$, $uAs$ via $\{\mu_3^{(12)}\}$, and $As$ via $\{\mu_1,\mu_2,\mu_2^{(12)},\mu_3^{(12)}\}$.
\end{example}

\begin{lemma}
\label{lem:candec}
Let $Q$ be an operad canonically decomposable via $Q'$, and $A$ be a $Q$-algebra. Then any morphism $f$ in $\Tw_Q(A)$ has unique representative of the form $q'(a',a_0,\dots,a_m)$ with operation $q'$ in $Q'$. Similarly, any morphism $f$ in $\mathcal{U}_Q(A)$ has unique representative of the form $q'(h_1, a_0,\dots,a_m)$ with $q'$ in $Q'$. Such representatives of morphisms will be called \emph{canonical representatives}.
\end{lemma}
\begin{proof}
For any representative $q(a',b_1,\dots,b_l)$ of a morphism $f$ in $\Tw_Q(A)$ take the canonical decomposition $q'(1,q_0(j_{01},\dots,j_{0k_0}),\dots,q_m(j_{m1},\dots,j_{mk_m}))$ of $q$ and let $a_l$ be $q_l(\overline{(b_l)})$. Then $q'(a',a_0,\dots,a_m)$ is a canonical representative of $f$. This gives a map from representatives of $f$ to canonical representatives of $f$. Any canonical representative is mapped to itself. If two representatives are equivalent via generating equivalence, then they are mapped to the same canonical representative. If two canonical representatives are equivalent, they are equivalent via a sequence of generating equivalences, and thus they are equal.
\end{proof}

Next we prove that some of the graph-substitution operads are canonically decomposable. 
 
\begin{figure}[b]
\centering
\begin{tikzpicture}[scale=0.8]
\foreach \place/\name/\label in {{(0,0)/a/$1$}, {(-2,2)/b/$2$}, {(0,2)/c/$3$}, {(2.5,2)/d/$4$}}
    \node[circ] (\name) at \place {\label};
\foreach \place/\name/\label in {{(0,2.5)/c1/$0$}, {(0.5,2.5)/c2/$2$}, {(2.5,2.5)/d1/$1$}}
    \node[leaf,label=\label] (\name) at \place {};
  \draw (a) to [out=150,in=150,looseness=1.5] (b);
  \draw (a) to [out=120, in=160,looseness=2] (c);
  \draw[use as bounding box] (a) to [out=90,in=60,looseness=1.2] (b);
  \draw (a) to [out=60, in=160,looseness=1.2] (d);
  \draw (a) to [out=30, in=120,looseness=5] (c);
  \draw (c) to (c1);
  \draw (c) to (c2);
  \draw (d) to (d1);
\end{tikzpicture}
\caption{A reduced modular graph. The upper vertices are indexed by $2$, $3$ and $4$, since they appear in that order as vertices adjacent to the first vertex. The order on edges of upper vertices is obtained from the order on edges of the first vertex and from the order on leaves.}
\label{fig:mop-decomposable}
\end{figure}
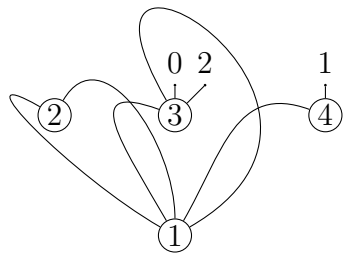

\begin{lemma}
\label{lem:graph-operads-canonically-decomposable}
The operads $pOp$, $sOp_C$, $cOp$, $uAs$, $iuAs$, $cuAs$, $ciuAs$, $iuAs_{Tr}$, $iuAs_{iTr}$, $mOp$, $mOp_{(g,n)}$, $mOp_{st}$, $mOp_{nc}$ and the operads encoding wheeled properads, wheeled operads, dioperads and $\frac{1}{2}$-props are canonically decomposable:
\begin{itemize}
\item $pOp$ is canonically decomposable via reduced planar trees, the planar rooted trees with $(m+2)$ vertices, where $m\geq 0$, with the first vertex having $m$ input edges and connected to $m$ upper vertices and to one root vertex, with the root vertex being indexed by $2$, and the upper vertices being indexed from $3$ to $(m+2)$ in planar order,
\item $sOp_C$ is canonically decomposable via reduced symmetric trees, the same trees as above, with the root vertex $v_2$ connected to the first vertex by the first edge of $v_2$, with the order on leaves such that for each vertex $v_j$ the indices of leaves of $v_j$ are ordered in planar order, see Figure~\ref{fig:3leveltrees},
\item $cOp$ is canonically decomposable via reduced cyclic trees, the planar trees of height 2, without the root leaf, with the root vertex indexed by $1$, with all $m$  edges of the first vertex being connected to the $m$ upper vertices, with upper vertices indexed in planar order, and with the order on leaves such that for each vertex $v_j$ the indices of leaves of $v_j$ are ordered in planar order,
\item $mOp$ is canonically decomposable via reduced modular graphs, the elements of $mOp$ without loop-edges and such that every vertex except the first one is adjacent only to the first vertex, with vertices ordered via the order on the edges of the first vertex, with  edges of each vertex except the first vertex ordered so that first come the edges adjacent to the first vertex, ordered via the order on the edges of the first vertex, and then come leaves, ordered via the global order on leaves.
\end{itemize}
\end{lemma}
\begin{proof}
In all cases except $mOp_{st}$ the proofs are similar: an operation $q$ of these operads is a graph endowed with additional structure; the graphs of operations $q_0,\dots,q_m$ are the maximal connected subgraphs of the graph of $q$ that do not contain the vertex indexed by $1$; the graph of the operation $q'$ is obtained from the graph of $q$ by contraction of the subgraphs  $q_0,\dots,q_m$; if a graph $q_l$ is the edge without vertices, then in $q'$ we replace the corresponding edge with the edge with one vertex; if necessary, additional conditions on orders on edges, leaves and neighbourhoods of vertices of $q'$ ensure that the operations $q'$ and $q_0,\dots,q_m$ are unique. See Figures~\ref{fig:mop-decomposable} and \ref{fig:3leveltrees}. The case of $mOp_{st}$ is slightly different: since $mOp_{st}$ does not contain $\mu_0$, the first vertex of the corresponding reduced graphs may be adjacent to leaves or to loops without vertices.
\end{proof}

\begin{figure}[t]
\begin{align*}
  \begin{aligned}\begin{tikzpicture}[optree,
      level distance=18mm,
      level 2/.style={sibling distance=50mm},
      level 3/.style={sibling distance=20mm}]
    \node{}
    child { node[circ,label=-3:$p_0$]{}
      child { node[label=above:$2$]{} }
      child { node[circ,label=-3:$p$]{}
        child { node[circ,label=183:$p_1$]{}
          child { node[label=above:$5$]{} }
          child { node[label=above:$3$]{} } }
        child { node[circ,label=-3:$p_2$]{} }
        child { node[circ,label=-3:$p_3$]{}
          child { node[label=above:$1$]{} } } }
      child { node[label=above:$4$]{}}};
  \end{tikzpicture}\end{aligned}
  =
  \begin{aligned}\begin{tikzpicture}[optree,
    level distance=18mm,
    level 2/.style={sibling distance=30mm},
    level 3/.style={sibling distance=20mm}]
  \node{}
  child { node[circ,label=-3:$p_0^{(12)}$]{}
    child { node[circ,label=183:$p$]{}
      child { node[circ,label=180:$p_1^{(12)}$]{}
        child { node[label=above:$3$]{} }
        child { node[label=above:$5$]{} } }
      child { node[circ,label=-3:$p_2$]{} }
      child { node[circ,label=-3:$p_3$]{}
        child { node[label=above:$1$]{} } } }
    child { node[label=above:$2$]{} }
    child { node[label=above:$4$]{}}};
\end{tikzpicture}\end{aligned}
\end{align*}
\caption{These trees represent the same morphism from $p$ in the twisted arrow category $\Tw_{sOp_C}(P)$ of an operad $P$, but only the right tree is the canonical representative. The input edges of the source vertex $p$ cannot be permuted. In the notation of Lemma~\ref{lem:candec} we have $a_0=p_0^{(12)}$, $a_1=p_1^{(12)}$, $a_2=p_2$, and $a_3=p_3$.}
\label{fig:3leveltrees}
\end{figure}
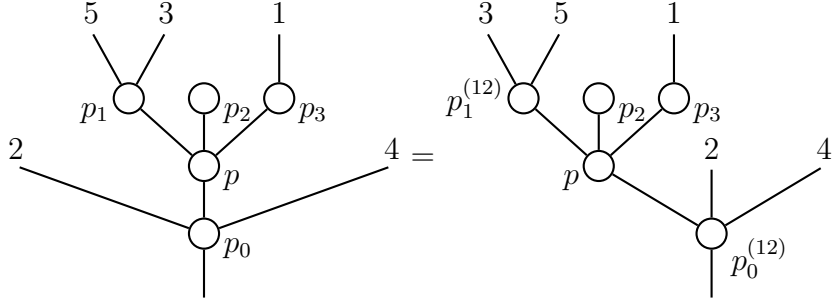

\begin{remark}
The operad of properads is not canonically decomposable, see Example~\ref{ex:properads-not-decomposable}.
\end{remark}

The first vertex in all reduced trees or graphs above will be called \emph{the source vertex}, and the remaining vertices will be called \emph{the non-source vertices}.

\begin{corollary}
\label{crl:twsop}
For a planar, symmetric, cyclic operad or an $mOp$-algebra $P$ the corresponding twisted arrow category $\Tw(P)$ has operations of $P$ as objects. A morphism from an operation $p$ is canonically represented by unique reduced planar, symmetric, cyclic tree, or modular graph  with vertices marked by operations of $P$, with the source vertex marked by $p$. The target of a morphism is equal to the evaluation of the corresponding tree or graph. A composition of morphisms $h:p\to r$ and $h':r\to t$ is obtained by grafting the non-source vertices of $h'$ to corresponding leaves of $h$, and then contracting the edges not adjacent to $p$ via composition in $P$.
The case of the enveloping category of $P$ is similar, except the source vertices of morphisms are not marked by operations of $P$.
\end{corollary}

\subsection{Main examples}
\label{sec:examples}

Several categories fundamental to homotopy theory are equivalent to twisted arrow categories of operads or cyclic operads. This includes the simplex category $\Delta$, Connes' cyclic category $\Lambda$, Segal's category $\Gamma$, the opposite of the category of finite sets, and, as we show later, Moerdijk--Weiss dendroidal category~$\Omega$. 

\begin{lemma}
For any operad $Q$ and terminal $Q$-algebra $A$ the twisted arrow category $\Tw_Q(A)$ is isomorphic to the enveloping category $\mathcal{U}_Q(A)$.
\end{lemma}
\begin{proof}
The fibres of the canonical opfibration $\Tw_Q(A)\to\mathcal{U}_Q(A)$ are singletons.
\end{proof}

\begin{proposition}
\label{pr:asdelta}
The twisted arrow category $\Tw_{pOp}(uAs_{pl})$ and the enveloping category $\mathcal{U}_{pOp}(uAs_{pl})$ of the terminal planar operad $uAs_{pl}$ are isomorphic to the category $\Delta$.
\end{proposition}
\begin{proof}
In the planar operad $uAs_{pl}$ there is exactly one operation $\mu_n$ for each arity $n\geq 0$. A reduced planar tree with lower vertex $\mu_{m_0}$ connected by the $i$-th edge to the source vertex $\mu_n$, and with upper vertices $\mu_{m_j},\,1\leq j \leq n$, is the canonical representative of a morphism from $\mu_n$ to $\mu_{m_0+\dots+m_n-1}$. 
The categories $\Tw_{pOp}(uAs_{pl})$ and $\Delta$ are isomorphic: the operation $\mu_n$ corresponds to $[n]$, and the morphism above corresponds to the map $f$ with $f(0)=i-1$, $f(j)-f(j-1)=m_j$. This correspondence is bijective. The set-map $f$ is equal to the following map from the indices of leaves of $\mu_n$ to the indices of leaves of the target: the $j$-th leaf is mapped to the rightmost, i.e.\@ to the next in the counter-clockwise direction, leaf of the tree that was grafted into the $j$-th leaf, including the $0$-th (root) leaf. Functoriality follows easily from this description.
\end{proof}

\begin{proposition}
\label{pr:comgamma}
Let $uCom$ be the terminal symmetric operad, the operad of  commutative monoids. The categories $\Tw_{sOp}(uCom)$ and $\mathcal{U}_{sOp}(uCom)$ are isomorphic to Segal's category $\Gamma=sk(FinSet_*^{op})$, the skeleton of the opposite category of the category of finite pointed sets.
\end{proposition}
\begin{proof}
Again, there is exactly one operation $\mu_n\in uCom(n)$ for each arity $n\geq 0$, and a reduced symmetric tree with the root vertex $\mu_{m_0}$ connected by the $1$-st edge to the source vertex $\mu_n$, and with upper vertices $\mu_{m_j}$, $1\leq j \leq n$, is the canonical representative of a morphism $\mu_n\to\mu_{m_0+\dots+m_n-1}$. Under the isomorphism of $\Tw_{sOp}(uCom)$ and $\Gamma$ the object $\mu_n$ corresponds to the set $[n]$ with marked point $0$, the tree above corresponds to the opposite of the map of sets $f:[m_0+\dots+m_n-1]\to [n]$, and this map is described by its fibres: $f^{-1}(j)$ is the set of indices of leaves adjacent to the $j$-th non-source vertex, with the root leaf and the root vertex indexed by $0$. This correspondence between trees and maps is well-defined, bijective and functorial.
\end{proof}

\begin{proposition}
Let $uCom_c$ be the terminal cyclic operad, the cyclic operad of cyclic commutative monoids. The categories $\Tw_{cOp}(uCom_c)$ and $\mathcal{U}_{cOp}(uCom_c)$ are the skeletons of  the opposite category of the category of non-empty finite sets. Let $uCom'_c$ be the terminal $cOp'$-algebra. The categories $\Tw_{cOp'}(uCom'_c)$ and $\mathcal{U}_{cOp'}(uCom'_c)$ are the skeletons of the opposite category of the category of finite sets.
\end{proposition}
\begin{proof}
As in the symmetric case, the set of indices of leaves of the $j$-th vertex is the fibre $f^{-1}(j)$ of the corresponding map of sets.
\end{proof}

Recall that the inclusion of operads $pOp\to sOp$ induces the extension functor $Sym$, called symmetrization, from planar operads to symmetric operads. Thus for any planar operad $P$ the inclusion $pOp\to sOp$ induces the functors $\Tw_{pOp}(P)\to\Tw_{sOp}(Sym(P))$ and $\mathcal{U}_{pOp}(P)\to\mathcal{U}_{sOp}(Sym(P))$.

\begin{proposition}
\label{pr:symmetrizationeq}
For any planar operad $P$ the functor $\Tw_{pOp}(P)\to\Tw_{sOp}(Sym(P))$ is an equivalence of categories. In particular, the category $\Tw_{sOp}(uAs)$, i.e.\@ the category $\Tw_{sOp}(Sym(uAs_{pl}))$, is equivalent to the category $\Tw_{pOp}(uAs_{pl})=\Delta$.
\end{proposition}
  \begin{proof} 
    An object in $\Tw_{sOp}(Sym(P))$ is an operation of the form $p^\sigma ,\, p\in P$, and $p^\sigma $ is isomorphic to $p$ via the morphism with trivial non-source vertices and with leaves permuted by $\sigma^{-1}$. The full subcategory $\mathcal{A}$ of $\Tw_{sOp}(Sym(P))$ on objects $p\in P$, i.e.\@ on images of objects of the category $\Tw_{pOp}(P)$, is equivalent to the category $\Tw_{sOp}(Sym(P))$ itself. 
    
    Any morphism $f$ in $\mathcal{A}$ is represented by a reduced symmetric tree $t'$ with vertices marked by elements of $Sym(P)$, with the source vertex marked by an element of $P$. Since the action of symmetric groups in $Sym(P)$ is free, there is exactly one symmetric tree $t$ of height $3$ (that is not assumed to be reduced) that represents $f$ and whose non-source vertices are marked by operations $p_i$ from $P$, not merely from $Sym(P)$ (compare with Figure~\ref{fig:3leveltrees}). The tree $t$ is obtained from $t'$ by permutation of edges of non-source vertices. Denote by $\sigma$ the permutation of the leaves of $t$. The target of $f$ is equal to $q^{\sigma}$, where $q$ is the evaluation of the planar tree $t$, i.e.\@ of $t$ endowed with trivial permutation on leaves. Both $q$ and $q^\sigma$ are in $P$. Since the action of symmetric groups on $P$ is free, the permutation $\sigma$ is trivial, and $t$ is reduced. Any morphism in $\mathcal{A}$ is represented by unique reduced planar tree. The essentially surjective functor $\Tw_{pOp}(P)\to\mathcal{A}$ is fully faithful.
  \end{proof}

\begin{proposition}
Let $uAs_c$ be the cyclic associative operad. Then $\Tw_{cOp}(uAs_c)$ is equivalent to Connes' cyclic category $\Lambda$.
\end{proposition}
\begin{proof}
The proof is similar to that of Proposition~\ref{pr:symmetrizationeq}. Any object $\mu_n^\sigma$ is isomorphic to $\mu_n$ via permutation isomorphism. The leaves of a tree that represents a morphism in the corresponding subcategory, and whose vertices are marked by operations $\mu_{n_j}$, are always cyclically ordered. To get an isomorphism with $\Lambda$ one can proceed as in Proposition~\ref{pr:asdelta}: then $\mu_n\to\mu_m$ corresponds to morphisms $[n]\to[m]$; or as in Proposition~\ref{pr:comgamma}: then $\mu_n\to\mu_m$ corresponds to $([m]\to[n])^{op}$.
\end{proof}

\subsection{Additional examples}

The functors that appeared in the works of Pirashvili and Richter (\cite{pirashvili25hochschild}) and Loday (\cite{Loday1998cyclic}) that connect Hochschild and cyclic homology theory with functor homology are equivalent to canonical opfibrations $\Tw(P)\to\mathcal{U}(P)$. We describe a similar functor related to 2-dimensional cobordisms.

\begin{proposition}
The category $\mathcal{U}_{cOp}(uAs_c)$ is the opposite of the category $\mathcal{F}(as)$ of non-commutative sets (\cite[A10]{feigin1987additive}, see also \cite{fiedorowicz1991crossed,pirashvili25hochschild}), and the category $\mathcal{U}_{sOp}(uAs)$ is the opposite of the subcategory $\Gamma(as)\subset \mathcal{F}(as)$ of maps preserving $0$.
\end{proposition}
\begin{proof}
The statement for $\Gamma(as)$ appears in \cite{fresse2014functor}, and is easy to prove directly.
\end{proof}

\begin{remark}
It is not clear whether the theorem of Pirashvili and Richter admits a generalization. This theorem has a simple proof based on existence of a certain strict factorization system on the enveloping categories of the operads $uAs$ and $uAs_c$ (\cite{slominska2003decompositions,zimmermann2004changement}).  Such strict factorization system does not exist 
for the symmetric operad $pOp$, or $sOp$, or any other operad $P$ such that there exist non-isomorphic objects of the same colour in $\Tw(P)$.
\end{remark}

The following commutative diagram is (transposed and opposite to) the diagram from \cite[section 1.4]{pirashvili25hochschild}. The right square is also described in that section. 

\begin{center}
\begin{tikzcd}
\Delta \arrow[r,"\hat{C}^{op}"] \arrow[d] & \Gamma(as)^{op} \arrow[d] \arrow[r] & sk(FinSet_*^{op}) \arrow[d] \\
\Lambda \arrow[r]          & \mathcal{F}(as)^{op} \arrow[r]         & sk(FinSet^{op}) 
\end{tikzcd}
\end{center}
The arrows in this diagram are the following functors. The two horizontal arrows on the left are equivalent to opfibrations $\Tw_{sOp}(uAs) \to \mathcal{U}_{sOp}(uAs)$ and $\Tw_{cOp}(uAs_c)\to\mathcal{U}_{cOp}(uAs_c)$. More precisely, the upper left arrow is the composition of the equivalence $\Tw_{pOp}(uAs)\to\Tw_{sOp}(Sym(uAs))$ with the canonical opfibration. The lower left arrow is also a similar composition, with $\Lambda$ being the twisted arrow category of the terminal ``planar cyclic'' operad. 

The inclusion of operads $sOp\to cOp$ induces the restriction functor. The restriction of $cOp$-algebra $uAs_c$ is $uAs$. This restriction induces the left and the middle vertical arrows. The left square is commutative due to functoriality of canonical opfibrations. 

The restriction of $uCom_c$ is $uCom$, and it induces the right vertical arrow. The restriction of the morphism $uAs_c\to uCom_c$ is the morphism $uAs\to uCom$, and it induces the right commutative square.
 
The outer square is also equivalent to the square consisting of twisted arrow categories of $uAs$, $uAs_c$, $uCom$ and $uCom_c$, with functors induced by the morphisms of operads, of cyclic operads, and by restriction. 

Notice also that the inclusion $\Delta\to\mathcal{F}(as)$ in \cite[Lemma 1.1]{pirashvili25hochschild} is equivalent to the composition $\Tw_{pOp}(uAs_{pl})\to\Tw_{cOp}(uAs_c)\to\Tw_{cOp}(uAs_c)^{op}\to\mathcal{U}_{cOp}(uAs_c)^{op}$, i.e.\@ the inclusion $\Delta\to\mathcal{F}(as)$ implicitly factors through the functor $\Lambda\to\Lambda^{op}$.

\paragraph{Cobordisms.} A similar diagram exists for the modular envelopes of cyclic operads $uCom_c$ and $uAs_c$. Modular envelope is the extension functor induced by the inclusion $cOp\to mOp$ or by the inclusion $cOp\to mOp_{(g,n)}$.

\begin{proposition}
Let $mOp$-algebra $uCom_m$ be the modular envelope of the terminal cyclic operad $uCom_c$. Let $mOp_{(g,n)}$-algebra $uCom'_m$ be the terminal $mOp_{(g,n)}$-algebra, which is also the modular envelope of $uCom_c$. The category $\mathcal{U}_{mOp}(uCom_m)$ is a subcategory of the category $Cob$ of orientable 2d-cobordisms. The categories $\Tw_{mOp}(uCom_m)$, $\Tw_{mOp_{(g,n)}}(uCom'_m)$, and $\mathcal{U}_{mOp_{(g,n)}}(uCom'_m)$ are all isomorphic to the full subcategory of $\emptyset/Cob$ on connected cobordisms from the empty set. Let $TmOp$ be the terminal $mOp$-algebra. The category $\Tw_{mOp}(TmOp)=\mathcal{U}_{mOp}(TmOp)$ is the subcategory of $\mathcal{U}_{mOp}(uCom_m)$ of cobordisms of total genus $0$.
\end{proposition}
\begin{proof}
The $mOp$-envelope of $uCom_c$ is the $mOp$-algebra freely generated by the elements of $uCom_c$ and factored by relations in $uCom_c$, i.e.\@ the elements of $uCom_m$ are equivalence classes of operadic graphs, with equivalence generated by contractions of subtrees and by permutations of edges adjacent to any single vertex. The map that computes the genus of an operadic graph induces isomorphisms between the sets $uCom_m(n)$ and $\mathbb{N}$.

Let $Cob$ be the category whose objects are finite sequences of circles and morphisms are orientable cobordisms between disjoint unions of circles. 
 The arity of objects of $\mathcal{U}_{mOp}(uCom_m)$ encodes the number of circles in objects of $Cob$, and the $j$-th non-source vertex of a morphism encodes the $j$-th connected component of the cobordism: the number of boundary components and the indices of the target and the source circles, and the genus of the surface attached to the $j$-th circle.  The subcategory $\mathcal{U}_{mOp}(uCom_m)$ of $Cob$ contains all morphisms of $Cob$ except the non-trivial morphisms from $\emptyset$.
In the remaining cases an object in $\emptyset/Cob$ corresponds to the choice of arity and genus. The rest of the proof is trivial.
\end{proof}

The operation $\beta$, the circle with two vertices, corresponds to cobordisms from (surfaces with) two circles to the empty set. This suggests that while the operad $mOp_{nc}$ (or $mOp_{st}$) is more common, the operads $mOp$ and $mOp_{(g,n)}$ also deserve attention.

The twisted arrow categories of modular envelopes of $uAs_c$ are related to ribbon graphs. It seems that for these categories there is no decomposition analogous to that of \cite{slominska2003decompositions,zimmermann2004changement}. The categories $\Tw_{mOp}(uCom_m)$ and $\Tw_{mOp}(uAs_m)$ are generalized Reedy, with the surface of genus $g$ with $n$ circles having degree $(2g+n)$. The category $Cob$ is not generalized Reedy. 

\subsection{Structure of twisted arrow categories of operads}
\label{sec:structure}
We describe the structure on twisted arrow categories and enveloping categories of symmetric and generalized operads that is used in the rest of this work. 

From now on we assume that operads are endowed with a nice grading, see Definition~\ref{def:graded-operad}. We start with a few lemmas. 

\begin{lemma}
\label{lem:splitepi}
Let $P$ be a planar, symmetric, cyclic operad or an $mOp$-algebra, and morphisms $e$ and $s$ in $\Tw(P)$ or in $\mathcal{U}(P)$ be such that $e\circ s = id$. Then the graphs of $e$ and $s$ are trees, non-source vertices of $s$ have arity at least $1$ (and thus grading at least $0$), and non-source vertices of $e$ have arity at most $1$ and grading at most $0$.
\end{lemma}
\begin{proof}
In the case of $mOp$, a composition with a non-tree produces a non-tree, thus the graphs of $e$ and $s$ are trees.
  In general, if $s$ would have non-source vertices of arity $0$, then $e\circ s$ would have the same non-source vertices of arity $0$. If $e$ would have non-source vertices of arity greater than $1$, then $e\circ s$ would have non-source vertices of arity greater than $1$. Finally, let $p$ be any non-source vertex of $s$, of arity $n$. Then composition with $e$ grafts $(n-1)$ operations $q_j$ of arity $0$ to leaves of $p$ and one operation $q'$ of arity $1$ to a leaf or to the root of $p$. Since the result of this composition is an identity operation, the operations $q_j$ are units of binary operations, and, by the niceness of grading, have grading $(-1)$. Let $p'$ be the composition of $p$ with all the operations $q_j$. The composition of $p'$ with $q'$ is an identity operation, which has grading $0$, and since both $p'$ and $q'$ have arity $1$, and thus grading not less than $0$, they have grading $0$.
\end{proof}

\begin{corollary}
\label{crl:isos-in-tw} 
For any planar, symmetric, cyclic operad or $mOp$-algebra $P$ the isomorphisms in $\Tw(P)$ and in $\mathcal{U}(P)$ are precisely the morphisms $f$ such that all non-source vertices of $f$ have arity $1$, grading $0$ and are marked by isomorphisms in the category $P(1)$. Isomorphic objects of $\Tw(P)$ have the same arity and grading, and isomorphic objects of $\mathcal{U}(P)$ have the same arity.
\end{corollary}
\begin{proof}
A morphism $f$ that satisfies the properties above has a left inverse $g$: the non-source vertices of $g$ are the inverses of the non-source vertices of $f$, and the permutation on leaves of $g$ is the inverse of the permutation on leaves of $f$. The morphism $g$ itself satisfies the properties above, thus $g$ has a left inverse, and $f$ is an isomorphism.
In the opposite direction, the preceding lemma implies that non-source vertices of isomorphisms have arity $1$ and grading $0$. If $f$ is an isomorphism, the non-source vertices of $f^{-1}$ are the inverses in $P(1)$ of the non-source vertices of $f$.
\end{proof}

\begin{definition}
  For a planar, symmetric or cyclic operad $P$ a permutation isomorphism in $\Tw(P)$ or in $\mathcal{U}(P)$ is a morphism with all non-source vertices marked by identity operations.
\end{definition}

\begin{proposition}
A natural equivalence of operads induces natural equivalence between the corresponding twisted arrow categories, and also induces natural equivalence between the corresponding enveloping categories.
\end{proposition}
\begin{proof}
A natural equivalence is a map $f:Q\to Q'$ of coloured operads such that the map of underlying categories $f(1):Q(1)\to Q'(1)$ is essentially surjective and the maps $f:Q(c_1,\dots,c_n;c_0)\to Q'(f(c_1),\dots,f(c_n);f(c_0))$ of sets are bijective for any choice of colours $c_j$. Essential surjectivity implies that any object $q$ of $\Tw(Q')$ or of $\mathcal{U}(Q')$ is isomorphic to an object $q'$ with all colours belonging to the image of $f$. Bijectivity implies that the object $q'$ is in the image of $\Tw(f)$, and thus implies essential surjectivity of $f_*$. Bijectivity on operations also implies bijectivity on reduced symmetric trees with fixed set of colours of inner and leaf edges, and thus implies that $f_*$ is fully faithful.
\end{proof}

\paragraph{Ternary strict factorization system.}

Twisted arrow categories and universal enveloping categories of operads are endowed with two compatible strict factorization systems. The first factorization system often generates generalized Reedy category structure, with precise criterion given by Theorem~\ref{thm:groupoid-reedy}. The second factorization system exists only for planar and symmetric operads, is fundamental to Segal conditions, and should be seen as a refinement of $(Active, Inert)$ orthogonal factorization system.

\begin{definition} Let $P$ be a planar, symmetric or cyclic operad endowed with a nice grading $\psi$. The class of morphisms either in $\Tw(P)$ or in $\mathcal{U}(P)$ denoted by $R_{\psi=-1}$ consists of morphisms such that the permutation on leaves of their canonical representatives is trivial, and the non-source vertices are marked either by identity operations or by operations of grading $(-1)$. The class of morphisms  either in $\Tw(P)$ or in $\mathcal{U}(P)$ denoted by $R_{\psi\geq 0}$ consists of morphisms such that all non-source vertices of their canonical representatives have grading at least $0$.
\end{definition}

\begin{proposition}
For any planar, symmetric or cyclic operad $P$ the pair of subcategories $(R_{\psi=-1},R_{\psi\geq 0})$ of $\Tw(P)$ or $\mathcal{U}(P)$ is a strict factorization system.
\end{proposition}  
\begin{proof}
Both $R_{\psi=-1}$ and $R_{\psi\geq 0}$ are closed under composition. For a factorization $r\circ l$ of a morphism $f$ the non-trivial vertices of $l$ are the vertices of $f$ of grading $(-1)$, and the vertices of $r$ are the vertices of $f$ of grading at least $0$. These vertices are determined uniquely by $f$.
\end{proof}

\begin{definition}
Let $P$ be a planar or symmetric operad. The class $Upper$ of \emph{upper} morphisms in $\Tw(P)$ or in $\mathcal{U}(P)$ consists of morphisms such that the lowest vertex in corresponding canonical representatives is marked by an identity operation. The class $Lower$ of \emph{lower} morphisms in $\Tw(P)$ or in $\mathcal{U}(P)$ consists of morphisms such that the upper vertices of canonical representatives are marked by identity operations, and the indices of all leaves of \emph{upper} vertices, considered together, increase in planar order (in general not sequentially).
\end{definition}

\begin{definition}
  Let $P$ be a planar or symmetric operad. An input map of an operation $p$ in $P$ is a \emph{lower} morphism from an identity operation to $p$. The output map $out$ of $p$ is the  unique \emph{upper} morphism from an identity operation to $p$. An operation of arity $n$ has exactly $n$ input maps. The $j$-th input map $in_j$ is the input map such that the leaf above the source vertex is indexed by $j$.
\end{definition}

\begin{proposition}
For any planar or symmetric operad $P$ the pair $(Upper, Lower)$ is a strict factorization system.
\end{proposition}
\begin{proof}
Both $Upper$ and $Lower$ are closed under composition. For a factorization $l\circ u$ of a morphism $f$ the upper vertices of $u$ are the upper vertices of $f$, and the lowest vertex of $l$ is the lowest vertex of $f$. The permutations on leaves of morphisms $l$ and $r$ are determined uniquely by the permutation on leaves of $f$.
\end{proof}

\begin{definition}
  The $(Active, Inert)$ orthogonal factorization system on twisted arrow category or enveloping category of an operad is the orthogonal factorization system generated by the corresponding $(Upper, Lower)$ strict factorization system.
\end{definition}

\begin{remark}
If $F:\mathcal{A}\to\mathcal{B}$ is a discrete (op)fibration of categories and $(L,R)$ is a strict factorization system on $\mathcal{B}$, then $(F^{-1}(L),F^{-1}(R))$ is a strict factorization system on $\mathcal{A}$. The strict factorization systems on $\Tw(P)$ described above arise in this manner from the strict factorization systems on $\mathcal{U}(P)$ via the canonical discrete opfibration $\Tw(P)\to\mathcal{U}(P)$.
\end{remark}

\begin{proposition}
For any planar or symmetric operad $P$ the strict factorization systems $(R_{\psi=-1},R_{\psi\geq 0})$ and $(Upper,Lower)$ form a ternary strict factorization system. When grading $\psi$ is different from the arity grading, this ternary factorization system is a part of quaternary strict factorization system.
\end{proposition}
\begin{proof}
Let $\phi$ be the arity grading.
We have $R_{\psi=-1}\subseteq R_{\phi=-1}\subseteq Upper$ and $Lower\subseteq R_{\phi\geq 0}\subseteq R_{\psi\geq 0}$. 
\end{proof}

For any ternary strict factorization system $((L_1,R_1),(L_2,R_2))$ it is natural to consider the class of morphisms $R_2\circ L_1$. These are the morphisms in which the middle term in their ternary factorization is an identity morphism.  

\begin{proposition}
For any planar or symmetric set-operad $P$ the class of morphisms $Lower\circ R_{\psi=-1}$ in $\Tw(P)$ or in $\mathcal{U}(P)$ is a category. 
\end{proposition}
\begin{proof}
This follows easily from the fact that the class $Lower\circ R_{\psi=-1}$ consists precisely of the morphisms such that the indices of upper leaves increase in planar order, the upper vertices of non-zero arity are marked by trivial operations, and the upper vertices of arity $0$ are marked by operations of grading $(-1)$.
\end{proof}

\paragraph{Lower-Upper factorizations.} We have shown that $(Upper, Lower)$ is a strict factorization system. For twisted arrow categories \emph{of categories} the pair $(Lower, Upper)$ is also a strict factorization system. For operads this factorization is unique up to contractible choice. 

\begin{definition}
For any operad $P$ the subcategory $Lower'$ of $\Tw(P)$ or $\mathcal{U}(P)$ is the subcategory of lower morphisms with trivial permutation on leaves.
\end{definition}

\begin{definition}
  Let $P$ be a symmetric operad and $f$ be a morphism in $\Tw(P)$ or in $\mathcal{U}(P)$. The canonical $(Lower', Upper)$-factorization of $f$ is the factorization $u\circ l$, where the lowest vertex of $l$ is the lowest vertex of $f$, the permutation of leaves of $l$ is trivial, the upper vertices of $u$ are the upper vertices of $f$ followed by trivial vertices, and the permutation on leaves of $u$ is the same as the permutation on leaves of $f$. 
\end{definition}

\begin{lemma}
\label{lem:lower-upper-factor}
Let $P$ be an operad and $f$ be a morphism in $\Tw(P)$ or in $\mathcal{U}(P)$. The canonical factorization of $f$ is a terminal object in the category of $(Lower,  Upper)$-factorizations of $f$. The morphisms from $(Lower, Upper)$-factorizations of $f$ to the canonical factorization of $f$ are upper morphisms.
\end{lemma}
\begin{proof}
Let $f$ be a morphism and $i$ be the permutation isomorphism inverse to the permutation on leaves of $f$. The permutation on leaves of $i\circ f$ is trivial. Composition with $i$ gives an isomorphism between the categories of $(Lower,  Upper)$-factorizations of $f$ and of $i\circ f$, and sends the canonical factorization to the canonical factorization. Assume then that the permutation on leaves of $f$ is trivial.
Any $(Lower, Upper)$-factorization of $f$ is isomorphic via permutation isomorphism to a $(Lower', Upper)$-factorization. Let $u'\circ l'$ be such a factorization of $f$.

Any \emph{lower} morphism $l$ from the factorization $u'\circ l'$ to the canonical factorization $u''\circ l''$ of $f$ is trivial. Indeed, let $p$ be the source of $f$, an operation of arity $n$. The morphisms $l''=l\circ l'$ and $l'$ preserve the indices of the first $n$ leaves, thus $l$  preserves the indices of the first $n$ leaves. The non-source vertices of $u''$ that come after the first $n$ upper vertices are trivial, thus the lower vertex of $l$ coincides with the lower vertex of $u''\circ l$, and since $u''\circ l=u'$, this vertex is trivial, and $l$ is trivial.

Any morphism of $(Lower, Upper)$-factorizations can be factored uniquely as an upper morphism followed by a lower morphism, and these morphisms are themselves morphisms of $(Lower, Upper)$-factorizations, see Lemma~\ref{lem:fact-cat-inherit}. Since any lower morphism with target $u''\circ l''$ is trivial, any morphism of $(Lower, Upper)$-factorizations with target $u''\circ l''$ is an upper morphism. 

Next we construct a morphism $u$ from $u'\circ l'$ to $u''\circ l''$. Notice that the first $n$ upper vertices of $u'$ coincide with the first $n$ upper vertices of $f$, and of $u''$, and the permutation on the leaves above the first $n$ vertices of $u'$ is trivial. Since the remaining vertices of $u''$ are trivial, the morphism $u'$ is equal to $u''\circ u$, where $u$ is the upper morphism with first $n$ vertices trivial, with remaining vertices the same as in $u'$, with trivial permutation on the first $n$ leaves. The morphism $u\circ l'$ is a lower morphism with trivial permutation on leaves, and $u''\circ u\circ l'=f$, thus $u\circ l'=l''$, and $u$ is a morphism from $u'\circ l'$ to $u''\circ l''$. 

For the proof of uniqueness of the morphism $u$ from $u'\circ l'$ to $u''\circ l''$, notice that since $u\circ l'$ must be a lower morphism with trivial permutation on leaves, the first $n$ vertices of $u$ must be trivial, and $u$ must have trivial permutation on the first $n$ leaves. The condition $u'=u''\circ u$ determines the morphism $u$ uniquely.
\end{proof}

\begin{remark}
Categories of $(Lower, Upper)$-factorizations may have more than one terminal object. If one is to attempt to generalize \cite{bergner20182}, uniqueness in the definition of \emph{stable} double category should be replaced by essential uniqueness.
\end{remark}

\paragraph{Upper-Lower pushout.}

Twisted arrow categories of operads have pushouts of special form.
\begin{lemma}
\label{lem:upper-lower-pushouts}
Let $P$ be a symmetric or planar set-operad. Let $f$ be an upper and $g$ be a lower map that have a common source, in $\Tw(P)$ or in $\mathcal{U}(P)$. The pushout $D$ of $f$ and $g$ always exists, and the pushout maps $f'$ and $g'$ can be chosen so that $f'$ is an upper and $g'$ is a lower map, and $f'\circ g$ is a terminal $(Lower, Upper)$-factorization.
\begin{center}
\begin{tikzcd}
A \arrow[r, "f"] \arrow[d, "g"] & B \arrow[d, "g'"] \\
C \arrow[r, "f'"]               & D                
\end{tikzcd}
\end{center}
\end{lemma}
\begin{proof}
The planar case follows from the symmetric case by Proposition~\ref{pr:symmetrizationeq}. Let $P$ be a symmetric operad. By applying a permutation isomorphism to $C$ we can assume that $g$ has trivial permutation on leaves. Let $n$ be the arity of $A$, $(n+k)$ be the arity of $C$, and $m$ be the arity of $B$. Let $g'$ be the lower map with trivial permutation on leaves and with lower vertex the same as in $g$. Let $f'$ be the upper map with the first $n$ upper vertices the same as the upper vertices of $f$, with indices of leaves above these vertices the same as in $f$, and with the remaining $k$ upper vertices being the identity operations, and the remaining $k$ leaves ordered trivially. We have $g'\circ f=f'\circ g$. The factorization $f'\circ g$ is the canonical $(Lower', Upper)$-factorization. By applying essentially the same permutation isomorphism to $B$ and to $D$ we can further assume that $f$, $f'$ and $g'$ have trivial permutation on leaves.

To show that $D$ is a pushout, let $h_B:B\to Z$ and $h_C:C\to Z$ be maps such that $h_B\circ f=h_C\circ g$. Then (1) the first $n$ upper vertices of $h_C$ are equal to the composition of the $n$ upper vertices of $f$ with the $m$ upper vertices of $h_B$; (2) the lower vertex of $h_B$ is equal to the composition of the lower vertex of $g$ with the lower and with the last $k$ upper  vertices of $h_C$. By applying permutation isomorphism to $Z$, we can assume that $h_C$ has trivial permutation on leaves, which implies that $h_B$ has trivial permutation on leaves. The map $h_D:D\to Z$ such that $h_D\circ g'=h_B$ and $h_D\circ f'=h_C$ is determined uniquely: since $g'$ is a lower map with trivial permutation on leaves, the $m$ upper vertices of $h_B$ coincide with the first $m$ upper vertices of $h_D$; since $f'$ is an upper map with the last $k$ upper vertices trivial, the remaining $k$ upper vertices and the lower vertex of $h_D$ coincide with the last $k$ upper vertices and with the lower vertex of $h_C$; finally, the permutation of leaves of $h_D$ must be trivial. This determines $h_D$ uniquely, and, using properties (1) and (2), we have $h_D\circ g'=h_B$ and $h_D\circ f'=h_C$.
\end{proof}

\begin{corollary}
Let $P$ be a planar or symmetric set-operad, $f:A\to B$ be an active map  and $g:A\to C$ be an inert map in $\Tw(P)$ or in $\mathcal{U}(P)$. The pushout $D$ of $f$ and $g$ always exists, the map $f':C\to D$ is active, and the map $g':B\to D$ is inert.
\end{corollary}

\subsection{Generalized Reedy structure}
\label{sec:reedy}

For operads that satisfy a simple criterion the corresponding twisted arrow categories, and often their enveloping categories, are generalized Reedy.

\begin{definition}
A pair of subcategories $(R_-, R_+)$ of a category $C$ together with a map $deg: Ob(C)\to \alpha$ to an ordinal $\alpha$, called the degree map, is a generalized Reedy structure  (\cite{berger2011extension}) for the category $C$ if:
\begin{enumerate}
  \item $(R_-, R_+)$ is an orthogonal factorization system for $C$,
  \item Non-isomorphisms in $R_-$ decrease degree, 
  \item Non-isomorphisms in $R_+$ increase degree,
  \item Isomorphisms preserve degree,
  \item For any morphism $f\in R_-$ and any isomorphism $\theta$, if $\theta\circ f = f$, then $\theta = id$.
\end{enumerate}
This structure is called \emph{dualizable} if for any morphism $f\in R_+$ and any isomorphism $\theta$, if $f\circ\theta = f$, then $\theta=id$.
\end{definition}

For any operad $P$ let $I$ be the class of all isomorphisms in $\Tw(P)$ or $\mathcal{U}(P)$. Denote by $R_-$ the class $I\circ R_{\psi=-1}$ and by $R_+$ the class $R_{\psi\geq 0}\circ I=R_{\psi\geq 0}$. By Proposition~\ref{pr:ufs-are-ofs} the pair $(R_-, R_+)$ is an orthogonal factorization system.

\begin{theorem}
\label{thm:groupoid-reedy}
For a planar, symmetric or cyclic operad $P$ with a nice grading $\psi$ let the degree map on $\Tw(P)$ be equal to $\psi$. With this choice of the degree map, the pair $(R_-, R_+)$ is a generalized Reedy structure for $\Tw(P)$ if and only if all operations of $P$ of grading $0$ have arity $1$ and the category $P(\psi=0)$ of these operations is a groupoid. 
\end{theorem}
\begin{proof}
Here we implicitly use Corollary~\ref{crl:isos-in-tw}. Suppose that $(R_-,R_+)$ is a generalized Reedy structure. For any operation $p$ of degree $0$ the output map of $p$ is in $R_+$ and preserves degree, i.e.\@ it is an isomorphism, and thus the operation $p$ has arity $1$ and is invertible as the morphism in the category $P(1)$.

In the opposite direction, the non-source vertices of degree preserving morphisms in $R_+$ are marked by elements of $P(\psi=0)$, thus these morphisms are isomorphisms. The same holds for $R_-$. For the last property, let $f\in R_-$ and $\theta\in I$ be such that $\theta\circ f = f$. The permutation of leaves of $\theta$ must be trivial. For any non-source vertex $a_i$ of $f$ of arity $1$ and corresponding non-source vertex $b_j$ of $\theta$ we have $b_j\circ a_i=a_i$ (or $a_0\circ b_0=a_0$ for the lower vertex), and, since $P(\psi=0)$ is a groupoid, $b_j=id$. Thus $\theta=id$.
\end{proof}

\begin{example}
Twisted arrow categories of graph-substitution operads are generalized Reedy. Notice that for operads $mOp$, $ciuAs$ and similar operads we have to use the previously described grading $\psi$ (see Definition~\ref{def:graded-operad}), and not the arity grading.
\end{example}

In general the analogue of Theorem~\ref{thm:groupoid-reedy}  does not hold for enveloping categories: it is not clear how to endow these categories with the degree map. Still, we have the following theorem.

\begin{theorem}
\label{thm:enveloping-image-reedy}
Let $P$ be a planar, symmetric or cyclic operad with a nice grading $\psi$ such that operations with the same multiset of input and output colours have the same grading. Then the image $\mathcal{V}(P)$ of the canonical discrete opfibration $\Tw(P)\to\mathcal{U}(P)$ can be endowed with the degree map. The orthogonal factorization system $(R_-,R_+)$ on $\mathcal{U}(P)$ restricts to $\mathcal{V}(P)$. This gives a generalized Reedy structure for $\mathcal{V}(P)$ if and only if all operations of $P$ of grading $0$ have arity $1$ and the category $P(\psi=0)$ of these operations is a groupoid. In this case the discrete opfibration $\Tw(P)\to\mathcal{V}(P)$ is a map of generalized Reedy categories.
\end{theorem}
\begin{proof}
As an image of discrete opfibration the category $\mathcal{V}(P)$ is a full subcategory of $\mathcal{U}(P)$ and contains all objects of $\mathcal{U}(P)$ that are isomorphic to objects of $\mathcal{V}(P)$. Factorizations of morphisms in $\mathcal{V}(P)$ via strict or orthogonal factorization system in $\mathcal{U}(P)$ lie in $\mathcal{V}(P)$. For any object $B$ in $\mathcal{V}(P)$ elements in the corresponding non-empty fiber over $B$ have the same degree, which we define to be the degree of $B$.
The proof is similar to that of Theorem~\ref{thm:groupoid-reedy}.
\end{proof}

\begin{example}
In all graph-substitution operads operations with the same multisets of colours have the same grading. For the operads with arity grading this is trivial. For the operad $mOp$ this is due to the fact that the first Betti number of an operadic graph is determined by the number of vertices and the number of inner edges. The latter is determined by the number of leaves and the number of half-edges, both in turn determined by the colours of the corresponding operation.
\end{example}

\paragraph{Dualizability.} In the setting of Theorem~\ref{thm:groupoid-reedy} and Theorem~\ref{thm:enveloping-image-reedy} the group of operations in $P(\psi=0)$ of a colour $c$ acts on the right by operadic composition on the set of operations in $P$ with the $j$-th input color $c$, and, in planar and symmetric cases, acts on the left on the set of operations in $P$ with the output colour $c$.

\begin{lemma}
In the setting of Theorem~\ref{thm:groupoid-reedy} or of Theorem~\ref{thm:enveloping-image-reedy} if the group actions above on operations of non-zero arity are free, then (and, in the latter case, only then) for any $f\in R_+$ with all non-source vertices of non-zero arity and any $\theta\in I$ such that $f\circ\theta=f$ we have $\theta=id$.
\end{lemma}
\begin{proof}
Each upper vertex of $f$ has arity at least $1$ and is adjacent to at least one leaf. Composition with $\theta$ does not permute the indices of leaves of $f$, thus the indices of leaves of $\theta$ are ordered trivially. Since the action is free, non-source vertices of $\theta$ are trivial. 

In the opposite direction, a fixed point of the action on $P(n)$ can be used to obtain non-trivial morphisms $f$ and $\theta$ such that $f\circ\theta= f$.
\end{proof}

\begin{example}
Most of the graph-substitution operads, including the operads $cuAs$, $iuAs_{iTr}$ and $iuAs_{Tr}$, satisfy these conditions. The exceptions are precisely the operads $mOp$, $mOp_{(g,n)}$, $mOp_{st}$, $mOp_{nc}$ and $ciuAs$. These operads contain a graph with one vertex of degree either $2$ or $4$ and with one loop, and a graph with one vertex of the same degree without loops, which together give a fixed point of the action of groups on operations. 

The twisted arrow categories of the operads $mOp$, $mOp_{(g,n)}$ and $ciuAs$ are not dualizable: these categories contain non-trivial automorphism of $\mu_0$ and at the same time have only one map $\mu_0\to\bigcirc$ in $R_+$. For the operads $mOp_{st}$ and $mOp_{nc}$ the \emph{left} actions on operations of non-zero arity are free: such an action permutes the indices of leaves of operadic graphs. Thus for any $f\in R_+$ and any $\theta\in I$ such that $f\circ\theta=f$ the upper vertices of $\theta$ are trivial and the permutation on leaves of $\theta$ is trivial. Thus the left action of the lower vertex of $\theta$ on the source of $\theta$ is trivial, and $\theta$ itself is either trivial or an automorphism of $\mu_0$, the only operation of arity $0$. In the latter case $f\circ\theta=f$ still implies that $\theta=id$. 

For twisted arrow categories of graph-substitution operads any morphism $f\in R_+$ with an upper vertex of arity $0$ is isomorphic to the output map $id_{-1}\to\bigcirc$, and again $f\circ\theta=f$ implies that $\theta=id$.  This shows that twisted arrow categories of all graph-substitution operads except $mOp$, $mOp_{(g,n)}$ and $ciuAs$ are dualizable generalized Reedy categories.
\end{example}

\section{Segal and 2-Segal presheaves}
\label{sec:segal}

\subsection{Single-object Segal presheaves} 

Algebras over a set-operad $Q$ correspond to special presheaves over $\Tw(Q)$.

\begin{definition}
  Let $P$ be an operad and $p$ and $q$ be composable operations in $P$. The partial composition pushout of $p$ and $q$ is the $Upper$-$Lower$ pushout of the $i$-th input map of $p$ and of the output map of $q$:
  \begin{center}
\begin{tikzcd}
id_c \arrow[r,"out"] \arrow[d,"in_i"'] & q \arrow[d] \\
p \arrow[r] & p\circ_i q
\end{tikzcd}
\end{center}
\end{definition}

\begin{definition}
Let $P$ be a $C$-coloured symmetric operad. A presheaf $X$ on $\Tw(P)$ satisfies single-object Segal condition if for any $p\in P(c_1,\dots,c_n;c_0)$ the product map $(X(in_j))_{1\leq j\leq n}: X(p)\to\prod_{j=1}^n X(id_{c_j})$ is bijective. In particular, for any operation $p$ of arity $0$ the set $X(p)$ is a point. Such presheaves $X$ will be called single-object Segal $P$-presheaves.
\end{definition}

\begin{theorem}
\label{thm:algebras-as-segal}
For any set-operad $P$ the category of single-object Segal $P$-presheaves is equivalent to the category of $P$-algebras.
\end{theorem}
\begin{proof}
Let $X$ be a single-object Segal $P$-presheaf. The corresponding  algebra $A$ is the $C$-coloured algebra with $A(c)=X(id_c)$ and with the algebra structure maps $A(p)$ equal to $\prod_{j=1}^n X(id_{c_j})\xleftarrow{\cong} X(p)\xrightarrow{out} X(id_{c_0})$. Next we prove that $A$ is a $P$-algebra.

First we show that the structure maps of $A$ are compatible with composition in $P$. Take any $p\in P(c_1,\dots,c_n;c_0)$ and $q\in P(d_1,\dots,d_m;c_i)$. We have to show that for any elements $a_j$ of $A$ of appropriate colour the element $(p\circ_i q)(a_1,\dots,a_{n+m-1})$ is equal to the element $p(a_1,\dots,a_{i-1},q(a_i,\dots,a_{i+m-1}),a_{i+m},\dots,a_{n+m-1})$. Observe that there exists the following commutative diagram in $\Tw(P)$, where $j\neq i$:

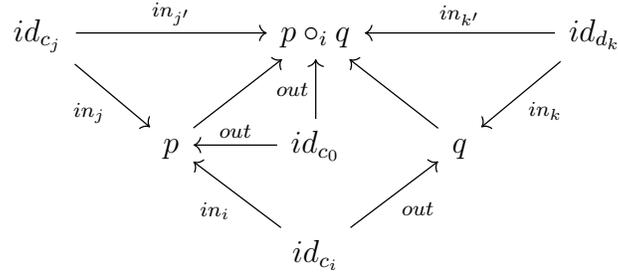
\begin{center}
\begin{tikzcd}
id_{c_j} \arrow[rd, "in_{j}"'] \arrow[rr, "in_{j'}"] &  & p\circ_i q & & id_{d_k} \arrow[ll, "in_{k'}"'] \arrow[ld, "in_k"] \\
& p \arrow[ru] & id_{c_0} \arrow[u, "out"] \arrow[l, "out"']    & q \arrow[lu] & \\
& & id_{c_i} \arrow[ru, "out"'] \arrow[lu, "in_i"] &&
\end{tikzcd}
\end{center}

Thus the following diagram is also commutative. 

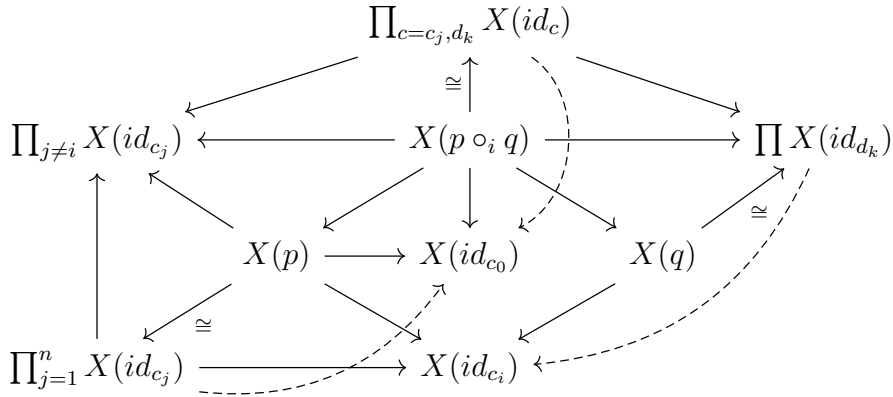
\begin{center}
\begin{tikzcd}[column sep=0.4cm]
& & \prod_{c=c_j,d_k} X(id_c) \arrow[rrd]\arrow[lld] \arrow[dd, dashed, bend left=60] & & \\
\prod_{j\neq i} X(id_{c_j}) & & X(p\circ_i q) \arrow[ll] \arrow[ld] \arrow[d] \arrow[rr] \arrow[rd] \arrow[u,"\cong"] &  & \prod X(id_{d_k}) \arrow[lldd, dashed, bend left] \\
 & X(p) \arrow[ld, "\cong"] \arrow[lu] \arrow[rd] \arrow[r] & X(id_{c_0})  & X(q) \arrow[ru, "\cong"'] \arrow[ld] &  \\
\prod_{j=1}^n X(id_{c_j}) \arrow[uu] \arrow[rr] \arrow[rru, dashed, bend right] & & X(id_{c_i})  &  &
\end{tikzcd}
\end{center}

The two upper horizontal maps in this diagram are the projection maps of a product. The upper dashed arrow is the map $A(p\circ_i q)$, the lower right dashed arrow is the map $A(q)$, and the lower left dashed arrow is the map $A(p)$. This diagram implies compatibility of the algebra structure maps with operadic composition: the map from the upper product to the lower left product is given by the identity maps on factors $X(c_j)$, $j\neq i$, and by the map $A(q)$ on the product of remaining factors. 

Since the input and the output maps of $id_c$ are the identity maps, the structure map corresponding to $id_c$ is the identity map. Finally, compatibility with symmetric group action is implied by the following commutative diagram:

\begin{center}
\begin{tikzcd}
& p \arrow[dd, "\sigma"] &\\
id_{c_j} \arrow[ru, "in_j"] \arrow[rd, "in_{\sigma^{-1}(j)}"'] & & id_{c_0} \arrow[lu, "out"'] \arrow[ld, "out"] \\
& p^\sigma &      
\end{tikzcd}
\end{center}

In the opposite direction, any $P$-algebra $A$ produces single-object Segal $P$-presheaf $X$ as follows. For any $p\in P(c_1,\dots,c_n;c_0)$ let $X(p)$ be equal to $\prod_{j=1}^n A(c_j)$. For a morphism $f:p\to r$ the map $X(f):X(r)\to X(p)$ is defined as follows. Let $q_j,\, j>0,$ be the upper vertices of $f$, and  $q_j\to r$ be the corresponding maps. For $X$ to be a single-object Segal $P$-presheaf the corresponding maps $X(r)\to X(q_j)$ have to be projections onto the subproducts of factors of $X(r)$ given by the indices of leaves above the vertices $q_j$ in $f$. For $X$ to be a presheaf the map $X(r)\to X(q_j)\xrightarrow[]{out^*} X(id_c)$ should coincide with $X(r)\to X(p)\xrightarrow[]{in_j^*} X(id_c)$. Thus we define the map $X(r)\to X(p)$ to be the composition of the permutation of factors (according to the permutation on leaves of $f:p\to r$), of the projection to the factors indexed by the leaves of the upper vertices, and of the product over $j$ of the algebra structure maps $A(q_j)=out^*:X(q_j)\xrightarrow[]{} X(id_c)$. 

Functoriality can be checked as follows. Let $s\to p$ and $p\to r$ be maps in $\Tw(P)$. The composition $X(r)\to X(p)\to X(s)$ is the composition of permutation of factors, projection, product of the algebra structure maps,  permutation, projection, and product of the algebra structure maps. It can be simplified to the composition of permutation of factors, projection, product of the algebra structure maps, and product of the algebra structure maps. The permutation and the projection are the same as in the tree $s\to p\to r$. Since $A$ is an algebra, the compositions of the algebra structure maps are the structure maps of compositions, i.e.\@ the structure maps of the upper vertices in $s\to p\to r$. 

Under this correspondence morphisms of single-object Segal presheaves coincide with morphisms of $P$-algebras.
\end{proof}

\begin{theorem}
\label{thm:evaluation-functor}
Let $P$ be an operad, $A$ be a $P$-algebra, and $X$ be the single-object Segal $P$-presheaf corresponding to the $P$-algebra $A$. There exists functor $ev(P,A):\Tw(P)/X\to \Tw_P(A)$.
\end{theorem}
\begin{proof}
Objects $g:p\to X$ of $\Tw(P)/X$ correspond to tuples of the form $(p,a_1,\dots,a_n)$, where the elements $a_j$ of $A$ are the images of $(g:p\to X)\in X(p)$ under the maps $in_j^*:X(p)\to X(id_{c_j})$. The functor $ev(P,A)$ sends a tuple $(p,a_1,\dots,a_n)$ to $p(a_1,\dots,a_n)\in A$, an object in $\Tw_P(A)$. 

A morphism $p\to r\to X$ in $\Tw(P)/X$ can be represented as a reduced symmetric tree corresponding to $p\to r$ with leaves marked by the same elements of $A$ as in the tuple corresponding to $r\to X$. Let $q_j$ be the $j$-th non-source vertex of $p\to r$. The output $q_j(a_{j,1},\dots,a_{j,k_j}),\,j>0,$ of the corolla $q_j\to X=q_j\to r\to X$ is equal to the $j$-th input of the corolla $p\to X$, and the input of the corolla $q_0\to X$ that is grafted into the output of $p\to X$ is equal to the output of $p\to X$. In this manner every edge in the tree $p\to r\to X$ can be marked by an element of $A$.

The functor $ev(P,A)$ sends the morphism $p\to r\to X$ to the corolla $q_0$, i.e.\@ to the morphism represented by $q_0(a',a_{0,1},\dots,a_{0,k_0})$, with $a'=p(a_1,\dots,a_n)$, $a_j=q_j(a_{j,1},\dots,a_{j,k_j})$. This map respects the composition of morphisms. Indeed, the lowest tree grafted into $s$ via $s\to p\to r\to X$ is obtained by grafting some of the corollas $q_i$ to the lower vertex $p_0$ of $s\to p$, and this tree, as an element of $TO(A)=F(U(A),U(A))/Rel\, A$, is equal to the tree obtained by grafting only $q_0$ into $p_0$, since each grafting of $q_i,\,i>0,$ into $p_0$ corresponds to generating equivalence of $TO(A)$. 
\end{proof}

Intuitively, the category of elements of the single-object Segal presheaf corresponding to an algebra $A$ describes the ways to decompose elements of the algebra $A$. The functor $ev$ can be seen as the evaluation.

\paragraph{Single-object Segal presheaves over categories $\mathcal{U}(P)$.} For any operad $P$ let $\alpha:\Tw(P)\to\mathcal{U}(P)$ be the canonical opfibration.  One can try to define single-object Segal presheaves over $\mathcal{U}(P)$ as presheaves $X$ such that for any operation $p$ in $P$ the map $(X(\alpha(in_j)))_{1\leq j\leq n}: X(\alpha(p))\to\prod X(\alpha(id_{c_j}))$ induced by the images of input maps of $p$ is a bijection. And while it seems possible to reconstruct a $P$-algebra from such a presheaf, this definition is incorrect, unless we additionally require that the map $(X(\alpha(in_j)))_{1\leq j\leq n}: X(\alpha(p))\to\prod X(\alpha(id_{c_j}))$ depends only on the image of $p$ in $\mathcal{U}(P)$, i.e.\@ the map is determined by the colours of $p$. With this additional requirement, we get the subcategory of presheaves over $\mathcal{U}(P)$ that is equivalent to the category of $P$-algebras. 

For any $C$-coloured operad $P$ and $P$-algebra $A$ we get the following commutative square.

\begin{center}
\begin{tikzcd}
\Tw_{sOp_C}(P)/A \arrow[r,"ev"] \arrow[d] & \Tw_P(A)\arrow[d,"\alpha"]\\
\mathcal{U}_{sOp_C}(P)/A \arrow[r,"ev"]& \mathcal{U}_P(A)
\end{tikzcd}
\end{center}
The presheaves denoted by $A$ are both presheaves corresponding to the algebra $A$, and the presheaf over $\Tw(P)$ is the restriction of the presheaf over $\mathcal{U}(P)$ along the canonical opfibration $\alpha:\Tw(P)\to\mathcal{U}(P)$.

\subsection{Segal conditions}

As first observed by Grothendieck, the category of small categories is equivalent to the category of Segal presheaves over $\Delta$ (\cite{segal1968classifying}). Two subcategories of $\Delta$ are used to define Segal conditions: the category $\Delta_0$ of lower (or inert) morphisms of $\Delta$ and the category $El(\Delta_0)$ of elementary objects, the full subcategory of $\Delta_0$ on objects $[0]$ and $[1]$. Segal condition for presheaves over $\Delta$ is the condition on the restriction of a presheaf over $\Delta$ to the subcategory $\Delta_0$. This condition has two equivalent versions: the condition to be a right Kan extension along the inclusion $El(\Delta_0)\to\Delta_0$, and the condition to be a sheaf for a certain coverage on $\Delta_0$.

For any set-operad $P$ there are two versions of Segal condition for presheaves over $\Tw(P)$. These generalize the two versions of Segal condition for presheaves over $\Delta$. For a rather general class of operads these two versions are equivalent. 

Three subcategories of the category $\Tw(P)$ can play the role of the category $\Delta_0$: the subcategories $Lower$, $Lower\circ Perm$ and $Inert$, where $Perm$ is the category of permutation isomorphisms. Both versions of Segal condition do not depend on the choice of one of these three subcategories. We have to consider these three subcategories since, while the choice of $Inert$ seems to be the most natural, it does not give the desirable notion of a free Segal $P$-presheaf.

\subsubsection{Condition on restrictions to be right Kan extensions}

The following form of Segal condition is a particular case of Segal condition of Chu and Haugseng (\cite{chu2019homotopy}). 

\begin{definition}
Let $P$ be a planar or symmetric operad endowed with a nice grading~$\psi$. An elementary object in $\Tw(P)$ is an identity operation or an operation $q$ such that $\psi(q)=-1$. The category $El(P, \psi, Lower)$, or $El(P, \psi, Lower\circ Perm)$, is the full subcategory of $Lower$ (equally, of $Lower\circ Perm$) on elementary objects.  The category $El(P, \psi, Inert)$ is the full subcategory of $Inert$ on elementary objects.
\end{definition}

\begin{definition}
  Let $P$ be a planar, symmetric or cyclic operad endowed with a nice grading. A petal map to an operation $t$ in $P$, or simply a petal of $t$, is a morphism $q\to t$ in $\Tw(P)$ from an operation $q$ of grading $(-1)$. For a presheaf $X$ on $\Tw(P)$ and a petal map $f:q\to p$ the petal of $x\in X(p)$ corresponding to $f$ is the element $f^*(x)\in X(q)$.
\end{definition}

Any morphism from an operation of arity $0$, thus any petal map, is a lower morphism. Petal maps are analogous to maps from $[0]$ in $\Delta$. For a Segal presheaf $X$ petals of an element in $X(id_c)$ are analogous to objects of a morphism in a category and to colours of an operation of a coloured operad.

\begin{definition}
  For a planar or symmetric set-operad $P$ let $\Tw(P)_0$ be the subcategory $Lower$, $Lower\circ Perm$ or $Inert$ of the category $\Tw(P)$. A Segal $P$-presheaf is a presheaf $X$ over $\Tw(P)$ such that the restriction of $X$ to $\Tw(P)_0$ is the right Kan extension of the restriction of $X$ to $El(\Tw(P)_0)$ along the inclusion $i:El(\Tw(P)_0)\to \Tw(P)_0$.
\end{definition}

This notion does not depend on the choice of the category $\Tw(P)_0$.

\begin{lemma}
For any operad $P$ and operation $p$ in $P$ the categories $El(P,\psi,Lower)/p$ and $El(P,\psi,Lower\circ Perm)/p$ are equal, and their inclusion into $El(P,\psi,Inert)/p$ is an equivalence of categories.
\end{lemma}
\begin{proof}
Here $El(P,\psi,C)/p$ is the comma category with respect to the inclusion of $p$ and $El(P,\psi,C)$ into $C$. The categories $El(P,\psi,Lower)/p$ and $El(P,\psi,Lower\circ Perm)/p$ are equal: their objects $f:q\to p$ are the input maps and the petal maps of $p$, and their morphisms $f_1\to f_2$ are the lower maps $g:q_1\to q_2$ such that $f_1=f_2\circ g$. The objects of the category $El(P,\psi,Inert)/p$ are the petal maps of $p$ and the inert maps from identity operations to $p$, and its morphisms $f_1\to f_2$ are the inert maps $g:q_1\to q_2$ such that $f_1=f_2\circ g$. Any object of $El(P,\psi,Inert)/p$ given by an inert map $id_c\to p$ with upper vertex $v$ is isomorphic to an input map $id_d\to p$ via the morphism $id_c\to id_d$ with upper vertex $v$ and lower vertex $v^{-1}$. Thus $El(P,\psi,Inert)/p$ is equivalent to its full subcategory $\mathcal{D}$ on petal and input maps of $p$. Since any inert morphism between input maps of $p$ is trivial, the category $\mathcal{D}$ coincides with the category $El(P,\psi,Lower)/p$.
\end{proof}

\begin{corollary}
\label{crl:segal-well-defined} 
For any operad $P$ and presheaf $X$ over $El(P,\psi,Inert)$ the restriction to $Lower$ of the right Kan extension of $X$ along the inclusion $El(P,\psi,Inert)\to Inert$ is the right Kan extension along the inclusion $El(P,\psi,Lower)\to Lower$ of the restriction of $X$ to $El(P,\psi,Lower)$. The same holds with $Inert$ or $Lower$ replaced by $Lower\circ Perm$.
\end{corollary}
\begin{proof}
The diagrams $G'$ computing the pointwise right Kan extension of the restriction of $X$ are equal to diagrams $G\circ F$, where $G$ is the diagram that computes the pointwise right Kan extension of $X$, and $F$ is the inclusion of the form $El(P,\psi,Lower)/p\to El(P,\psi,Inert)/p$. The inclusion $F$ is an equivalence, thus the limits of $G$ and $G'$ coincide.
\end{proof}

\begin{example}
The category of single-object Segal presheaves is equivalent to the category of Segal presheaves $X$ such that $X(q)$ is a singleton for all operations $q$ of grading $(-1)$: a single-object Segal presheaf $X$ over $\Tw(P)$ is a Segal presheaf, and $X(q)$ is a singleton for all operations $q$ of arity $0$, and in particular, for operations of grading $(-1)$; if $Y$ is a Segal presheaf over $\Tw(P)$ such that  $Y(q)$ is a singleton for all $q$ such that $\psi(q)=-1$, then $Y(q)$ is a singleton for all operations of arity $0$, since the comma categories $El(\Tw(P)_0)/p$ for operations $p$ of arity $0$ are either empty or consist only of petal maps.
\end{example}

\begin{example}
The category of Segal presheaves over $\Tw_{sOp}(iuAs)$ is the category of small categories with anti-involution, and not the category of dagger categories. This also shows that the set of colours of a coloured \emph{cyclic} operad  should be endowed with involution.
\end{example}

\subsubsection{Condition on restrictions to be sheaves} 

We will use the following notation for coverages and sheaves (\cite{johnstone2002sketches}).

\begin{definition}
  A coverage $T$ on a category $C$ is an assignment, for each object $U$ of $C$, of a collection $T(U)$ of families $(f_i:U_i\to U,\, i\in I)$ of morphisms with common codomain $U$, called $T$-covering families, such that for any $T$-covering family $(f_i:U_i\to U,\, i\in I)$ and any morphism $g:V\to U$ there exists a $T$-covering family $(h_j:V_j\to V,\,j\in J)$ such that each composition $g\circ h_j$ factors through some morphism $f_i$ via a map $V_j\to U_i$. 
\end{definition}

\begin{definition}
  A presheaf $X$ on a category $C$ satisfies the sheaf axiom for a family of morphisms $(f_i:U_i\to U,\,i\in I)$ if for any \emph{compatible} family of elements $s_i\in X(U_i)$ there is exactly one element $s\in X(U)$ such that $f_i^*(s)=s_i$ for all $i$ in $I$. \emph{Compatibility} is the condition that for any $g:V\to U_i$ and $h:V\to U_j$ (where $i$ and $j$ might be equal) such that $f_i\circ g=f_j\circ h$ we have $g^*(s_i)=h^*(s_j)$. A presheaf $X$ on a category $C$ is a sheaf for a coverage $T$ on the category $C$ if it satisfies the sheaf axiom for all $T$-covering families.
 \end{definition}

There is a coverage $S'$ on $\Delta_0$ that assigns a single family of morphisms to each object. The family $S'([0])$ consists of the identity map, and the families $S'([n])$, $n\neq 0$, consist of all input maps of $[n]$. A Segal presheaf over $\Delta$ is a presheaf whose restriction to $\Delta_0$ is a sheaf. The coverage $S_\psi$ on $\Tw(P)_0$ that plays the same role as (but is not a generalization of) the coverage $S'$ on $\Delta$ is defined as follows.

\begin{definition}
Let $P$ be a planar or symmetric operad with a nice grading $\psi$. The Segal coverage $S_\psi$ on $\Tw(P)_0$ assigns to each operation $p$ of $P$ a single family $S_\psi(p)$ of morphisms. The family $S_\psi(p)$ is the union of all input and  petal maps of $p$. 
\end{definition}

\begin{proposition}\label{pr:segal-cov-correct}
Segal coverage is a coverage for the categories $Lower$, $Lower\circ Perm$ and $Inert$.
\end{proposition}
\begin{proof}
The case of $Lower$ and of $Lower\circ Perm$ is simple: a composition of a petal map with any morphism is a petal map, and a composition of an input map with a lower morphism or a morphism in $Lower\circ Perm$ is an input map. For the case of $Inert$, let $h:id_c\to V$ be an input map and $g:V\to U$ be an inert map. Let $v$ be the operation that marks the upper vertex of $g\circ h$. Then $v$ is an isomorphism in the category $P(1)$, and there is isomorphism $k:id_c\to id_d$ with upper vertex $v$ and lower vertex $v^{-1}$. The morphism $g\circ h$ is equal to the composition of $k$ with the $j$-th input map of $U$, where $j$ is the index of the leaf above the source vertex of $g\circ h$. 
\end{proof}

\begin{lemma}
\label{lem:factor-at-most-one-input}
An inert morphism $f:p\to t$ from an operation of arity $1$ can factor in $\Tw(P)$ through at most one input map of $t$.
\end{lemma}
\begin{proof}
If $f$ factors through the $j$-th input map of $t$, then the index  of the leaf above the source vertex of $f$ is $j$.
\end{proof}

\begin{lemma}
Let $P$ be a planar or symmetric set operad with a nice grading $\psi$, $X$ be a presheaf over $Inert$, $U$ be an operation of $P$, the family $(f_i:U_i\to U,\,i\in I)$ be the covering family $S_\psi(U)$, and $s_i$ be a family of sections of $X(U_i)$. Then the family of sections $s_i$ is compatible over $Inert$ if and only if it is compatible over $Lower$. The same holds with $Lower$ or $Inert$ replaced by $Lower\circ Perm$.
\end{lemma}
\begin{proof}
Implication of compatibility from $Inert$ to $Lower$ is obvious. Assume the family $s_i$ is compatible over $Lower$. Let $g:V\to U_i$ and $h:V\to U_j$ be inert morphisms such that $f_i\circ g=f_j\circ h$. Let $g'\circ u$ and $h'\circ u$ be the $(Upper,Lower)$-factorizations of $g$ and $h$. Since $f_i$ and $f_j$ are lower morphisms, the morphisms $g$ and $h$ have the same upper morphism $u$ in their factorization.  The lower morphisms $f_i\circ g'$ and $f_j\circ h'$ are equal, and thus by compatibility assumption $g'^*(s_i)=h'^*(s_j)$, and $g^*(s_i)=h^*(s_j)$.
\end{proof}

\begin{corollary}
For any planar or symmetric set-operad $P$ a presheaf $X$ over  $Inert$ is a sheaf if and only if the restriction of $X$ to $Lower$ is a sheaf. The same holds with $Lower$ or $Inert$ replaced by $Lower\circ Perm$.
\end{corollary}

\subsubsection{Connection between the two versions}

\begin{proposition}
For any planar or symmetric operad $P$ with a nice grading the right Kan extension $i_*:Psh(El(\Tw(P)_0))\to Psh(\Tw(P)_0)$ is fully faithful, and any presheaf in its image is a sheaf.
\end{proposition}
\begin{proof}
The inclusion of $El(\Tw(P)_0)$ into $\Tw(P)_0$ is fully faithful, thus the right Kan extension $i_*$ is fully faithful. Let $Y$ be a presheaf over $El(\Tw(P)_0)$ and $p$ be an operation with $S$-covering family $(p_i\to p)$. Any family $s_i$ of sections of $Y$ over objects $p_i$ corresponds to a choice of elements $s_i$ in the objects indexed by the diagram that computes $i_*Y(p)$. If this family of sections is compatible, then the choice of elements $s_i$ is compatible with the mentioned diagram, and thus corresponds to an element in the limit indexed by this diagram, i.e.\@ corresponds to an element in $i_*Y(p)$. Since any element in $i_*Y(p)$ gives a compatible family of sections and is determined uniquely by this family, the presheaf $i_*Y$ is a sheaf.
\end{proof}

\paragraph{When petals factor through inputs.} There is a property of operads that ensures that Segal presheaves over corresponding twisted arrow categories behave particularly well, and that generalized Reedy structure on these categories is particularly nice. Analogous property appears in a different place and plays a more fundamental role in the generalization of the present theory to Lawvere theories.

\begin{definition}
  Let $P$ be a symmetric or planar operad. We will say that `petals factor through inputs' in $\Tw(P)$ if for any petal map $f:q\to t$, \emph{with $t$ having arity at least $1$,} the morphism $f$ can be factored as $q\to id_c\to t$, where $id_c\to t$ is an input map. 
\end{definition}

\begin{example}
Let $P$ be a graph-substitution operad and $p$ be an operation in $P$ of non-zero arity. Input maps to $p$ correspond to vertices of $p$. Petal maps to $p$ correspond to edges of $p$, endowed with an orientation if $P(1;1)$ is non-trivial. A petal $f:q\to p$ factors through an input map $g:id_c\to p$ if and only if the edge corresponding to $f$ is adjacent to the vertex corresponding to $g$. In twisted arrow categories of graph-substitution operads petals factor through inputs.
\end{example}

This property is related to Segal conditions as follows. Notice that the Segal coverage on $\Tw(uAs_{pl})_0$ is slightly different from the coverage $S'$ on $\Delta$. The proper generalization appears in the following proposition.

\begin{proposition}
Let $P$ be a planar or symmetric operad with a nice grading $\psi$, and $S'_\psi$ be the function assigning to each object $p$ the following single family of morphisms in $\Tw(P)_0$: all input maps of $p$ if $p$ has non-zero arity, and all petal maps of $p$ if $p$ has arity $0$. Petals factor through inputs in $\Tw(P)$ if and only if $S'_\psi$ is a coverage of $Lower$ (or equivalently, a coverage of $Inert$, or of $Lower\circ Perm$). In this case the Segal coverage $S_\psi$ is equivalent to the coverage $S'_\psi$, i.e.\@ the corresponding sheaf conditions coincide.
\end{proposition}
\begin{proof}
We will denote $S_\psi$ and $S'_\psi$ simply by $S$ and $S'$. Assume $S'$ is a coverage, and let $f:q\to p$ be a petal map to an operation $p$ of non-zero arity. The morphism $id_q$ is in $S'(q)$, thus $f\circ id_q=f$ factors through a morphism in the family $S'(p)$ of input maps of $p$. 

In the opposite direction, assume that petals factor through inputs and let $g:p\to t$ be an inert morphism. If $p$ and $t$ both have arity $0$, then for any morphism $f$ in $S'(p)$ the composition $g\circ f$ is in $S'(t)$. If $p$ and $t$ both have non-zero arity, then any morphism $f$ in $S'(p)$ is an input map, and, as in the proof of Proposition~\ref{pr:segal-cov-correct}, the composition $g\circ f$ factors through an input map of $t$. Finally, if $p$ has arity $0$ and $t$ has non-zero arity, then any morphism $f$ in $S'(p)$ is a petal, and thus $g\circ f$ is a petal, and this petal factors though an input of $t$ by assumption.

If petals factor through inputs, then the coverages $S$ and $S'$ generate the same sieves, and thus their sheaves coincide.
\end{proof}

\begin{proposition}
Let $P$ be a planar or symmetric operad such that: (1) either any morphism in the category $P(1)$ has at most one left inverse (2) or petals of operations of arity $1$ factor through inputs. Then the right Kan extension $i_*:Psh(El(\Tw(P)_0)\to Sh(\Tw(P)_0)$ is an equivalence of categories.
\end{proposition}
\begin{proof}
We have to show that $i_*$ is essentially surjective. It suffices to show that for any sheaf $X$ over $\Tw(P)_0$ and any family of morphisms $(f_i:p_i\to p)$ in $S$ a family of sections over $p_i$ that is compatible with respect to morphisms in $El(\Tw(P)_0)$ is compatible (with respect to morphisms in the sieve generated by the family $(f_i:p_i\to p)$). Sources of morphisms in this sieve have arity $0$ or arity $1$. Morphisms from operations of arity $0$ to operations of grading $(-1)$ and to identity operations are necessary petals. Morphisms in the sieve and not in $El(\Tw(P)_0)$ are precisely the lower (or inert) morphisms from non-identity operations of arity $1$ to identity operations. Lower morphisms from an operation $t\in P(a; b)$ to the operation $id_a$ correspond to left inverses of $t$ in $P(1)$. Any inert morphism $g:t\to id_c$ is a composition of a lower morphism $t\to id_a$,  determined uniquely by $g$, and an isomorphism $id_a\to id_c$,  determined uniquely by the upper part of $g$.

Take any two morphisms $g:t\to p_j$ and $h:t\to p_k$ in the sieve and not in $El(\Tw(P)_0)$ and such that $f_j\circ g=f_k\circ h$. By Lemma~\ref{lem:factor-at-most-one-input} we have $j=k$. The upper parts of $g$ and $h$ coincide. If condition (1) holds, then the lower parts of $g$ and $h$ coincide, and then compatibility with respect to $El(\Tw(P)_0)$ implies compatibility with respect to the sieve.

The condition (2) and the sheaf condition for $X$ imply that for any operation $t\in P(a;b)$ the map $in^*:X(t)\to X(id_a)$ induced by the input map of $t$ is injective. Let $g:t\to p_j$ and $h:t\to p_k$ be two inert maps such that $f_j\circ g=f_k\circ h$. Again $j=k$, and we have to show that $g^*=h^*$, and since the upper parts of $g$ and $h$ coincide, we can assume that $g$ and $h$ are lower maps. The composition $g\circ in$ with the input map of $t$ is the identity map of $id_a$, thus the map $in^*:X(t)\to X(id_a)$ is surjective, and bijective, and the map $g^*$ is the inverse of $in^*$. Similarly, the map $h^*$ is the inverse of $in^*$, and $g^*=h^*$. This implies compatibility with respect to the sieve.
\end{proof}

\begin{example}
Graph-substitution operads satisfy both of the conditions above. The first condition is implied by the following property: any morphism in $P(1)$ that has a left inverse is an isomorphism. This property holds for operads that satisfy the condition of Theorem~\ref{thm:groupoid-reedy}. Indeed, for a planar or symmetric operad $P$ with a nice grading any morphism in $P(1)$ that has a left inverse has grading $0$. 
\end{example}

\paragraph{Segal presheaves as multi-object algebras.} Consider the following informal idea. ``A multi-object algebra'' is a partial algebra endowed with maps to sets of ``objects'' such that (1) composability of elements of such a partial algebra (2) and the objects of this composition are both determined by the objects of composed elements. 

Let $P$ be an operad such that petals factor through inputs in $\Tw(P)$ and $X$ be a Segal presheaf over $\Tw(P)$. Then $X$ is a multi-object algebra in the following way. The first condition for multi-object algebras is satisfied in the sense that for any operation $p$ the elements of $X(p)\cong \lim_{s\to p} X(s)$ correspond to tuples of elements in the sets $X(id_c)$ that have compatible petals, and by definition these tuples are precisely the tuples of composable elements in the partial algebra $X$, and their composition via $p$ is the output of the corresponding element in $X(p)$. The second condition for multi-object algebras is satisfied since each petal of the output of $p$ is a petal of $p$, and thus is a petal of some input of $p$.

\subsubsection{Some examples}

\paragraph{Nodeless loop and graphical categories.} The use of grading in Segal conditions is motivated by graph-substitution operads containing the operation $\bigcirc$. First we describe maps to and from $\bigcirc$ in twisted arrow categories of graph-substitution operads. For all graph-substitution operads except $iuAs_{iTr}$ the only morphism from $\bigcirc$ is the trivial morphism, and there is only one morphism $id_{-1}\to\bigcirc$, since the only element with at least one input colour equal to $(-1)$ is $id_{-1}$, the vertex without leaves. The operad $iuAs_{iTr}$ additionally contains the operation $\theta$, thus in $\Tw(iuAs_{iTr})$ there is exactly one non-trivial morphism $f$ from $\bigcirc$, the morphism $f$ is an involutive automorphism, and there are two morphisms from $id_{-1}$ and two morphisms from $\theta$ to $\bigcirc$.

For any graph-substitution operad $P$ let $f:p\to \bigcirc$ be a morphism in $\Tw(P)$, with $p$ different from $id_{-1}$ and $\theta$. The morphism $f$ is defined uniquely by the morphisms $l$ and $u$ in the factorization $l\circ u$ of $f$ into upper and lower morphisms. The target of $u$ has arity $0$, and is either $\mu_0$ or $\bigcirc$. The morphism $u$ exists if and only if all vertices of $p$ have degree $2$, and the morphism $u$ is determined uniquely by $p$. We have already described endomorphisms of $\bigcirc$. Suppose then that the source of $l$ is $\mu_0$. For operads $iuAs_{Tr}$ and $iuAs_{iTr}$ there are exactly two morphisms from $\mu_0$ to $\bigcirc$, and for the remaining graph-substitution operads there is only one such morphism.  

\begin{remark}
This suggests that the definition of the graphical category $\widetilde{U}$ should be changed (\cite[Remark 4.8]{hackney2019graphical}). Any twisted arrow category that has two endomorphisms of $\bigcirc$ has two morphisms from $id_{-1}$ to $\bigcirc$. It is not clear if the operad that contains both $mOp$ and $iAs_{Tr}$ exists, unless the vertices of the corresponding graphs are endowed with genus map.
\end{remark}

Let $P$ be a graph-substitution operad containing $\bigcirc$, and different from $iuAs_{Tr}$ and $iuAs_{iTr}$. The operad $P$ has two natural gradings: the grading that was denoted by $\psi$ and the arity grading $\phi$. Segal condition for the grading $\psi$ implies that the map $X(\mu_0\to\bigcirc)$ is an isomorphism. In Segal condition for the arity grading $\phi$ the map $id_\bigcirc$ is a petal map, and there are no restrictions on the set $X(\bigcirc)$. The choice of the grading $\psi$ here is more natural, and in the case of $cuAs$ recovers familiar structure.

\paragraph{Cyclic nerve and trace.} Any object in the category $\Tw(cuAs)$ is isomorphic either to $\mu_n$, to $\nu\circ\mu_n$ or to $id_{-1}$, thus the full subcategory $\mathfrak{C}$ on these objects is equivalent to $\Tw(cuAs)$. The subcategory of $\mathfrak{C}$ on objects $\mu_n$ is isomorphic to $\Delta$, and any morphism in $\mathfrak{C}$ to $\mu_n$ is in $\Delta$. There is only trivial morphism to $id_{-1}$, and there is exactly one morphism from $id_{-1}$ to $\nu\circ\mu_n$ for all $n$. The target of any morphism from $\nu\circ\mu_n$ is of the form $\nu\circ\mu_m$ for some $m$. The full subcategory of $\mathfrak{C}$ on objects $\nu\circ\mu_n$ is isomorphic to the category $\Lambda$ endowed with additional terminal object $\bigcirc$. Indeed, any morphism in this subcategory is an upper morphism, and has unique representative with vertices marked by $\mu_n$. These representatives correspond to representatives of morphisms in the full subcategory of $\Tw_{cOp}(uAs_c)$ on operations $\mu_n$, with indices of leaves shifted by $1$. Any morphism from $\mu_n$ to $\nu\circ\mu_m$ uniquely factors as $i\circ f\circ u$, where $u$ is as an upper morphism in $\Delta$, $f:\mu_k\to\nu\circ\mu_m$ is a lower morphism  such that the lower vertex of $f$ is marked by $\nu\circ\mu_{m-n+1}$ and is connected to the source vertex by the first edge, and the permutation on leaves of $f$ is trivial, and $i$ is an automorphism of $\nu\circ\mu_m$.

Recall that the trace of a category $C$ is the quotient map $Tr:End(C)\to End(C)/{\sim}$, with the equivalence generated by $fg\sim gf$ for all $f:x\to y$ and $g:y\to x$ in $C$.

\begin{proposition}
Segal presheaves $X$ over $\Tw(cuAs)$ correspond to small categories $C$ endowed with a map $t$ from $End(C)$ that factors through the trace. The restriction of $X$ to $\Delta$ is the nerve of the corresponding category $C$, the restriction of $X$ to $\Lambda$ is the cyclic nerve of $C$, the set $X(\bigcirc)$ is the set $X(\mu_0)=Obj(C)$, the set $X(\nu)$ is the set $End(C)$, and the map $X(\nu)\to X(id_{-1})$ is the map $t$.
\end{proposition}
\begin{proof}
The proof is straightforward. Endomorphisms $fg$ and $gf$ are sent to the same element in $X(id_{-1})$ since the compositions of the two maps $\nu\to\nu\circ\mu_2$ with $id_{-1}\to\nu$ are equal.
\end{proof}

\subsection{2-Segal presheaves, or decomposition spaces}
\label{sec:decomposition-spaces}

Existence of Upper-Lower pushouts in twisted arrow categories of operads allows to generalize the notion of (discrete) decomposition space (\cite{decompositionspaces1,decompositionspaces2,decompositionspaces3}), or 2-Segal space (\cite{Dyckerhoff2019higher}), to presheaves over twisted arrow categories of operads. For any operad $P$ discrete decomposition spaces over $\Tw(P)$ correspond to special morphisms of operads into $P$.

\begin{definition}
  Let $P$ be a planar or symmetric operad. A $P$-decomposition space, or a 2-Segal presheaf, is a presheaf $X$ over $\Tw(P)$ such that the pushout of any pair $(f,g)$ of maps with common source, with $f$ active and $g$ inert, is mapped by $X$ to a pullback.
\end{definition}

\begin{lemma}
For any operad $P$ a presheaf $X$ over $\Tw(P)$ is a discrete decomposition space if and only if for all composable operations $p$ and $q$ in $P$ the restriction of $X$ to the partial composition pushout of $p$ and $q$ is a pullback. 
\end{lemma}
\begin{proof}
The \emph{only if} part is trivial: a partial composition pushout is a pushout of upper and lower maps. In the opposite direction, the assumption is that for all composable $p$ and $q$ we have $X(p\circ_i q)=X(p)\times_{X(id_c)} X(q)$. Let $f$ and $g$ be an active and an inert maps with a common source. By replacing these with isomorphic maps we can assume that $f$ and $g$ are upper and lower maps with trivial permutation on leaves, and these have the form $f:p\to\gamma(p,q_1,\dots,q_n)$ and $g:p\to q_0\circ_1 p$, where $\gamma$ is the operadic composition. The assumption implies that $X(\gamma(p,q_1,\dots,q_n))$ is the fibered product of $X(p)$ with $\prod_{j>0}X(q_j)$ over $\prod_{j>0}X(id_{c_j})$, that $X(q_0\circ_1 p)$ is the fibered product of $X(p)$ with $X(q_0)$ over $X(id_{c_0})$, and that $X(q_0\circ_1\gamma(p,q_1,\dots,q_n))$ is the fibered product of $X(p)$ with $\prod_{j\geq 0} X(q_j)$ over $\prod_{j\geq 0} X(id_{c_j})$. This implies that $X(q_0\circ_1\gamma(p,q_1,\dots,q_n))$ is the fibered product of $X(\gamma(p,q_1,\dots,q_n))$ and $X(q_0\circ_1 p)$ over $X(p)$.
\end{proof}

\begin{example}
For any operad $P$ \emph{single-object} Segal presheaves over $\Tw(P)$ are discrete decomposition spaces. 
\end{example}

Recall that for any small category $B$ the category of presheaves over  $B$ is equivalent to the category of discrete fibrations from small categories to the category $B$. A discrete fibration $G:E\to B$ generates the presheaf $Y$ such that $Y(b)=G^{-1}(b)$ and $Y(f:b'\to b)(e)=e'$, where $e'$ is the source of the unique lift $h:e'\to e$ with target $e$ of a morphism $f$. The category $E$ is isomorphic to the category $B/Y$ of elements of $Y$, and this isomorphism induces isomorphism between the functor $G$ and the projection functor $B/Y\to Y$. Any morphism in the category of discrete fibrations over $B$ is a discrete fibration. Morphisms $f:X\to Y$ of presheaves over the category $B$ correspond to discrete fibrations $f:B/X\to B/Y$ over $B$. The presheaf $X'$ over $B/Y$ corresponding to a discrete fibration $f:B/X\to B/Y$ over $B$ is the subpresheaf of the restriction of $X$ to $B/Y$ such that $X'(y)=f^{-1}(y)$. The category $(B/Y)/X'$ of elements of $X'$ is isomorphic to the category $B/X$ of elements of $X$.

\begin{lemma}
\label{lem:discrete-criterion}
Let $f:P_1\to P$ be a morphism of operads. The induced functor $f_*:\Tw(P_1)\to \Tw(P)$ is a discrete fibration if and only if for any $p$ and $q$ in $P$ and $z$ in $P_1$ such that $f(z)=p\circ_i q$ there exist unique $x$ and $y$ in $P_1$ such that $x\circ_i y=z$, $f(x)=p$, and $f(y)=q$. 
\end{lemma}
\begin{proof}
Assume that $f_*$ is a discrete fibration. Given $p$ and $q$ as above, let $g:p\to p\circ_i q$ be the morphism in the partial composition pushout of $p$ and $q$. Let $g':x\to z$ be the unique lift of $g$ with target $z$. Since at most one non-source vertex of $g$ is non-trivial, uniqueness of $g'$ implies that at most one non-source vertex of $g'$ is non-trivial, and $g'$ is of the form $x\to x\circ_i y$, with $f(y)=q$, and with $x$ and $y$ defined uniquely by $g$. This also implies that partial composition pushouts lift to partial composition pushouts.

In the opposite direction, the condition implies that only identity operations are mapped to identity operations, that only permutation isomorphisms are mapped to permutation isomorphisms, and that lifts of the morphisms with exactly one non-trivial non-source vertex exist, are unique, and have exactly one non-trivial non-source vertex. It suffices to show that any morphism with trivial permutation on leaves has a unique lift. Such a morphism is a unique composition of morphisms with trivial permutation on leaves and with only $j$-th non-source vertex being possibly non-trivial, with $j$ decreasing from $n$ to $0$, where $n$ is the arity of the source of the morphism. The morphisms in the composition have unique lifts, and the composition itself has unique lift.
\end{proof}

\begin{proposition}
Let $P$ be a planar or symmetric operad and $Y$ be a presheaf over $\Tw(P)$. The presheaf $Y$ corresponds to discrete fibration $f_*:\Tw(P_1)\to \Tw(P)$ induced by a morphism of operads $f:P_1\to P$ if and only if $Y$ is a discrete decomposition space.
\end{proposition}
\begin{proof}
The first part of the proof continues the preceding lemma. Assume that $Y$ corresponds to the discrete fibration $f_*$ induced by a morphism of operads $f$. Since only identity operations are mapped by $f$ to identity operations, the set of identity operations of $P_1$ is the set $\bigsqcup_c f^{-1}(id_c)=\bigsqcup_c Y(id_c)$. Let $p$ and $q$ be composable operations in $P$, and $id_c\to p$ be the $i$-th input map of $p$ and $id_c\to q$ be the output map of $q$. The set $Y(p)\times_{Y(id_c)} Y(q)$ consists of all composable operations $x$ and $y$ in $P_1$ lying over $p$ and $q$ respectively. By preceding lemma the map $\circ_i:Y(p)\times_{Y(id_c)} Y(q)\to Y(p\circ_i q)$, the restriction of the partial composition map, is a bijection. For all pairs $(x,y)$ in $Y(p)\times_{Y(id_c)} Y(q)$ we have $x=Y(p\to p\circ_i q)(x\circ_i y)$ and $y=Y(q\to p\circ_i q)(x\circ_i y)$, i.e.\@ the bijection above is the isomorphism of pullbacks. This shows that $Y$ sends partial composition pushouts to pullbacks, i.e.\@ $Y$ is a discrete decomposition space.

In the opposite direction, define $\bigsqcup_c Y(id_c)$-coloured operad $P_1$ as follows. Operations in $P_1$ of arity $n$ are the elements of $Y(p)$ for all operations $p$ in $P$ of arity $n$. The $i$-th input and the output colour of $x$ in $Y(p)$ are the $i$-th input and the output of $x$, i.e.\@ the images of $x$ under the maps $Y(p)\to Y(id_c)$. Again, the set $Y(p)\times_{Y(id_c)} Y(q)$ consists of all composable operations $x$ and $y$ in $P_1$ lying over $p$ and $q$ respectively. The partial composition $x\circ_i y$ is computed via the canonical isomorphism $Y(p)\times_{Y(id_c)} Y(q)\to Y(p\circ_i q)$ between pullbacks. Operad axioms for $P_1$ follow from relations in $\Tw(P)$, e.g.\@ associativity of composition can be proved by factoring the obvious morphisms from operations $p$, $q$, $r$ to the operations $(p\circ_i q)\circ_j r$ in two ways.
\end{proof}

\begin{example}
For an operad $P$ and a \emph{single-object} Segal $P$-presheaf $X$ the operad over $P$ corresponding to the discrete decomposition space $X$ is the operad of $P$-algebras over the $P$-algebra corresponding to $X$, i.e.\@ the Baez--Dolan plus construction $(P, X)^+$.
\end{example}

\begin{definition}
  A morphism of operads $f:P_1\to P$ is a decomposition morphism if the induced functor $f_*:\Tw(P_1)\to \Tw(P)$ is a discrete fibration.
\end{definition}

\begin{example}
A morphism of categories $C_1\to C$ is a decomposition morphism if and only if it is a discrete Conduch\'e fibration (\cite{MR1667312}).
\end{example}

\begin{lemma}
Decomposition morphisms form a wide subcategory of the category of operads. For any decomposition morphisms $P_i\to P$ over the same base $P$ morphisms $P_1\to P_2$ over $P$ are decomposition morphisms, and any morphism $h:\Tw(P_1)\to \Tw(P_2)$ of categories over $\Tw(P)$ is induced by a decomposition morphism $f:P_1\to P_2$. 
\end{lemma}
\begin{proof}
A composition of discrete fibrations is a discrete fibration. Morphisms $E_1\to E_2$ between discrete fibrations $E_i\to B$ over the same base are discrete fibrations. Let $P_i\to P$ be decomposition morphisms, and $h:\Tw(P_1)\to \Tw(P_2)$ be a morphism over $\Tw(P)$. Since the functor $\Tw(P_1)\to \Tw(P)$ and lifts along the discrete fibration $\Tw(P_2)\to \Tw(P)$ preserve the classes of input maps, output maps,  partial composition pushouts and permutation isomorphisms, the functor $h$ preserves the aforementioned classes of maps and diagrams, and thus $h$ is induced by a morphism of operads.
\end{proof}

\begin{corollary}
For any operad $P$ the category of discrete decomposition spaces over $\Tw(P)$ is equivalent to the category of decomposition morphisms over $P$.
\end{corollary}

\subsection{Segal presheaves as algebras}

We introduce a class of operads called \emph{palatable}. For a palatable operad $P$ Segal $P$-presheaves can be seen as algebras over operads related to $P$.

\begin{definition}
  Let $P$ be a planar or symmetric operad with a nice grading. The operad $P$ is \emph{palatable} if any Segal presheaf over $\Tw(P)$ is 2-Segal. For a palatable operad $P$ and a 2-Segal presheaf over $\Tw(P)$ the operad $P_X$ is the operad over $P$ that corresponds to $X$.
\end{definition}

\begin{definition}
  For a planar or symmetric operad $P$ with a nice grading $\psi$ the category $\mathcal{P}_\psi$ is the full subcategory of $\Tw(P)$ on objects $q$ of grading $(-1)$. Let $X$ be a presheaf over $\Tw(P)$ and $k$ be the inclusion of $\mathcal{P}_\psi$ into $\Tw(P)$. The presheaf $Pl(X)$ over $\Tw(P)$ is the presheaf $k_*k^*X$, and $\eta_X:X\to Pl(X)$ is the unit of the adjunction $k^*\dashv k_*$. In other words, presheaves $Pl(X)$ are determined by their petals, presheaves $X$ and $Pl(X)$ have the same petals, and $Pl(X)$ is terminal among presheaves with petals the same as in $X$.
\end{definition}

\begin{proposition}
For any operad $P$ presheaves $Pl(X)$ over $\Tw(P)$ are Segal.
\end{proposition}
\begin{proof}
Morphism in $\Tw(P)$ from objects of $\mathcal{P}_{\psi}$ are lower morphisms. Thus for any presheaf $Y$ over $\mathcal{P}_\psi$ the restriction $j^*k_*Y$ of the right Kan extension of $Y$ along $k$ to the subcategory $Lower$ and the right Kan extension of $Y$ along the inclusion of $\mathcal{P}_\psi$ into $Lower$ are computed by the same formula, i.e.\@ the two presheaves are isomorphic. The inclusion of $\mathcal{P}_\psi$ into $Lower$ is the composition of the inclusion of $\mathcal{P}_\psi$ into $El(Lower)$ with the inclusion $i:El(Lower)\to Lower$, thus $j^*k_*Y$ is a right Kan extension along $i$, and the presheaf $k_*Y$ is Segal.
\end{proof}

\begin{proposition}
Let $P$ be an operad and $X$ be a presheaf over $\Tw(P)$ such that $Pl(X)$ is 2-Segal. If $X$ is Segal, then the corresponding presheaf $X'$ over $\Tw(P_{Pl(X)})$ is single-object Segal. This correspondence gives an equivalence between the category of single-object Segal presheaves over $\Tw(P_{Pl(X)})$ and the category of Segal presheaves $Y$ over $\Tw(P)$ such that $Pl(Y)=Pl(X)$ and petal-preserving morphisms.
\end{proposition}
\begin{proof}
For any operation $p'$ in $P_{Pl(X)}$ that lies over an operation $p$ in $P$ the set $X'(p')$ consists of elements in $X(p)$ that have petals equal to petals of $p'$. If $X$ is Segal, then $X'(p')$ is the subset of the limit taken over $El(Lower)/p$, the subset of the product indexed by petals and inputs of $p$, with values over petals determined uniquely by $p'$. This implies that the map $X'(p')\to\prod X'(id_{c'})$ is a bijection, and $X'$ is single-object Segal.

Let $Y'$ be a single-object Segal presheaf over $\Tw(P_{Pl(X)})$ and $Y$ be the corresponding presheaf over $\Tw(P)$. For operations $q'$ of arity $0$ in $P_{Pl(X)}$ the sets $Y'(q')$ are singletons, thus $Pl(X)=Pl(Y)$. By the same reasoning as in the first half of the proof, $Y$ is Segal.
\end{proof}

If an operad $P$ is palatable, then for any presheaf $X$ over $\Tw(P)$ Segal presheaf $Pl(X)$ is 2-Segal. The opposite is also true.

\begin{proposition}
Let $P$ be an operad such that presheaves $Pl(X)$ over $\Tw(P)$ (i.e.\@ right Kan extensions of presheaves over $\mathcal{P}_\psi$) are 2-Segal. Then $P$ is palatable. 
\end{proposition}
\begin{proof}
Let $X$ be a Segal presheaf. The corresponding presheaf $X'$ over $\Tw_{Pl(X)}$ is 2-Segal as a single-object Segal presheaf. The presheaf $X$ corresponds to the composition $\Tw(P_{Pl(X)})/X'\to \Tw(P_{Pl(X)})\to \Tw(P)$ of discrete fibrations induced by decomposition morphisms. Composition of decomposition morphisms is a decomposition morphism, i.e.\@ the presheaf $X$ is 2-Segal.
\end{proof}

\begin{proposition}
Let $P$ be a palatable operad and $X$ be a presheaf over $\Tw(P)$. Then the operad $P_{Pl(X)}$ is palatable. Presheaves over $\Tw(P)$ that correspond to Segal presheaves over $\Tw(P_{Pl(X)})$ are Segal. Let $Z'$ be a presheaf over $\Tw(P_{Pl(X)})$ and $Z$ be the corresponding presheaf over $\Tw(P)$. Then the presheaf $Pl(Z)$ corresponds to $Pl(Z')$. 
\end{proposition}
\begin{proof}
We start with the last claim. Let $Z_1$ be the presheaf over $\Tw(P)$ that corresponds to $Pl(Z')$. Pairs of presheaves $Pl(Z)$ and $Z$, $Z$ and $Z'$, $Z'$ and $Pl(Z')$, $Pl(Z')$ and $Z_1$ have the same petals (i.e.\@ the same sections over operations of grading $(-1)$), thus $Pl(Z)$ and $Z_1$ have the same petals. For any operation $p$ in $P$ the sets $Z_1(p)$ and $Pl(Z)(p)$ are computed by essentially the same limit over $\mathcal{P}_\psi/p$, thus presheaves $Z_1$ and $Pl(Z)$ are isomorphic.

For the first claim, for any presheaf $Z'$ the presheaf $Pl(Z')$ corresponds to the morphism of presheaves $Pl(Z)\to Pl(X)$, and since $P$ is palatable, this morphism is a morphism between 2-Segal presheaves, and thus corresponds to a decomposition morphism $P_{Pl(Z)}\to P_{Pl(X)}$. This implies that the presheaf $Pl(Z')$ is 2-Segal, and $P_{Pl(X)}$ is palatable.

For the second claim, let $Y''$ be the presheaf over $\Tw(P_{Pl(Y)})=\Tw((P_{Pl(X)})_{Pl(Y')})$ corresponding to a Segal presheaf $Y'$. Since $Y'$ is Segal, $Y''$ is single-object Segal, and $Y$ is Segal.
\end{proof}

To check if an operad is palatable we use the following notion.

\begin{definition}
  An operad $P$ endowed with a nice grading is \emph{strongly palatable} if for any composable operations $p$ and $q$ in $P$ the set of petals of $p\circ_i q$ is the pushout of the sets of petals of $p$ and $q$.
\end{definition}

\begin{example}
The operad $uCom$ is strongly palatable. The operad $mOp_{nc}$ and its graph-substitution suboperads, and the operads $cuAs$, $iuAs_{Tr}$ and $iuAs_{iTr}$, all graded by $\psi$, are strongly palatable: if a composition $p\circ_i q$ has arity greater than $0$, then the set of edges of $p\circ_i q$ is the pushout of the sets of edges of $p$, $q$ and $id_{c_i}$; if arity of $p\circ_i q$ is $0$, the condition can be checked directly. 
\end{example}

\begin{example}
Some of the graph-substitution operads are not strongly palatable, which is the reason why in general the corresponding presheaves $Pl(X)$ are not 2-Segal. In graph-substitution operads containing $\bigcirc$ and endowed with the arity grading $\phi$ the set of petals of the operation $\bigcirc=\nu\circ_1\mu_0$ is not the pushout of the sets of petals of $\mu_0$ and of $\nu$, since the map $id_\bigcirc$, a petal of $\bigcirc$, is not in the pushout. In the operads $mOp$ and $ciuAs$ endowed with grading $\psi$ the operation $\bigcirc$ has only one petal, while the pushout corresponding to the composition $\nu\circ_1\mu_0$ consists of two petals. 
\end{example}

\begin{proposition}
Let $P$ be an operad with a nice grading $\psi$. If $P$ is strongly palatable then $P$ is palatable. If the category $\mathcal{P}_\psi$ is a groupoid and $P$ is palatable, then $P$ is strongly palatable.
\end{proposition}
\begin{proof}
For any presheaf $X$ over $\Tw(P)$ the presheaf $Pl(X)$ is 2-Segal if and only if the following square is a pullback for all composable operations $p$ and $q$.
\begin{center}
\begin{tikzcd}
Pl(X)(p\circ_i q) \arrow[r] \arrow[d] & Pl(X)(p) \arrow[d] \\
Pl(X)(q) \arrow[r] & Pl(X)(id_c)
\end{tikzcd}
\end{center}
The above square is obtained by taking limits of the form $\lim_{s\to t}X(s)$ over the categories in the following commutative square.
\begin{center}
\begin{tikzcd}
\mathcal{P}_\psi/{id_c} \arrow[r] \arrow[d] & \mathcal{P}_\psi/p \arrow[d]  \\
\mathcal{P}_\psi/q \arrow[r]& \mathcal{P}_\psi/{p\circ_i q}
\end{tikzcd}
\end{center}
If $P$ is strongly palatable, then the lower square is a pushout, and by Lemma~\ref{lem:pushout-part2} the upper square is a pullback. For any presheaf $X$ over $\mathcal{P}_\psi$ the square of limits of the form $\lim_{s\to t} X(s)$ over the categories in the lower square is equivalent to the upper square. If $P_\psi$ is a groupoid and the upper square is always a pullback, then by Lemma~\ref{lem:pushout-part3} the lower square is a pushout.
\end{proof}

\begin{remark}
It is not clear whether the requirement above to form a groupoid is necessary. Notice that for operads satisfying the conditions of Theorem~\ref{thm:groupoid-reedy} that $P(\psi=0)$ or $P(1)$ is a groupoid, which is very likely satisfied by all operads of any interest, the category $\mathcal{P}_\psi$ is a groupoid. For this case the proposition gives a criterion of palatability.
\end{remark}

\begin{remark}
Let $P$ be a palatable operad such that $P(\psi=0)$ is a groupoid and petals factor through inputs in $\Tw(P)$. Let $X$ be a presheaf over $\Tw(P)$. An operation $p'$ in $P_{Pl(X)}$ over an operation $p$ of non-zero arity in $P$ corresponds to a corolla with vertex marked by $p$  and with leaves coloured by the petals of the input or the output of $p'$. Since petals factor through inputs, the correspondence between operations of non-zero arity in $P_{Pl(X)}$ and the corollas above is bijective. The operations of arity $0$ in $P_{Pl(X)}$ that can be composed with operations of arity greater than $1$ are determined by the petals of their outputs. For these operations composition in $P_{Pl(X)}$ works as composition in $P$. Palatability ensures that this composition is well-defined.  The only operations that do not fit this description are the operations of arity $0$ that can be composed only with operations of arity $1$.
\end{remark}

\subsection{The dendroidal category \texorpdfstring{$\Omega$}{Omega}}

\begin{proposition}
The category of Segal presheaves over $\Tw(sOp)$ is equivalent to the category of coloured operads.
\end{proposition}
\begin{proof}
Let $X$ be a presheaf over $\Tw(sOp)$. Then $sOp_{Pl(X)}$ is the operad $sOp_{X(\mu_0)}$ of $X(\mu_0)$-coloured operads: the petals of the $j$-th input of an element in $Pl(X)$ are the $X(\mu_0)$-colours of the edges adjacent to the $j$-th vertex in the corresponding operadic tree in $sOp_{X(\mu_0)}$.

A Segal presheaf $X$ over $\Tw(sOp)$ corresponds to a single-object Segal presheaf $X'$ over $\Tw(sOp_{X(\mu_0)})$, i.e.\@ to an $sOp_{X(\mu_0)}$-algebra, or to an $X(\mu_0)$-coloured operad. Any $X(\mu_0)$-coloured operad, or a single-object Segal $sOp_{Pl(X)}$-presheaf $Y'$, corresponds to a Segal presheaf $Y$ over $\Tw(sOp)$, and $X(\mu_0)=Y(\mu_0)$. Morphisms of Segal $sOp$-presheaves correspond to morphisms of coloured operads. 
\end{proof}

\begin{definition}
  Let $P$ be a palatable operad and $X$ be a Segal $P$-presheaf. The twisted arrow category $\Tw_P(X)$ of $X$ is the twisted arrow category of the corresponding $P_{Pl(X)}$-algebra $X'$.
\end{definition}

\begin{remark}
It is possible to define categories $\Tw_P(X)$ for Segal presheaves $X$ over non-palatable graph-substitution operads $P$,  since these operads satisfy palatability for all operations $p$ and $q$ except when $p\circ_i q=\bigcirc$. The operation $\bigcirc$ is not used in the construction of corresponding twisted arrow categories.
\end{remark}

\begin{remark}
If $P$ is additionally such that petals factor through inputs, then morphisms in $\Tw_P(X)$ are given by equivalence classes of expressions $p(x',x_1,\dots,x_n)$, where $p$ is an operation of $P$ and the elements $x',x_1,\dots,x_n$ from the sets $X(id_c)$  have petals compatible with respect to $p$. This observation is used in the following propositions.
\end{remark}

\begin{proposition}
The Yoneda embedding of $\Tw(sOp)$ consists of Segal presheaves.
\end{proposition}
\begin{proof}
For an operation $r$ in $sOp$ the presheaf $Hom(-,r)$ is Segal if and only if for any operation $p$ in $sOp$ of arity $n$ with input colours $c_j$ and for any collection of maps $g_j:id_{c_j}\to r$ in $\Tw(sOp)$ compatible with petals of $p$ there is unique map $f:p\to r$ in $\Tw(sOp)$ such that $g_j=f\circ in_j$. A map $id_{c_j}\to r$, or its upper vertex, corresponds to a choice of subtree of $r$ endowed with order on leaves, and its lower vertex corresponds to the tree obtained by contraction of the chosen subtree. 
An element in the limit in Segal condition for $Hom(-,r)$ over $p$ corresponds to a choice of $n$ subtrees in $r$, endowed with order on leaves of subtrees, such that for any petal of $p$ that factors through any two inputs of $p$ the corresponding leaves of the two subtrees $id_{c_j}\to r$ are the two different half-edges in the same internal edge of $r$. This implies that the $n$ subtrees do not intersect each other, and their union is a subtree of $r$. There is unique morphism $f:p\to r$ with upper vertices corresponding to the chosen subtrees of $r$, and with the lower vertex corresponding to the tree obtained by contraction of the chosen subtrees into a single vertex.
\end{proof}

\begin{proposition}
The twisted arrow category $\Tw(sOp)$ is equivalent to Moerdijk--Weiss dendroidal category $\Omega$.
\end{proposition}
\begin{proof}
For any Segal presheaf $X$ over $\Tw(sOp)$ the set $Hom(Hom(-,p),X)=X(p)$ is isomorphic to the set $\lim_{q\to p} X(q)=\lim_{q\to p}Hom(Hom(-,q),X)$ of tuples of maps from inputs of $p$ to $X$ that are compatible with respect to petals of $p$. Thus $Hom(-,p)$ is the free coloured operad generated by inputs of $p$. The category $\Omega$ and (the Yoneda embedding of) the category $\Tw(sOp)$ are both full subcategories of the category of coloured operads on free operads generated by trees.
\end{proof}

\begin{proposition}
The presheaves in the images of Yoneda embeddings of $\Tw(pOp)$ and of $\Tw(cOp)$ are Segal, but those of $\Tw(mOp)$ are in general not Segal presheaves.
\end{proposition}
\begin{proof}
For $\Tw(pOp)$ and $\Tw(cOp)$ the proofs are analogous to that of $\Tw(sOp)$. Let $p\in mOp$ be an operadic graph given by ``a line'' with three vertices of degree 4 and $r\in mOp$ be ``a circle'' with two vertices of degree $4$. If $Hom(-,r)$ is  Segal, then the morphisms from $Hom(-,p)$ to $Hom(-,r)$ correspond to morphisms from vertices of $p$ to $r$ that are compatible with respect to petals of $p$. However, there in no morphism in $\Tw(mOp)$ from $p$ to $r$ that sends the first and the third vertex of $p$ to the first vertex of $r$ and the second vertex of $p$ to the second vertex of $r$.
\end{proof}

\begin{proposition}
For any operad $P$ there is a sequence of categories 
\[ \Omega/P\to\Tw(P)\to(\Gamma^+_P)^{op}.\]
By restricting to a category $C$ we get the sequence
\[ \Delta/C\to\Tw(C)\to C^{op}\times C.\]
\end{proposition}
\begin{proof}
The functor $\Omega/P\to\Tw(P)$ is equivalent to the functor $\Tw(sOp)/P\to\Tw_{sOp}(P)$. The functor $\Delta/C\to\Tw(C)$ is equivalent to the functor $\Tw(uAs)/C\to\Tw_{uAs}(C)$. The map $\Tw(P)\to(\Gamma^+_P)^{op}$ is the map $\Tw(P)\to\mathcal{U}(P)$, where the category $\Gamma^+_P$ is the category defined in \cite{fresse2014functor}.  The map $\Tw(C)\to C^{op}\times C$ is the map $\Tw(C)\to\mathcal{U}(C)$.
\end{proof}

\begin{proposition}
For any operad $P$ and $P$-algebra $A$ the following square is a pullback:
\begin{center}
\begin{tikzcd}
\Omega/(P,A)\arrow[r] \arrow[d] & \Tw(P)/A \arrow[d]\\
\Omega/P\arrow[r] & \Tw(P)
\end{tikzcd}
\end{center}
\end{proposition} 
\begin{proof}
Here $(P,A)$ is the dendroidal nerve of the algebra $A$, and its elements are the trees in $P$ with edges labeled by elements of $A$, so that the labeling is compatible with the algebra structure of $A$. The map $\Omega/(P,A)\to \Tw(P)/A$ evaluates the trees, preserving the labeling of inputs. The map $\Omega/(P,A)\to \Omega/P$ forgets the labeling, and is a discrete fibration. The pullback property is easy to check.
\end{proof}

\section{Twisted arrow categories and \texorpdfstring{$\infty$}{infinity}-localizations}
\label{sec:final}

For any operad $P$ and $P$-algebra $A$ we introduce a strict 2-category $\mathbb{T}_P(A)$. The localization of $\mathbb{T}_P(A)$ by all 2-morphisms is the $(\infty, 1)$-localization of $\Tw(P)/A$ by the upper morphisms. The homotopy category of $\mathbb{T}_P(A)$ is equivalent to $\Tw_P(A)$. This shows that, up to equivalence of categories, $\Tw_P(A)$ is the localization of $\Tw(P)/A$ by the upper morphisms. The 2-category $\mathbb{T}_P(A)$ is similar to bicategories of correspondences.

For the model of $(\infty,1)$-localization we will use the hammock localization $L^H$. 

\begin{lemma}
\label{lem:infty-localization}
Let $C$ be a category endowed with:
\begin{itemize}
    \item a pair of subcategories $L$ and $R'$,
    \item a subset $S$ of objects of $C$, such that for any object $X$ in $C$ there exists unique morphism in $L$ from an object in $S$ to $X$,
    \item for any morphism $f$ in $C$, a choice of terminal $(R', L)$-factorization $l\circ r$ of $f$, such that morphisms of factorizations of $f$ into $l\circ r$ are in $L$; for any morphism $r$ in $R'$ the factorization is chosen to be $id\circ r$,
    \item for any morphisms $r$ in $R'$ and $l$ in $L$ with a common source, a choice of the pushout maps $r'$ in $R'$ and $l'$ in $L$, so that $r'\circ l= l'\circ r$,
    \item and such that for any morphisms $f$ in $C$ and $l$ in $L$, if $f\circ l$ is in $L$, then $f$ is in $L$.
\end{itemize}
Then for any objects $s$ and $s'$ in $S$  the nerve of the category of $(R', L)$-cospans (with morphisms in $L$) from $s$ to $s'$ is a deformation retract of the simplicial set $L^H(C, L)(s, s')$.
\end{lemma}
\begin{proof}
We will deform the simplicial set $L^H(C,L)(s,s')$ in several steps. Due to the nature of hammock localization, it suffices to describe the deformation retraction on 1-simplices. The 1-simplices are represented by diagrams in which the horizontal arrows pointing left and the vertical arrows pointing down belong to the subcategory $L$. The deformation corresponds to a sequence of morphisms of diagrams. At the first step a simplex
\begin{center}
\begin{tikzcd}
{} \arrow[r, dotted] & {A_{0,i}} \arrow[d] & {B_{0,i}} \arrow[l] \arrow[r] \arrow[d] & {A_{0,i+1}} \arrow[d] & {} \arrow[l, dotted] \\
{} \arrow[r, dotted] & {A_{1,i}}           & {B_{1,i}} \arrow[r] \arrow[l]           & {A_{1,i+1}}           & {} \arrow[l, dotted]
\end{tikzcd}
\end{center}
is deformed into the 1-simplex
\begin{center}
\begin{tikzcd}
{} \arrow[r, dotted] & {A_{0,i}} \arrow[d] & b_i \arrow[l] \arrow[r] \arrow[d,"id"'] & {A_{0,i+1}} \arrow[d] & {} \arrow[l, dotted] \\
{} \arrow[r, dotted] & {A_{1,i}}           & b_i \arrow[r] \arrow[l]           & {A_{1,i+1}}           & {} \arrow[l, dotted]
\end{tikzcd}
\end{center}
via the morphisms of diagrams induced by the unique morphisms $b_i\to B_{k,i}$ in $L$ from objects $b_i$ in $S$ and the identity morphisms on $A_{k,i}$. This simplex is deformed into 1-simplex
\begin{center}
\begin{tikzcd}
{} \arrow[r, dotted] & {A'_{0,i}} \arrow[d] & b_i \arrow[l] \arrow[r,"R'"] \arrow[d,"id"'] & {A'_{0,i+1}} \arrow[d] & {} \arrow[l, dotted] \\
{} \arrow[r, dotted] & {A'_{1,i}}           & b_i \arrow[r,"R'"] \arrow[l]           & {A'_{1,i+1}}           & {} \arrow[l, dotted]
\end{tikzcd}
\end{center}
via the morphisms in $L$ from the canonical $(R',L)$-factorizations $b_i\to A'_{k,i+1}\to A_{k,i+1}$. The morphism $A'_{0,i+1}\to A'_{1,i+1}$ is the morphism in $L$ from the $(R',L)$-factorization $b_i\to (A'_{0,i+1}\to A_{0,i+1}\to A_{1,i+1})$ to the terminal $(R', L)$-factorization $b_i\to A'_{1,i+1}\to A_{1,i+1}$. The morphisms $b_i\to A'_{k,i}$ are determined uniquely. 

If the rows in the diagram that represents the obtained 1-simplex start with morphisms pointing left, then, since $s$ is in $S$, these morphisms are identity morphisms, and we replace this diagram with the diagram in which the first morphisms point to the right. If the last morphisms in the rows point right, then we add identity morphisms pointing left. Finally, the obtained simplex is deformed into the 1-simplex 
\begin{center}
\begin{tikzcd}
{} \arrow[r, dotted] & {A''_{0,i}} \arrow[d] & {A''_{0,i}} \arrow[l, "id"'] \arrow[r, "R'"] \arrow[d] & {A''_{0,i+1}} \arrow[d] & {} \arrow[l, dotted] \\
{} \arrow[r, dotted] & {A''_{1,i}}           & {A''_{1,i}} \arrow[r, "R'"] \arrow[l, "id"']           & {A''_{1,i+1}}           & {} \arrow[l, dotted]
\end{tikzcd}
\end{center}
via the morphisms $A'_{k,i}\to A''_{k,i}$ constructed by induction, and the morphisms $b_i\to A''_{k,i}$ being the obvious composition. The first maps $A'_{k,0}\to A''_{k,0}$ are the identity maps. Once the map $A'_{k,i}\to A''_{k,i}$ is defined, the map $b_i\to A''_{k,i}$ is also defined. The maps $A'_{k,i+1}\to A''_{k,i+1}$ in $L$ and $A''_{k,i}\to A''_{k,i+1}$ in $R'$ constructed at the next step are the maps in the canonical pushout of $b_i\to A'_{k,i+1}$ and $b_i\to A''_{k,i}$. The map $A''_{0,i+1}\to A''_{1,i+1}$ is the unique map from the pushout $A''_{0,i+1}$, induced by the maps $A'_{0,i+1}\to A'_{1,i+1}\to A''_{1,i+1}$ and $A''_{0,i}\to A''_{1,i}\to A''_{1,i+1}$. And since the composition  $A'_{0,i+1}\to A''_{0,i+1}\to A''_{1,i+1}$ is equal to $A'_{0,i+1}\to A'_{1,i+1}\to A''_{1,i+1}$, which is in $L$, and $A'_{0,i+1}\to A''_{0,i+1}$ is in $L$, the map $A''_{0,i+1}\to A''_{1,i+1}$ is in $L$. 

Since the left arrows except the last one are trivial, the final simplex is represented by the diagram of the following form.
\begin{center}
\begin{tikzcd}
s \arrow[r, "R'"] \arrow[d, "id"] & {A''_{0,n}} \arrow[d, "L"] & s' \arrow[l, "L"'] \arrow[d, "id"] \\
s \arrow[r, "R'"]                 & {A''_{1,n}}                & s' \arrow[l, "L"']                
\end{tikzcd}    
\end{center}
This simplex does not change during the deformation.
\end{proof}

\begin{remark}
Let $t$ and $t'$ be objects in $C$, and $s$ and $s'$ be the objects in $S$ corresponding to $t$ and $t'$. The morphisms $s\to t$ and $s'\to t'$ in $L$ induce homotopy equivalence between the simplicial sets $L^H(C,L)(t,t')$ and $L^H(C,L)(s,s')$. Thus the theorem describes the homotopy type of all simplicial sets $L^H(C,L)(t,t')$.
\end{remark}

\begin{definition}
  Let $P$ be an operad and $A$ be a $P$-algebra. The strict 2-category $\mathbb{T}_P(A)$ has the elements of $A$ as objects. For any elements $s$ and $s'$ in $A$ the category $Hom_{\mathbb{T}_P(A)}(s, s')$ is the category of $(Lower', Upper)$-cospans from $id_c/s$ to $id_{c'}/s'$ in $\Tw(P)/A$. Notice that these cospans are determined uniquely by their $Lower'$ part, and that morphisms between these cospans are upper morphisms with first vertex trivial, with index above the first vertex equal to $1$. The horizontal composition of cospans corresponding to $Lower'$ morphisms $r_1$ and $r_2$ is the cospan corresponding to the $Lower'$ morphism $r$, where the lower vertex of $r$ is the composition of the lower vertices of $r_1$ and $r_2$. The horizontal composition of upper morphisms $f_1$ and $f_2$ is the upper morphism $f$, with the first vertex trivial, followed by the second to the last vertices of $f_1$, followed by the second to the last vertices of $f_2$; the permutation on leaves is the obvious permutation in $S_1\times S_{n-1}\times S_{m-1}\subset S_{n+m-1}$. The composition is associative.
  
  The strict locally groupoidal $(2, 1)$-category $\mathbb{T}'_P(A)$ is the strict $2$-category obtained from $\mathbb{T}_P(A)$ by groupoidification of categories of morphisms.
\end{definition}

\begin{theorem}
\label{thm:infty-localization}
Let $P$ be an operad and $A$ be a $P$-algebra. The $(2, 1)$-category $\mathbb{T}'_P(A)$ is the $(\infty, 1)$-localization of $\Tw(P)/A$ by the upper morphisms.
\end{theorem}
\begin{proof}
Let $S$ be the set of objects of $\Tw(P)/A$ of the form $id_c/a$, and the categories $L$ and $R'$ be the categories $Upper$ and $Lower'$. Lemma~\ref{lem:infty-localization} holds: its third and fourth assumptions have been proved in Lemma~\ref{lem:lower-upper-factor} and Lemma~\ref{lem:upper-lower-pushouts}, and the remaining assumptions hold trivially. The $2$-category $\mathbb{T}'_P(A)$ is equivalent to $L^H(\Tw(P)/A, Upper)$.
\end{proof}

\begin{proposition}
\label{prp:canonical-decomp-initial-objects}
An operad $P$ is canonically decomposable if and only if for any $P$-algebra $A$ \emph{connected components} of Hom-categories in $\mathbb{T}_P(A)$ have initial objects.
\end{proposition}

The proof is trivial. The next corollary implies \textbf{Theorem 2} of \cite{walde20172}, \textbf{Theorem 1.1} of \cite{deBrito20dendroidal} and the asphericity of $\Omega$, proved in \cite{ara2019dendroidal}.

\begin{corollary}
For any canonically decomposable operad $P$ the functor $\Tw(P)/A\to \Tw_P(A)$ is the $(\infty, 1)$-localization by the upper morphisms. In particular, the functor $\Omega/P\to\Tw(P)$ is the $(\infty, 1)$-localization by boundary preserving morphisms. In particular, the functor $\Omega/sOp\to\Omega$, or the functor $\Tw(sOp)/sOp\to\Tw(sOp)$, is $(\infty,1)$-localization.
\end{corollary}

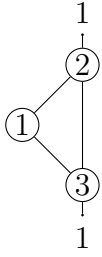
\begin{figure}[h]
\centering
\begin{tikzpicture}[scale=0.4]
\foreach \place/\name/\label in {{(0,0)/a/$1$}, {(2,2)/b/$2$}, {(2,-2)/c/$3$}}
    \node[circ] (\name) at \place {\label};
  \node[leaf,label=$1$] (b1) at (2,3) {};
  \node[leaf,label=below:$1$] (c1) at (2,-3) {};
  \draw (a) to (b);
  \draw (b) to (c);
  \draw (a) to (c);
  \draw (b) to (b1);
  \draw (c) to (c1);
\end{tikzpicture}
\caption{The element of the operad of properads used in Example~\ref{ex:properads-not-decomposable}.}
\label{fig:properad-counterexample}
\end{figure}

In contrast, the following example shows that twisted arrow categories of PROPs and of properads in general are not $(\infty,1)$-localizations of the corresponding slice categories of Segal presheaves.

\begin{example}
\label{ex:properads-not-decomposable}
Let $diOp$ be the operad of dioperads and $prOp$ be the operad of properads or the operad of PROPs. Let $B$ be the dioperad without operations of input arity $0$, without operations of output arity $0$, and with exactly one operation in $B(m;n)$ for any positive input arity $m$ and positive output arity $n$. Let $B'$ be the properadic or the PROP envelope of $B$, i.e.\@ the extension of the $diOp$-algebra $B$ to $prOp$-algebra $B'$ induced by the morphism of operads $diOp\to prOp$. Let $A$ be the properad or the PROP $B'\sqcup F(x)/{\sim}$, where $x$ is an operation of input and output arity $1$, and where relations encode that any composition of $x$ with any operation $p$ in $B'$ is equal to $p$. Consider the 1-morphism $f:id\to b'$ in $\mathbb{T}_{prOp}(A)$ represented by the operation in $prOp$ from Figure~\ref{fig:properad-counterexample} and by the three elements from $A$ that are of appropriate arity and that belong to the dioperad $B$. We will show that the connected component $\mathcal{C}$ of the object $f$ in the category $Hom_{\mathbb{T}_{prOp}(A)}(id,b')$ is not contractible. This will imply that $\Tw(prOp)/A\to\Tw_{prOp}(A)$ is not $(\infty,1)$-localization.

The 1-morphisms that are objects of $\mathcal{C}$ are represented by planar graphs that can be obtained from the planar graph of $f$ by adding vertices of input and output arity $1$, these vertices necessarily being marked either by identity operation or by $x^k$ for some $k$.  Consider the full subcategory $\mathcal{C}'$ of $\mathcal{C}$ on 1-morphisms such that all of the additional vertices are marked by $x$. The category $\mathcal{C}'$ is a reflective subcategory of $\mathcal{C}$, with the unit of adjunction being the 2-morphism that removes the vertices marked by identity operations and replaces the vertices marked by $x^k$ with $k$ vertices marked by $x$. Let $\mathcal{C}''$ be the subcategory of $\mathcal{C}'$ on 1-morphisms with all additional vertices located between the second and the third vertex of the triangle. The category $\mathcal{C}''$ is a coreflective subcategory of $\mathcal{C}'$. The category $\mathcal{C}''$ is isomorphic to the category $\mathcal{D}$ whose objects are natural numbers and whose morphisms $(a,b):n\to n+a+b$ correspond to pairs of natural numbers, with composition given by addition. The spaces $\mathcal{BC}$, $\mathcal{BC}''$ and $\mathcal{BD}$ are homotopy equivalent. The space $\mathcal{BD}$ is not contractible.
\end{example}

One can easily come up with a definition of a Segal presheaf over twisted arrow category of a PROP. The above example shows that the analogous naive notion of a \emph{weak} Segal presheaf over a twisted arrow category of a PROP is in general not the correct one. This problem may appear in practice as follows.

\begin{example} 
To generalize the present work further to the case of $\infty$-operads one needs the notion of $\infty$-category endowed with a factorization system. There is a coloured PROP $FactSys$ whose algebras are monoids endowed with a strict factorization system.  Segal presheaves over $\Tw(FactSys)$ should correspond to small categories endowed with a strict factorization system. Yet it is not immediately clear if weak Segal presheaves or some other nice presheaves over $\Tw(FactSys)$ correctly encode $\infty$-categories endowed with a factorization system, and in particular if $\Tw(prOp)/FactSys\to\Tw(FactSys)$ is $\infty$-localization. At least one can expect that $\infty$-categories endowed with a factorization system are correctly modeled by nice presheaves over $\Tw(prOp)/FactSys$.
\end{example}

Finally, since twisted arrow categories of planar operads are equivalent to twisted arrow categories of their symmetrizations, one may ask if analogous equivalence holds while passing from symmetric operads to more general structures that encode algebras. Cartesian operads, better known as multi-sorted Lawvere theories, are the simplest among structures that generalize symmetric operads, and the present theory can be generalized to some extent to the case of sufficiently nice cartesian operads. However passage from symmetric operads to cartesian operads does not induce equivalence of corresponding twisted arrow categories. 

\appendix
\section{Category theory}
\label{sec:appendix-cat-th}
\subsection{Strict factorization systems}
A strict factorization system is a refinement of an orthogonal factorization system.

\begin{definition}
An orthogonal factorization system for a category $C$ is a pair $(L, R)$ of subcategories of $C$ such that both $L$ and $R$ contain all isomorphisms of $C$, any morphism in $C$ can be factored as $r\circ l,\, l\in L, r\in R$, and this factorization is unique up to unique isomorphism of factorizations.
\end{definition}

\begin{definition}
  A strict factorization system for a category $C$ is a pair $(sL, sR)$ of subcategories of $C$ such that for any morphism in $C$ there is unique factorization $r\circ l,\, l\in sL, r\in sR$.
\end{definition}

A pair of wide subcategories $(L, R)$ of a category $C$ is a strict factorization system if for any morphism $f$ in $C$ the category of $(L, R)$-factorizations of $f$ is a point, and is an orthogonal factorization system if for any morphism $f$ in $C$ the category of $(L, R)$-factorizations is a contractible groupoid. 

In the rest of this appendix we denote by $I$ the class of all isomorphisms of a category $C$. Morphisms that contain $i$, $l$ or $r$ in their name are the morphisms from $I$, $sL$ or $sR$ respectively. In particular, $i_l
$ is in $I\cap sL$ and $i_r$ is in $I\cap sR$.

\begin{proposition}
\label{pr:ufs-are-ofs}
  Let $(sL, sR)$ be a strict factorization system for a category $C$. Then $((I\cap sR)\circ sL, sR\circ (I\cap sL))$ is an orthogonal factorization system for $C$.
\end{proposition}

We prove this proposition by a sequence of lemmas. We implicitly use the fact that any morphism is representable as composition $r'\circ l'$ for unique $l'\in sL$ and $r'\in sR$.

\begin{lemma}
For any composable $i\in I$ and $l\in sL$ let $r' \circ l' = l\circ i$ be the $(sL,sR)$-factorization of their composition. Then $r'$ is an isomorphism.
\end{lemma}
\begin{proof}
We have $l = l \circ i \circ i^{-1} = r' \circ l' \circ i^{-1} = r' \circ r'' \circ l''$, where $l'\circ i^{-1}=r''\circ l''$, thus $l=l''$ and $r'\circ r'' = id$. Similarly,
$l' = l' \circ i^{-1} \circ i = r'' \circ l'' \circ i = r'' \circ l \circ i = r'' \circ r' \circ l'$, thus $r''\circ r' = id$, and $r'$ is an isomorphism.
\end{proof}

\begin{lemma}
For any composable $r\in sR$ and $i\in I$ let $r' \circ l' = i \circ r$ be the $(sL,sR)$-factorization of their composition. Then $l'$ is an isomorphism.
\end{lemma}
\begin{proof}
We have $r=i^{-1}\circ i\circ r=i^{-1}\circ r'\circ l'= r''\circ l''\circ l'$, where $i^{-1}\circ r'=r''\circ l''$, thus $r=r''$ and $l''\circ l'=id$. Similarly, $r'=i\circ i^{-1}\circ r'=i\circ r''\circ l''=i\circ r\circ l''=r'\circ l'\circ l''$, thus $l'\circ l''=id$, and $l'$ is an isomorphism.
\end{proof}

\begin{lemma}
For any isomorphism $i$ let $r \circ l=i$ be the $(sL,sR)$-factorization of $i$. Then the morphisms $l$ and $r$ are isomorphisms.
\end{lemma}
\begin{proof}
We have $id = i^{-1} \circ r \circ l = r' \circ l' \circ l$. By preceding lemma $l'$ is an isomorphism. Since $l'\circ l= id$, the morphism $l$ is an isomorphism, and $r$ is an isomorphism.
\end{proof}

\begin{corollary}
The classes of morphisms $I\circ sL$ and $sR\circ I$ are closed under composition, and $I \circ sL = (sR\cap I) \circ sL$, $sR\circ I = sR\circ (I\cap sL)$.
\end{corollary}

\begin{lemma}
For any morphism $f$ in $C$ the groupoid of factorizations of  $f$ via $(I\circ sL, sR \circ I)$ is connected.
\end{lemma}
\begin{proof}
Let $r\circ i_l \circ i_r \circ l$ be a factorization of the morphism $f$ via $(I\circ sL, sR \circ I)$. By preceding lemma $i_l\circ i_r = i'_r\circ i'_l$. Let $l'= i'_l\circ l$, $r' = r\circ i'_r$, $i = i'_l\circ i_r^{-1}$. Then $r'\circ l'$ is the unique factorization of $f$ via $(sL, sR)$, and $i$ is the isomorphism from the original factorization to $r'\circ l'$.
\end{proof}

\begin{lemma}
The automorphism group of any factorization via $(I\circ sL, sR \circ I)$ is trivial.
\end{lemma}
\begin{proof}
The groupoid of factorizations of a morphism is connected, thus automorphism groups of factorizations of the same morphism are isomorphic. Let $i=i_r\circ i_l$ be an automorphism of the unique $(sL,sR)$-factorization $r\circ l$ of a morphism. Then $r = r\circ i_r \circ i_l$ and $l = i_r\circ i_l \circ l$, and $i_l = i_r = id$.
\end{proof}
We have proved the proposition. 

\paragraph{Ternary factorization systems.}
If $(L_1,R_1)$ and $(L_2,R_2)$ are orthogonal factorization systems on the same category, then $L_1\subseteq L_2$ implies $R_2\subseteq R_1$. The analogous statement for strict factorization systems is false.

\begin{example}
Let $C$ be a category with two objects $A$ and $B$ and with three non-trivial morphisms: automorphism $i$ of $A$, and two morphisms $f,g:A\to B$; and relations $i^2=id$, $f\circ i=g$, $g\circ i = f$. Let $L$, $R_1$ and $R_2$ be the wide subcategories of $C$ with exactly one non-trivial morphism, this morphism being $i$, $f$ and $g$ respectively. Then $(L,R_1)$ and $(L, R_2)$ are strict factorization systems. Notice that these systems generate the same orthogonal factorization system.
\end{example}

\begin{definition}
  A ternary strict factorization system for a category $C$ is a pair of strict factorization systems $(L_1,R_1)$ and $(L_2,R_2)$ for  $C$ such that $L_1\subseteq L_2$ and $R_2\subseteq R_1$.
\end{definition}

\begin{proposition}
Let a pair $(L_1,R_1)$ and $(L_2,R_2)$ of strict factorization systems be a ternary strict factorization system for a category $C$. Then any morphism of $C$ can be decomposed uniquely as $r\circ m \circ l$, with $l\in L_1$, $m\in L_2\cap R_1$, and $r\in R_2$.
\end{proposition}
\begin{proof}
Let $f$ be a morphism in $C$, and $r_1\circ l_1$ be the decomposition of $f$ via $(L_1,R_1)$. Let $r_2\circ m$ be the decomposition of $r_1$ via $(L_2, R_2)$. Let $r_3\circ l_3$ be the decomposition of $m$ via $(L_1, R_1)$. Since $r_1=r_2\circ r_3\circ l_3$, $r_1\in R_1$, $r_2\circ r_3\in R_1$, $l_3\in L_1$, we have $l_3=id$, and $m=r_3\circ l_3=r_3\in R_1$. Thus $r_2\circ m\circ l_1$ is a decomposition of $f$ of the desired form.

The morphism $m\circ l_1$ is in $L_2$ and $r_2$ is in $R_2$, thus $m\circ l_1$ and $r_2$ are determined uniquely by $f$. Similarly, $r_2\circ m$ and $l_1$ are determined uniquely by $f$. The morphism $m$ is in $R_1$, and is determined uniquely by $m\circ l_1$, and thus $m$ is determined uniquely by $f$.
\end{proof}

Uniqueness of ternary decomposition in a category does not imply that the corresponding structure is a ternary factorization system, as the following example shows.

\begin{example}
Let $C$ be a category with five objects $A$, $B$, $C$, $X$ and $Y$, generated by morphisms $r:A\to B$, $l:B\to C$, $l':A\to X$, $m':X\to Y$ and $r':Y\to C$, with relation $l\circ r=r'\circ m'\circ l'$. The classes of morphisms $L$, $M$, $R$ containing all identity morphisms together with $l$ and $l'$; $m'$; $r$ and $r'$ respectively are wide subcategories of $C$. Any morphism decomposes uniquely via $(L,M,R)$. These do not correspond to a ternary factorization system: the class $M\circ L$, which should be a category $L_2$, is not a category.
\end{example}

\paragraph{$(sR,sL)$-factorizations.} We will use the following lemma on $(sR, sL)$-factorizations. 
\begin{lemma}
\label{lem:fact-cat-inherit}
Let $C$ be a category endowed with a strict factorization system $(sL, sR)$. For any morphism $f$ in $C$ the category of $(sR,sL)$-factorizations of $f$ has a strict factorization system induced by $(sL,sR)$.
\end{lemma}
\begin{proof}
Let $l'\circ r'$ and $l''\circ r''$ be two $(sR, sL)$-factorizations of $f$, and $r\circ l$ be a morphism from the first to the second factorization.
The morphism $r''=r\circ l\circ r'$ is in $sR$, thus $l\circ r'$ is in $sR$. The morphism $l'=l''\circ r\circ l$ is in $sL$, thus $l''\circ r$ is in $sL$. The morphisms $l$ and $r$ are morphisms of $(sR,sL)$-factorizations.
\end{proof}

\subsection{Pushouts of special form}

\begin{lemma}
\label{lem:pushouts-of-categories}
Let $\mathcal{C}$ be a subcategory of a small category $\mathcal{B}$, and $A\to B$ and $A\to C$ be morphisms in $\mathcal{B}$. The functors $\mathcal{C}/A\to\mathcal{C}/B$ and $\mathcal{C}/A\to\mathcal{C}/C$ have a pushout $\mathcal{D}$. The square
\begin{center}
\begin{tikzcd}
\mathcal{C}/A \arrow[r] \arrow[d] & \mathcal{C}/B \arrow[d]\\
\mathcal{C}/C \arrow[r] & \mathcal{D}
\end{tikzcd}
\end{center}
is a pushout if and only if it is a pushout on the sets of objects and the two functors to $\mathcal{D}$ are discrete fibrations.
\end{lemma}
\begin{proof}
The pushout $\mathcal{D}$ of the pair of functors $\mathcal{C}/A\to\mathcal{C}/B$ and $\mathcal{C}/A\to\mathcal{C}/C$ is constructed as follows. The set of objects of $\mathcal{D}$ is the pushout of the sets of objects of the corresponding categories. Notice that if morphisms $s\to t$ and $s'\to t'$ in $\mathcal{B}$ represent the same object in $\mathcal{D}$, with $t$ and $t'$ being in $\{B, C\}$, then $s=s'$. The set $Hom_{\mathcal{D}}(-,[s\to t])$ is defined to be the set of objects of $\mathcal{C}/s$. The set $Hom_{\mathcal{D}}(-,[s\to t])$ is isomorphic to the set $Hom_{\mathcal{C}/t}(-,s\to t)$, which allows to define sources and composition of morphisms in $\mathcal{D}$, and these are well-defined. In all categories in the square the set of morphisms with a target $s\to t$ or $[s\to t]$ is the set $\mathcal{C}/s$, and the four functors of the square are discrete fibrations.

In the opposite direction, assume that all the four functors are discrete fibrations, and that the square is a pushout on the sets of objects. Let $\mathcal{E}$ be a category and $F_B:\mathcal{C}/B\to\mathcal{E}$ and $F_C:\mathcal{C}/C\to\mathcal{E}$ be a pair of compatible functors. There is unique functor $F_D:\mathcal{D}\to\mathcal{E}$ compatible with $F_B$ and $F_C$. On objects of $\mathcal{D}$ the functor $F_D$ is defined uniquely by the pushout condition. Let $g$ be a morphism in $\mathcal{D}$. Since the set of objects of $\mathcal{D}$ is the pushout of sets of objects of $\mathcal{C}/B$ and $\mathcal{C}/C$, $g$ has a lift $g_B$ to $\mathcal{C}/B$ or a lift $g_C$ to $\mathcal{C}/C$. The morphism $F_D(g)$ then has to be equal to $F_B(g_B)$ or to $F_C(g_C)$. If $g_B$ or $g_C$ has a lift $g_A$ to $\mathcal{C}/A$, then $F_B(g_B)$ or $F_C(g_C)$ is equal to $F_A(g_A)$. Again by pushout and discrete opfibration properties, any two lifts of $g$ are connected by zig-zag via functors $\mathcal{C}/A\to\mathcal{C}/B$ and $\mathcal{C}/A\to\mathcal{C}/C$. This implies that $F_D(g)$ is well-defined, and that the square is a pushout.
\end{proof}

\begin{lemma}
\label{lem:pushout-part2}
Continuing Lemma~\ref{lem:pushouts-of-categories}, with $\mathcal{D}$ as the pushout, let $Y$ be a presheaf on $\mathcal{C}$. Define presheaves $Y'$ on the four categories in the pushout square by $Y'(s\to t)=Y(s)$, with obvious maps. The corresponding square of limits of the diagrams $Y'$ is a pullback square.
\begin{center}
\begin{tikzcd}
L_D\arrow[r] \arrow[d] & L_B \arrow[d] \\
L_C\arrow[r] & L_A
\end{tikzcd}
\end{center}
\end{lemma}
\begin{proof}
In general for any diagram $F:\mathcal{I}\to\mathcal{E}$ in some category $\mathcal{E}$ with appropriate limits and any functor $\mathcal{J}\to\mathcal{I}$ there is a canonical map from the limit of $F$ to the limit of the restriction of $F$ to $\mathcal{J}$. The maps in the square are of this form. 

Let $E\to L_B$ and $E\to L_C$ be two compatible maps. Composition of these maps with the projection maps $L_B\to Y(s)$ and $L_C\to Y(s)$ gives well-defined maps $E\to Y(s)$ indexed by the objects $[s\to t]$ of $\mathcal{D}$. These maps are compatible with morphisms in $\mathcal{D}$: any two lifts of the same morphism $[s\to s'\to t]$ in $\mathcal{D}$ give the same map $Y(s')\to Y(s)$. We have constructed the map from $E$ to the diagram indexed by $\mathcal{D}$, which gives the unique map $E\to L_D$.
\end{proof}

\begin{lemma}
\label{lem:pushout-part3}
Continuing Lemma~\ref{lem:pushouts-of-categories}, with $\mathcal{D}$ as the pushout, let additionally $D'$ be an object of $\mathcal{B}$ with maps $B\to D'$ and $C\to D'$, and let $\mathcal{D}\to\mathcal{C}/D'$ be the map of the pushout. If $\mathcal{C}$ is a groupoid and for any presheaf $Y$ over $\mathcal{C}$ the canonical map from the limit $L_{D'}$ to the pullback and limit $L_D$ is an isomorphism, then the functor $\mathcal{D}\to\mathcal{C}/D'$ is an isomorphism.
\end{lemma}
\begin{proof}
In general, the limit of a diagram $F:\mathcal{G}\to Sets$ indexed by a groupoid $\mathcal{G}$ is the product $\prod_{[G]\in\pi_0(\mathcal{G})} F(G)^{Aut(G)}$ indexed by the set $\pi_0(\mathcal{G})$ of connected components of $\mathcal{G}$, where the set $F(G)^{Aut(G)}$ is the set of fixed points of the action of the automorphism group of a representative $G$ of a connected component $[G]$.

Since $\mathcal{C}$ is a groupoid, the categories $\mathcal{C}/A, \dots, \mathcal{C}/D', \mathcal{D}$ are groupoids. Let $Y$ be the presheaf on $\mathcal{C}$ with $Y(s)=\{0,1\}$ for all $s$ in $\mathcal{C}$, with maps being trivial (the corresponding limits $L_A,\dots, L_{D'}, L_D$ can be seen as the $0$-th cohomology with $\mathbb{Z}/2$ coefficients of the categories $\mathcal{C}/A,\dots,\mathcal{C}/D', \mathcal{D}$). Since $L_{D'}\to L_D$ is an isomorphism, the functor $\mathcal{D}\to\mathcal{C}/D'$ induces isomorphism on the sets of connected components, and thus, as a discrete fibration of groupoids, this functor is surjective on objects.

For any object $s$ in $\mathcal{C}$ let $Y$ be the presheaf on $\mathcal{C}$ with $Y(s')=Hom_{Set}(Hom_{\mathcal{C}}(s,s'),\mathbb{Z})$. For any object $s'$ in $\mathcal{C}$ that is not in the connected component $[s]$ the set $Y(s')$ is a singleton, thus the factors $Y'(s'\to t)^{Aut(s')}$ in the products that compute the limits $L_A,\dots,L_{D'},L_D$ are singletons. The remaining factors in these products are indexed by connected components that contain the objects of $\mathcal{C}/A,\dots,\mathcal{C}/D', \mathcal{D}$ of the form $s\to A, \dots, s\to D'$, or $[s\to t]$ respectively. A factor indexed by $[s\to t]$ is isomorphic to $Hom_{Set}(Aut_{\mathcal{C}}(s)/Aut(s\to t), \mathbb{Z})$. We have already proved that the indexing sets of $L_D$ and $L_{D'}$ are isomorphic. Since the map $L_{D'}\to L_{D}$ is an isomorphism, the functor $\mathcal{D}\to\mathcal{C}/D'$ induces isomorphism on the groups $Aut(s\to t)$. This implies that the functor $\mathcal{D}\to\mathcal{C}/D'$ is injective, and thus bijective, on objects, and is an isomorphism.
\end{proof}

\section{EZ-categories}
\label{sec:ez-cats}

There are two classes of particularly nice generalized Reedy categories: EZ-categories and Eilenberg--Zilber categories. Here we show how to check if the twisted arrow category or the enveloping category of a reasonable operad belongs to these classes. 

\begin{definition}
A generalized Reedy category is EZ-category (\cite{berger2011extension}) if $R_+$ is the set of all monomorphisms, $R_-$ is the set of all split epimorphisms, and any pair of split epimorphisms with common source has an absolute pushout.
\end{definition}

\begin{definition}
A generalized Reedy category is Eilenberg--Zilber category (\cite{moerdijk2016minimal}) if $R_-$ is the set of all split epimorphisms, and any two split epimorphisms that have the same set of sections are equal.
\end{definition}

\begin{proposition}
Under conditions of Theorem~\ref{thm:groupoid-reedy} (or Theorem~\ref{thm:enveloping-image-reedy}) any split epimorphism in $\Tw(P)$ (or in $\mathcal{V}(P)$) is in $R_-$.
\end{proposition}
\begin{proof}
By Lemma~\ref{lem:splitepi} non-source vertices of a split epimorphism $e$ have arity at most $1$ and grading at most $0$. The conditions imply that non-source vertices of $e$ of arity $1$ are marked by isomorphisms, i.e.\@ that $e$ is in $R_-$.
\end{proof}

For any operad $P$ petal maps in $\Tw(P)$ are related to maps that we call \emph{elementary degeneracies}. Elementary degeneracies in $\Delta$ and $\Omega$ in the usual sense are precisely the maps  isomorphic (but not necessarily equal) to elementary degeneracies in the following sense. 

\begin{definition}
  Let $P$ be a planar, symmetric or cyclic operad with a nice grading $\psi$. An elementary degeneracy in $\Tw(P)$ is a morphism in $R_{\psi=-1}$ such that exactly one non-source vertex is marked by a non-trivial operation.
\end{definition}

Petal maps and elementary degeneracies are determined by a similar data. An elementary degeneracy $(p)'\circ_i q$ is determined by its source $p$, the operation $q$ and the index $i$ of the leaf of $p$ to which the operation $q$ is grafted. In planar case petal maps $p\circ_i (q)'$ are determined by the very same data. In symmetric and cyclic cases this data is determined up to the action of a cyclic group on the index $i$ and on the orbit of $p$. In canonical representatives of morphisms the index $i$ is assumed to be equal to $1$, but we will use non-canonical representatives of morphisms.

For any symmetric operad $P$ let $f:q\to t$ be a petal map. The morphism $f$ can be factored as $q\to id_c\to t$, where $id_c\to t$ is an input map, if and only if there are operations $r$ and $p$ and an index $i$ such that $p\circ_i(q)'$ is a representation of $f$, the operation $q$ is a left unit for the operation $r$, $t=(p\circ_i q)$ and $p=t\circ_i r = (p\circ_i q)\circ_i r$. The planar case differs only in that $q$ might be not a left, but a right unit for $r$, so that $p=(p\circ_i q)\circ_{i-1} r$.

\begin{proposition}
\label{pr:elem-deg-split-epi}
Under conditions of Theorem~\ref{thm:groupoid-reedy} (or Theorem~\ref{thm:enveloping-image-reedy}) suppose additionally that petals factor through inputs in $\Tw P$. Then any morphism in $R_-$ is a split epimorphism if (and, trivially, only if) elementary degeneracies from operations of arity $1$ (or their images in $\mathcal{V}(P)$) are split epimorphisms.
\end{proposition}
\begin{proof}
Any morphism in $R_-$ is a composition of an isomorphism and of elementary degeneracies. It suffices to show that elementary degeneracies are split epimorphisms. We have assumed that these have a section if their source has arity $1$. Suppose then that $f=(s)'\circ_i q$ is an elementary degeneracy, with operation $s$ of arity greater than $1$. The petal map $g=s\circ_i (q)'$ factors through an input map of $s\circ_i q$. In case of symmetric operad this implies existence of operations $r$ and $p$ such that $g=p\circ_j (q)'$, $r\circ_1 q=id$ and $(p\circ_j q)\circ_j r = p$ for some $j$. The morphism $f_1=(p)'\circ_j q$ has a section $(p\circ_j q)'\circ_j r$. Finally, notice that $f=f_1\circ h$, where $h:s\to p$ is the obvious permutation isomorphism, and thus $f$ is a split epimorphism. Planar and cyclic cases, and the enveloping category case, are similar.
\end{proof}

\begin{example}
For the graph-substitution operads containing $\bigcirc$ the morphism $\nu\to\bigcirc$, and its image in $\mathcal{V}(P)$, is an elementary degeneracy without a section. For the rest of graph-substitution operads $P$ the subcategory $R_-$ of $\Tw P$ (or of $\mathcal{V}(P)$) is the subcategory of split epimorphisms.
\end{example}

\begin{proposition}
Under conditions of Theorem~\ref{thm:groupoid-reedy} (or Theorem~\ref{thm:enveloping-image-reedy}) assume that petals factor through inputs in $\Tw P$. All monomorphisms in $\Tw P$ (or in $\mathcal{V}(P)$) are in $R_+$ if and only if elementary degeneracies from operations of arity $1$ (or from their images) are not monomorphisms.
\end{proposition}
\begin{proof}
Let $g\circ l$ be the factorization of a monomorphism via $(R_-, R_+)$. The morphism $l$ is a composition of elementary degeneracies $l_j$ followed by an isomorphism. Since $g\circ l$ is a monomorphism, the first of these elementary degeneracies $l_1$ is a monomorphism. Thus monomorphisms are in $R_+$ if and only if elementary degeneracies are not monomorphisms.

As in the Proposition~\ref{pr:elem-deg-split-epi}, let $f=(s)'\circ_i q$ be an elementary degeneracy from an operation of arity greater than $1$, and $f_1=(p)'\circ_j q$ be the corresponding elementary degeneracy of the special form. It suffices to show that $f_1$ is not a monomorphism. Take two non-equal morphisms from an identity operation to $p$: the $(j+1)$-th input map of $p$ and the morphism with the upper vertex  marked by $r$ and the lower vertex marked by $(p\circ_j q)$ and connected to the source vertex by the $j$-th edge. Compositions of these morphisms with $f_1$ are equal, thus $f_1$ is not a monomorphism, and elementary degeneracies from operations of arity greater than $1$ are not monomorphisms.
\end{proof}

\begin{example}
In graph-substitution operads an operation of arity $1$ that can be a source of elementary degeneracy is isomorphic either to $id_1$, or to $\nu$. The elementary degeneracy $id_1\to\mu_0$ is not a monomorphism: its compositions with the two elementary degeneracies $\mu_2\to id_1$ coincide. Likewise, the compositions of the elementary degeneracy $\nu\to\bigcirc$ with either of the two elementary degeneracies $\nu\circ_1\mu_2\to\nu$ coincide. Thus for all graph-substitution operads monomorphisms in their twisted arrow categories or in the categories $\mathcal{V}(P)$ are in $R_+$.
\end{example}

\begin{proposition}
Under conditions of Theorem~\ref{thm:enveloping-image-reedy}, the operadic composition of operations of non-zero arity with operations of non-zero arity is injective (i.e.\@ for any operation $p$ in $P$ of non-zero arity the maps $p\circ_i -$ and $-\circ_i p$, restricted to the set of operations of non-zero arity, are injective for all $i$) if and only if any morphism in $\mathcal{V}(P)$ with all non-source vertices of non-zero arity is a monomorphism.
\end{proposition}
\begin{proof}
Let $f$ be a morphism in $\mathcal{V}(P)$ with all non-source vertices having non-zero arity (such a morphism $f$ is always in $R_+$). Injectivity of partial composition ensures that non-source vertices and indices of leaves of a morphism $g$ in $\mathcal{V}(P)$ can be recovered from the non-source vertices and indices of leaves of $f$ and $f\circ g$. The opposite direction is trivial.
\end{proof}

\begin{example}
The operads $mOp$, $mOp_{(g,n)}$, $mOp_{st}$, $mOp_{nc}$ and $ciuAs$ do not satisfy the injectivity condition, since the action of the groupoid $P(\psi=0)$ is not free. The operads containing $\bigcirc$ do not satisfy the injectivity condition, since $\nu\circ_1\mu_2=\nu\circ_1(\mu_2^{(12)})$. 

The graph-substitution operads based on trees satisfy the injectivity condition, and for these operads the subcategory $R_+$ of the category $\mathcal{V}(P)$ is the subcategory of monomorphisms.
\end{example}

For twisted arrow categories the ``only if'' part of the proposition above does not hold: any morphism in the twisted arrow category of the planar operad with three binary operations $p_1,p_2\in P(c_1,c_2;c_0)$ and $q\in P(c_3,c_4;c_1)$ and one ternary operation $p_1\circ_1 q=p_2\circ_1 q$ is a monomorphism.

\begin{proposition}
Under conditions of Theorem~\ref{thm:groupoid-reedy}, suppose additionally that for any operation $p$ in $P$ of non-zero arity the maps $-\circ_i p$, restricted to the set of operations of non-zero arity, are injective for all $i$. Then any morphism $f:t\to r$ with all vertices (including the source) having non-zero arity is a monomorphism.
\end{proposition}
\begin{proof}
Injectivity ensures that any morphism $g:s\to t$ in $\Tw P$ is determined uniquely by $f\circ g$: the upper vertices of $g$ and the indices of their leaves are determined by $f\circ g$, and  the lowest vertex of $g$ is determined by the upper vertices of $g$ and by the operations $s$ and $t$.
\end{proof}

\begin{example}
For all graph-substitution operads compositions $-\circ_i p$ of operations of non-zero arity with operations of non-zero arity are injective: a graph $r$ can be recovered by contracting the subgraph $p$ in the graph $r\circ_i p$ and permuting the edges of the obtained vertex according to the permutation on leaves of $p$. To check that the subcategories $R_+$ of twisted arrow categories of graph-substitution operads consist of monomorphisms it suffices to check that morphisms in $R_+$ with at least one vertex of arity $0$ are monomorphisms.

For operads $ciuAs$, $mOp$, $mOp_{(g,n)}$ the morphism $\mu_0\to\bigcirc$ is not a monomorphism: the compositions of this morphism with the two maps $\mu_0\to\mu_0$ are equal. For the remaining graph-substitution operads the subcategories $R_+$ of their twisted arrow categories consist of monomorphisms.
\end{example}

Often twisted arrow categories and enveloping categories of operads are such that elementary degeneracies with a common source have a split pushout. Recall that a split pushout is a diagram of the following form.
\begin{center}
\begin{tikzcd}
A \arrow[r, "f"'] \arrow[d, "g"]           & B \arrow[d, "k"'] \arrow[l, "s"', bend right] \\
C \arrow[r, "h"] \arrow[u, "t", bend left] & P \arrow[u, "u"', bend right]                
\end{tikzcd}
\end{center}
Here the morphisms $t$, $s$ and $u$ are sections, and $f\circ t=u\circ h$ and $k\circ f=h\circ g$. Split pushouts are absolute pushouts. Split pushout is asymmetric notion: there are two ways in which morphisms with common source can have a split pushout.

\begin{proposition}
Under conditions of Theorem~\ref{thm:enveloping-image-reedy}, suppose that any morphism in $R_-$ is a split epimorphism. Then any pair of elementary degeneracies with a common source has a  pushout if and only if for any colour $c$ in $P$ there is at most one operation of grading $(-1)$ with the output colour $c$. Such pushouts are split pushouts, in both possible ways.
\end{proposition}
\begin{proof}
If there are two operations of grading $(-1)$ with the same output colour, then the images of the corresponding output maps do not have a pushout. Suppose then that for any colour $c$ in $P$ there is at most one operation of grading $(-1)$ with the output colour $c$. Let $f$ and $g$ be two different elementary degeneracies with common source $A$. If $g$ and $f$ graft an operation of grading $(-1)$ to the $i$-th and to the $j$-th leaves of $A$ respectively, with $i<j$, then $k$ and $h$ are the elementary degeneracies that graft the operations of grading $(-1)$ to the leaves now indexed by $i$ and $(j-1)$ respectively.

Let $u$ be a section of $k$. It has one non-trivial non-source vertex. If this vertex is a lower vertex, the index of its leaf is $i$. If this vertex is an upper vertex grafted in the $l$-th leaf, then the indices of its leaves are $l$ and $i$. Let $t$ be a morphism with only one non-trivial non-source vertex, equal to the non-trivial non-source vertex of $u$, grafted to the essentially the same leaf, and with leaves indexed by $i$ and, if the vertex is upper, by the index of the leaf to which this vertex is grafted. The morphism $t$ is a section of $g$, and $f\circ t=u\circ h$. The case with $i>j$ is similar.
\end{proof}

For twisted arrow categories the condition above is necessary, but not sufficient, since in this case the target of the morphism $t$ might be different from $A$. 

\begin{proposition}
For graph-substitution operads not containing $\bigcirc$ elementary degeneracies with common source have a split pushout, in both possible ways.
\end{proposition} 
\begin{proof}
Let $f$ and $g$ be two elementary degeneracies with common source $A$. Morphisms $f$ and $g$ replace vertices $v$ and $w$ of degree two in the graph $A$ by edges. If the vertex $w$ is adjacent to a vertex $z$ different from $v$, then the sections $t$ and $u$ in the split pushout can be chosen to be upper morphisms corresponding to insertions of a tree on two vertices (corresponding to $z$ and $w$) into the vertex $z$. If the vertex $w$ is adjacent only to $v$, then the sections $t$ and $u$ can be chosen to be lower morphisms, corresponding to insertion of the sources into the vertex of the tree with two vertices, one of which is the vertex $w$.
\end{proof}

\begin{proposition}
Under conditions of Theorem~\ref{thm:groupoid-reedy} or Theorem~\ref{thm:enveloping-image-reedy}, suppose that any morphism in $R_-$ is a split epimorphism and any pair of elementary degeneracies with a common source has a split pushout, in both possible ways. Then (1) any pair of split epimorphisms with a common source has an absolute pushout, and (2) any two split epimorphisms that have the same set of sections are equal.
\end{proposition}
\begin{proof}
The first part is simple: any split epimorphism is a composition of elementary degeneracies followed by an isomorphism, and the morphisms in the pushout square of elementary degeneracies are elementary degeneracies. The pushout square of two split epimorphisms can be built from pushout squares of elementary degeneracies.

For the second part observe that by vertically stacking the pushout diagrams of elementary degeneracies we can show that any split epimorphism $g$ and elementary degeneracy $f$ have a split pushout (as in the above diagram, and $g$ and $f$ cannot be exchanged here). If two split epimorphisms $g, g_1:A\to C$ are different, then there is an elementary degeneracy $f$ such that $g_1=l\circ f$ for some morphism $l$, and $g\neq l'\circ f$ for any morphism $l'$. Let $t$ be the section of $g$ in the split pushout of $f$ and $g$. If $t$ is a section of $g_1$, then $id=g_1\circ t=l\circ f\circ t=l\circ u\circ h$. The morphism $h$ is an elementary degeneracy, and thus does not have a left inverse. The section $t$ of $g$ is not a section of $g_1$. 
\end{proof}

We have shown that for operads $uAs$, $iuAs$, $pOp$, $sOp$, and $cOp$ the corresponding twisted arrow categories and enveloping categories are both Eilenberg--Zilber and EZ-categories. The twisted arrow categories of the operads $mOp_{nc}$ and $mOp_{st}$ are also Eilenberg--Zilber and EZ-categories, while the corresponding enveloping categories are Eilenberg--Zilber, but not EZ-categories.

\section{Twisted arrow quasi-categories}
\label{sec:quasi-cat}
Twisted arrow $\infty$-categories of simplicial operads were introduced in \cite{hoang2020quillen}. We propose an analogous construction for $\infty$-operads modeled by dendroidal sets. Unfortunately, the construction is ad hoc. We do not know if there is a good general context analogous to the discrete case. 

Denote by $\Omega'$ the full subcategory of $\Tw(sOp)$ on operadic trees with vertices ordered in depth-first search order. This category is equivalent to the category $\Omega$ and allows to ignore the order on vertices.

\begin{proposition}
There is a faithful functor $T:\Delta/\Gamma\to \Omega'$.
\end{proposition}
\begin{proof}
A functor $[0]\to\Gamma$, an object of $\Delta/\Gamma$, corresponds to an object $[n]$ of $\Gamma$. The functor $T$ sends this object to the tree with one vertex and with $n$ leaves ordered trivially. An object $[1]\to \Gamma$ that corresponds to a morphism $f$ in $\Gamma$ is sent by $T$ to the reduced symmetric tree such that leaves of the corolla adjacent to the $j$-th edge of the source vertex are indexed by $f^{-1}(j)$. An object $G:[n]\to\Gamma$ is sent to the tree of total height $(1+2n)$ obtained via sequence of graftings of corollas, with $k$-th grafting corresponding to $k$-th 1-simplex $G\circ([1]\to [n])$. The $i$-th face map $G\circ ([n-1]\to [n])$ is sent by $T$ to the composition of face maps in $\Omega'$ that contract inner edges of distance $i$ from the source vertex, or graft corollas into leaves in case $i=n$. The $i$-th degeneracy map $G\circ([n]\to[n-1])$ is sent by $T$ to the composition of degeneracy maps that replace inner edges of distance $i$ from the source vertex by edges with one vertex of degree $2$.
\end{proof}

\begin{definition}
Let $\pi:\Delta/\Gamma\to\Delta$ be the projection. The functor $\Tw:dSet\to sSet$, the twisted arrow set of a dendroidal set, is the composition of restriction along $T$ with the left Kan extension along $\pi$.
\end{definition}

\begin{proposition}
For any operad $P$ the twisted arrow set of the dendroidal nerve of $P$ is the nerve of the twisted arrow category of $P$.
\end{proposition}

\begin{remark}
It is known\footnote{https://mathoverflow.net/questions/295318} that the nerve of the twisted arrow category $\Tw(C)$ of a~category~$C$ is the diagonal of the bisimplicial set obtained via join maps $([n],[m])\to [n+m+1]$ from the nerve of the category $C$, or, equivalently, obtained from the (unpointed) stable double category corresponding to the category $C$ (\cite{bergner20182}). Pointed stable double categories are equivalent to 2-Segal sets, or to discrete decomposition spaces. The construction of stable double category corresponding to a category relies on the fact that both $(Upper, Lower)$ and $(Lower,  Upper)$ are strict factorization systems for twisted arrow categories of categories. The twisted arrow set of a dendroidal set is also the diagonal of bisimplicial set, constructed using the $(Upper, Lower)$ factorization system and canonical $(Lower',Upper)$-factorizations.
\end{remark}

\begin{remark}
The functor $T$ is most likely analogous to the dendrification functor of \cite{heuts2016equivalence}. This suggests that there should be three model category structures on presheaves over $\Delta/\Gamma$ that model $\infty$-operads: one is used in Lurie's works, one should come from twisted arrow categories, and one should come from universal enveloping categories of operads.

Furthermore, there is a functor $\Omega'/\TwOp(uCom)\to \Tw(sOp)$ that generalizes the functor $T$. Composition of the restriction along this functor with the left Kan extension along the projection $\Omega'/\TwOp(uCom)\to\Omega'$ gives the functor $sd:dSets\to dSets$, which is likely the generalization of the edgewise subdivision of simplicial sets to dendroidal sets. The edgewise subdivision of the nerve of an operad is the nerve of its twisted arrow operad.
\end{remark}

\bibliographystyle{amsalpha}
\bibliography{main}

\end{document}